\documentclass[12pt]{amsart}
\usepackage{amscd,amssymb,latexsym}
\usepackage{tikz}
\usepackage{stmaryrd}
\usepackage{fullpage}
\usepackage{hyperref}
\usepackage{resizegather}
\usepackage{caption}
\usepackage{color}
\usetikzlibrary{shapes.geometric}
\input{xypic}
\newcommand{\Mdef}[2]{\newcommand{#1}{\relax \ifmmode #2 \else $#2$\fi}}

\DeclareMathOperator{\Ctwo}{C_2}

\newcommand{\sm}{\wedge}

\newcommand{\tensor}{\otimes}

\newcommand{\rost}{\bigstar}
\newcommand{\rhost}{*\rho -}

\newcommand{\Hom}{\mathrm{Hom}}

\newcommand{\Ext}{\mathrm{Ext}}
\Mdef{\bhom}{\mathbf{\hat{H}om}}

\Mdef{\Mod}{\mathrm{mod}}

\newcommand{\st}{\; | \;}

\numberwithin{equation}{section}
\newtheorem{thm}[equation]{Theorem}
\newtheorem{lemma}[equation]{Lemma}
\newtheorem{prop}[equation]{Proposition}
\newtheorem{cor}[equation]{Corollary}

\theoremstyle{definition}

\newtheorem{defn}[equation]{Definition}

\newtheorem{example}[equation]{Example}
\newtheorem{examples}[equation]{Examples}
\newtheorem{remark}[equation]{Remark}

\newcommand{\qqed}{\qed \\[1ex]}
\renewenvironment{proof}[1][\hspace*{-.8ex}]{\noindent {\bf Proof #1:\;}}{\qqed}

\Mdef{\PH} {\Phi^H}
\Mdef{\PK} {\Phi^K}
\Mdef{\PL} {\Phi^L}
\Mdef{\PT} {\Phi^{\T}}

\Mdef{\ef}{E{\cF}_+}
\Mdef{\etf}{\widetilde{E}{\cF}}
\Mdef{\eg}{E{G}_+}
\Mdef{\etg}{\tilde{E}{G}}

\newcommand{\Mpi}{\underline{\pi}^{C_2}}

\Mdef{\infl}{\mathrm{inf}}
\Mdef{\defl}{\mathrm{def}}
\Mdef{\res}{\mathrm{res}}
\Mdef{\ind}{\mathrm{ind}}
\Mdef{\coind}{\mathrm{coind}}

\Mdef{\univ}{\mathcal{U}}

\newcommand{\CP}{\mathbb{C} P}

\Mdef{\Fp}{\mathbb{F}_p}
\Mdef{\Zpinfty}{\Z /p^{\infty}}
\Mdef{\Zpadic}{\Z_p^{\wedge}}

\newcommand{\bi}{\begin{itemize}}
\newcommand{\be}{\begin{enumerate}}
\newcommand{\bc}{\begin{center}}
\newcommand{\bd}{\begin{description}}
\newcommand{\ei}{\end{itemize}}
\newcommand{\ee}{\end{enumerate}}
\newcommand{\ec}{\end{center}}
\newcommand{\ed}{\end{description}}

\newcommand{\lra}{\longrightarrow}

\Mdef{\we}{\mathbf{we}}
\Mdef{\fib}{\mathbf{fib}}
\Mdef{\cof}{\mathbf{cof}}
\Mdef{\BI}{\mathcal{BI}}

\newcommand{\holim}{\mathop{ \mathop{\mathrm {holim}} \limits_\leftarrow} \nolimits}
\newcommand{\hocolim}{\mathop{  \mathop{\mathrm {holim}}\limits_\rightarrow} \nolimits}
\DeclareMathOperator*{\hcolim}{hocolim}
\DeclareMathOperator*{\clim}{colim}
\DeclareMathOperator*{\hlim}{holim}

\Mdef{\A}{\mathbb{A}}
\Mdef{\B}{\mathbb{B}}
\Mdef{\C}{\mathbb{C}}
\Mdef{\D}{\mathbb{D}}
\Mdef{\E}{\mathbb{E}}
\Mdef{\T}{\mathbb{T}}
\Mdef{\F}{\mathbb{F}}
\Mdef{\G}{\mathbb{G}}
\Mdef{\I}{\mathbb{I}}
\Mdef{\N}{\mathbb{N}}
\Mdef{\Q}{\mathbb{Q}}
\Mdef{\R}{\mathbb{R}}
\Mdef{\bbS}{\mathbb{S}}
\Mdef{\Z}{\mathbb{Z}}

\Mdef{\bA}{\mathbb{A}}
\Mdef{\bB}{\mathbb{B}}
\Mdef{\bC}{\mathbb{C}}
\Mdef{\bD}{\mathbb{D}}
\Mdef{\bE}{\mathbb{E}}
\Mdef{\bF}{\mathbb{F}}
\Mdef{\bG}{\mathbb{G}}
\Mdef{\bH}{\mathbb{H}}
\Mdef{\bI}{\mathbb{I}}
\Mdef{\bJ}{\mathbb{J}}
\Mdef{\bK}{\mathbb{K}}
\Mdef{\bL}{\mathbb{L}}
\Mdef{\bM}{\mathbb{M}}
\Mdef{\bN}{\mathbb{N}}
\Mdef{\bO}{\mathbb{O}}
\Mdef{\bP}{\mathbb{P}}
\Mdef{\bQ}{\mathbb{Q}}
\Mdef{\bR}{\mathbb{R}}
\Mdef{\bS}{\mathbb{S}}
\Mdef{\bT}{\mathbb{T}}
\Mdef{\bU}{\mathbb{U}}
\Mdef{\bV}{\mathbb{V}}
\Mdef{\bW}{\mathbb{W}}
\Mdef{\bX}{\mathbb{X}}
\Mdef{\bY}{\mathbb{Y}}
\Mdef{\bZ}{\mathbb{Z}}

\Mdef{\cA}{\mathcal{A}}
\Mdef{\cB}{\mathcal{B}}
\Mdef{\cC}{\mathcal{C}}
\Mdef{\mcD}{\mathcal{D}} 
\Mdef{\cE}{\mathcal{E}}
\Mdef{\cF}{\mathcal{F}}
\Mdef{\cG}{\mathcal{G}}
\Mdef{\mcH}{\mathcal{H}} 
\Mdef{\cI}{\mathcal{I}}
\Mdef{\cJ}{\mathcal{J}}
\Mdef{\cK}{\mathcal{K}}
\Mdef{\mcL}{\mathcal{L}}

\Mdef{\cM}{\mathcal{M}}
\Mdef{\cN}{\mathcal{N}}
\Mdef{\cO}{\mathcal{O}}
\Mdef{\cP}{\mathcal{P}}
\Mdef{\cQ}{\mathcal{Q}}
\Mdef{\mcR}{\mathcal{R}}
\Mdef{\cS}{\mathcal{S}}
\Mdef{\cT}{\mathcal{T}}
\Mdef{\cU}{\mathcal{U}}
\Mdef{\cV}{\mathcal{V}}
\Mdef{\cW}{\mathcal{W}}
\Mdef{\cX}{\mathcal{X}}
\Mdef{\cY}{\mathcal{Y}}
\Mdef{\cZ}{\mathcal{Z}}

\Mdef{\ca}{\mathcal{a}}
\Mdef{\ct}{\mathcal{t}}

\Mdef{\At}{\tilde{A}}
\Mdef{\Bt}{\tilde{B}}
\Mdef{\Ct}{\tilde{C}}
\Mdef{\Et}{\tilde{E}}
\Mdef{\Ht}{\tilde{H}}
\Mdef{\Kt}{\tilde{K}}
\Mdef{\Lt}{\tilde{L}}
\Mdef{\Mt}{\tilde{M}}
\Mdef{\Nt}{\tilde{N}}
\Mdef{\Pt}{\tilde{P}}

\Mdef{\tA}{\tilde{A}}
\Mdef{\tB}{\tilde{B}}
\Mdef{\tC}{\tilde{C}}
\Mdef{\tE}{\tilde{E}}
\Mdef{\tH}{\tilde{H}}
\Mdef{\tK}{\tilde{K}}
\Mdef{\tL}{\tilde{L}}
\Mdef{\tM}{\tilde{M}}
\Mdef{\tN}{\tilde{N}}
\Mdef{\tP}{\tilde{P}}

\Mdef{\ft}{\tilde{f}}
\Mdef{\xt}{\tilde{x}}
\Mdef{\yt}{\tilde{y}}

\Mdef{\Ab}{\overline{A}}
\Mdef{\Bb}{\overline{B}}
\Mdef{\Cb}{\overline{C}}
\Mdef{\Db}{\overline{D}}
\Mdef{\Eb}{\overline{E}}
\Mdef{\Fb}{\overline{F}}
\Mdef{\Gb}{\overline{G}}
\Mdef{\Hb}{\overline{H}}
\Mdef{\Ib}{\overline{I}}
\Mdef{\Jb}{\overline{J}}
\Mdef{\Kb}{\overline{K}}
\Mdef{\Lb}{\overline{L}}
\Mdef{\Mb}{\overline{M}}
\Mdef{\Nb}{\overline{N}}
\Mdef{\Ob}{\overline{O}}
\Mdef{\Qb}{\overline{Q}}
\Mdef{\Rb}{\overline{R}}
\Mdef{\Sb}{\overline{S}}
\Mdef{\Tb}{\overline{T}}
\Mdef{\Ub}{\overline{U}}
\Mdef{\Vb}{\overline{V}}
\Mdef{\Wb}{\overline{W}}
\Mdef{\Xb}{\overline{X}}
\Mdef{\Yb}{\overline{Y}}
\Mdef{\Zb}{\overline{Z}}

\Mdef{\db}{\overline{d}}
\Mdef{\hb}{\overline{h}}
\Mdef{\qb}{\overline{q}}
\Mdef{\rb}{\overline{r}}
\Mdef{\tb}{\overline{t}}
\Mdef{\ub}{\overline{u}}
\Mdef{\vb}{\overline{v}}

\Mdef{\hc}{\hat{c}}
\Mdef{\he}{\hat{e}}
\Mdef{\hf}{\hat{f}}
\Mdef{\hA}{\hat{A}}
\Mdef{\hH}{\hat{H}}
\Mdef{\hJ}{\hat{J}}
\Mdef{\hM}{\hat{M}}
\Mdef{\hP}{\hat{P}}
\Mdef{\hQ}{\hat{Q}}

\Mdef{\thetab}{\overline{\theta}}
\Mdef{\phib}{\overline{\phi}}

\Mdef{\uA}{\underline{A}}
\Mdef{\uB}{\underline{B}}
\Mdef{\uC}{\underline{C}}
\Mdef{\uD}{\underline{D}}
\Mdef{\uvb}{\underline{\vb}}
\Mdef{\ul}{\underline{l}}

\Mdef{\bolda}{\mathbf{a}}
\Mdef{\boldb}{\mathbf{b}}
\Mdef{\bfD}{\mathbf{D}}

\Mdef{\fm}{\frak{m}}
\Mdef{\fp}{\frak{p}}

\Mdef{\eps}{\epsilon}

\newcommand{\Zu}{\underline{\Z}}
\newcommand{\HZu}{H\Zu}

\newcommand{\kR}{k\R}
\newcommand{\MUR}{M\R}
\newcommand{\pp}{\delta} 
\newcommand{\Ftwo}{\mathbb{F}_2}
\newcommand{\siftys}{S^{\infty \sigma}}
\newcommand{\eqp}{(EC_2)_+}
\newcommand{\BPn}{BP\langle n \rangle}
\newcommand{\BPRn}{BP\R \langle n \rangle}

\newcommand{\BPR}{BP\R}

\newcommand{\vbn}[1]{\overline{v}_{#1}}

\newcommand{\cEtn}{\widetilde{\cE}_n}
\newcommand{\vob}{\vbn{1}}
\newcommand{\vtb}{\vbn{2}}
\newcommand{\tmfot}{tmf_1(3)}
\newcommand{\cEt}{\widetilde{\mathcal{E}}}
\newcommand{\xu}{\mathbf{x}}

\newcommand{\Pb}[1]{\overline{P}_{#1}}
\newcommand{\Jbn}{\overline{J}_n}
\newcommand{\alphat}{\tilde{\alpha}}
\newcommand{\cell}{\mathrm{Cell}}
\usepackage{color}
\definecolor{todo}{rgb}{1,0,0}

\begin{document}
\title{Gorenstein duality for Real spectra}

\author{J.P.C.Greenlees}
\address{School of Mathematics and Statistics, Hicks Building, 
Sheffield S3 7RH. UK.}
\email{j.greenlees@sheffield.ac.uk}

\author{L.Meier}
\address{Universit\"at Bonn, Endenicher Allee 60, 53115 Bonn}
\email{lmeier@math.uni-bonn.de}
\date{}

\begin{abstract}
Following Hu and Kriz, we study the $C_2$-spectra $\BPRn$ and $E\R(n)$
that refine the usual truncated Brown--Peterson and the
Johnson--Wilson spectra. In particular, we show that they satisfy
Gorenstein duality with a representation grading shift and identify
their Anderson duals. We also compute the associated local cohomology spectral sequence in the cases $n=1$ and $2$.
\end{abstract}

\thanks{We are grateful to the Hausdorff Institute of Mathematics in
  Bonn for providing us the opportunity in Summer 2015 for the
  discussions starting this work.
We also  thank Vitaly Lorman and Nicolas Ricka for helpful discussions. }
\maketitle

\tableofcontents
\newpage

\section{Introduction}
\subsection{Background}
\subsubsection*{Philosophy}
For us, \emph{Real spectrum} is a loose term for a $C_2$-spectrum built
from the $C_2$-spectrum $M\R$ of Real bordism, considered by Araki,
Landweber and Hu--Kriz \cite{HK}. The present article shows that
bringing together Real spectra and Gorenstein duality reveals rich and interesting structures. 

It is part of our philosophy that theorems about Real spectra can
often be shown in the same style as theorems for the underlying
complex oriented spectra although the details might be more
difficult, and groups needed to be graded over the real representation
ring $RO(C_2)$ (indicated by $\rost$) rather than over the integers
(indicated by  $*$).   This extends a well known
phenomenon: complex orientability of equivariant spectra
makes it easy to reduce questions to integer gradings, and we show 
that even  in the absence of complex orientability, good
behaviour of coefficients can be seen by grading with
representations.  

\subsubsection*{Bordism with reality}
In studying these spectra, the real regular representation $\rho=\R
C_2$ plays a special role. We write $\sigma$ for the sign
representation on $\R $ so that $\rho =1+\sigma$. 
 One of the crucial features of $M\R$ is that 
 it is \emph{strongly even} in the sense of \cite{HM}, i.e.
\begin{enumerate}
\item \label{item:Restr} the restriction functor $\pi_{k\rho}^{C_2}M\R \to \pi_{2k}MU$ is an isomorphism for all $k\in\Z$, and
\item the groups $\pi_{k\rho-1}^{C_2}M\R$ are zero for all $k\in\Z$.
\end{enumerate}

We now localize at 2, and (with two  exceptions) all spectra and
abelian groups will henceforth  be $2$-local. The Quillen idempotent
has a $C_2$-equivariant refinement, and this defines the
$C_2$-spectrum $\BPR$ as a summand of $M\R_{(2)}$. The isomorphism (\ref{item:Restr})  allows us to lift the usual $v_i$ to classes $\vb_i\in\pi_{(2^i-1)\rho}^{C_2}\BPR$. The Real spectra we are interested in are quotients of $\BPR$ by sequences of $\vb_i$ and localizations thereof. For example, we can follow \cite{HK} and \cite{Hu} and define 
$$\BPRn = \BPR/(\vb_{n+1},\vb_{n+2},\dots)$$ 
and 
$$E\R(n) = \BPRn[\vb_n^{-1}].$$
These spectra are still strongly even, as we will show. Apart from the big literature on $K$-theory with Reality (e.g.\ \cite{Ati66}, \cite{Dugger} and \cite{B-G10}), Real spectra have been studied by Hu and Kriz, in a series of papers by Kitchloo and Wilson (see e.g.\ \cite{K-W15} for one of the latest installments), by Banerjee \cite{Banerjee}, by Ricka \cite{Ricka} and by Lorman \cite{Lorman}. 
 A crucial point is that a morphism between strongly even
 $C_2$-spectra is an equivalence if it is an equivalence of underlying
 spectra \cite[Lemma 3.4]{HM}. 

We are interested in two dualities for Real spectra: Anderson duality
and Gorenstein duality. These are closely related \cite{GS} but apply
to different classes of spectra. 

\subsubsection*{Anderson duality}
The Anderson dual $\Z^X$ of a spectrum $X$ is an integral version of
its Brown-Comenetz dual (in accordance with our general principle,
$\Z$ denotes the $2$-local integers). The homotopy groups of the
Anderson dual  lie in a short exact sequence
\begin{align}\label{eq:IntroductionSES} 0 \to \Ext_{\Z}^1(\pi_{-*-1}X ,\Z) \to \pi_*(\Z^X) \to
\Hom_{\Z}(\pi_{-*}X, \Z) \to 0. \end{align}

One reason to be interested in the computation of Anderson duals is
that they show up in universal coefficient sequences (see
\cite{Anderson} or Section \ref{subsec:Anderson}). The situation is
nicest for spectra that are Anderson self-dual in the sense that
$\Z^X$ 
is a suspension of $X$. Many important spectra like $KU$, $KO$, $Tmf$ \cite{Sto12} or
$Tmf_1(3)$ are indeed Anderson self-dual. These examples are all unbounded as the sequence \eqref{eq:IntroductionSES} nearly forces them to be.  

Anderson duality also works $C_2$-equivariantly
as first explored by \cite{Ricka}; the only change in the above short
exact sequence is that equivariant homotopy groups are used. The $C_2$-spectra $K\R$
\cite{H-S14} and $Tmf_1(3)$ \cite{HM} are also $C_2$-equivariantly
self-Anderson dual, at least if we allow suspensions by
\emph{representation spheres}. 

One simpler example is essential background: if $\Zu$ denotes the constant Mackey functor (i.e., with restriction being the identity and
induction being multiplication by 2) then the Anderson dual of its
Eilenberg-MacLane spectrum is the Eilenberg-MacLane spectrum for the
dual Mackey functor  $\underline{\Z}^*=\Hom_{\Z}(\underline{\Z}, \Z)$
(i.e., with restriction being multiplication by 2 and induction being
the identity).  It is then easy to check that in fact $H(\Zu^*)\simeq \Sigma^{2(1-\sigma)}\HZu$.
(From one point of view this is the fact that $\R P^1=S(2\sigma)/\Ctwo$ is equivalent to the circle). 
The dualities we find are in a sense all 
dependent on this one. 

\subsubsection*{Gorenstein duality}
By contrast with Anderson self-duality,  many connective ring spectra
are Gorenstein in the sense of \cite{DGI}. We sketch the theory here,
 explaining it more  fully in Sections \ref{sec:kR} and \ref{sec:dishonest}.

The starting point is a connective commutative ring $C_2$-spectrum $R$,
whose $0$th homotopy Mackey functor is constant at $\Z$: 
$$\underline{\pi}^{\Ctwo}_0(R)\cong \Zu.$$ 
This gives us a map $R\lra \HZu$ of commutative ring spectra by killing
homotopy groups.  We say that $R$
is {\em Gorenstein} of shift $a\in RO(C_2)$ if there is an equivalence of $R$-modules
$$\Hom_R(\HZu ,R)\simeq \Sigma^a\HZu. $$

We are interested in the duality this often entails. 
Note that the  Anderson dual $\Z^R$ obviously  has the Matlis lifting property
$$\Hom_R(\HZu, \Z^R)\simeq \HZu^*,  $$
where $\Z^*=\Hom_{\Z}(\Zu, \Z)$ as above. Thus if $R$ is Gorenstein,
in view of the 
equivalence  $H(\Zu^*)\simeq \Sigma^{2(1-\sigma)}\HZu$,  we have equivalences
\begin{align*}\Hom_R(\HZu ,\cell_{\HZu} R)&\simeq \Hom_R(\HZu ,R)\\
&\simeq \Sigma^a\HZu\\
&\simeq \Sigma^{a-2(1-\sigma)}H(\Zu^*) \\
&\simeq
\Hom_R(\HZu, \Sigma^{a-2(1-\sigma)}\Z^R).\end{align*}
Here, $\cell_{\HZu}$ denotes the $\HZu$-$\mathbb{R}$-cellularization as in Section \ref{sec:Cell}. We would like to remove the $\Hom_R(\HZu, \cdot)$ from the composite equivalence above. 

\begin{defn}
We say that $R$ has {\em Gorenstein duality} of shift $b$ if we have
an equivalence of $R$-modules
$$\cell_{\HZu} R \simeq \Sigma^b \Z^R.$$
\end{defn}

 As in the non-equivariant setting, the passage from Gorenstein to
 Gorenstein duality requires showing that
the above composite equivalence is compatible with the right  
action of $\cE =\Hom_R(\HZu, \HZu)$. This turns out to be considerably
more delicate than the non-equivariant counterpart because
connectivity is harder to control; but if  one can lift the
$R$-equivalence to  an $\cE$-equivalence,  the conclusion is that if $R$
is Gorenstein of shift $a$ then it has Gorenstein duality of shift $b=a-2(1-\sigma)$. 


\subsubsection*{Local cohomology}
The duality statement becomes more interesting when the cellularization can be
constructed algebraically.  For any finitely generated ideal $J$ of the $RO(C_2)$-graded
coefficient ring  $R_{\rost}^{C_2}$, we may form the stable Koszul
complex $\Gamma_JR$, which only depends on the radical of $J$. In our
examples, this applies to the augmentation ideal $J=\ker(R_{\rost}^{C_2}\lra
\HZu_{\rost}^{C_2})$, which may be radically generated by finitely
many  elements $\vb_i$ in degrees which are multiples of
$\rho$.  Adapting the usual proof to the Real context, 
Proposition \ref{prop:cell}  shows that 
$\Gamma_JR\lra R $ is
$\HZu$-$\R$-cellularization: 
$$\cell_{\HZu}R\simeq \Gamma_JR. $$
The $RO(C_2)$-graded homotopy groups of $\Gamma_JR$ can be computed using a spectral sequence involving local cohomology. 

\subsubsection*{Conclusion}
 In favourable cases the Gorenstein condition on a ring spectrum $R$
 implies Gorenstein duality for $R$; this in turn establishes a strong
 duality property on the $RO(C_2)$-graded coefficient ring, expressed using local cohomology. 

\subsection{Results}
Our main theorems establish Gorenstein duality for a large family of Real spectra. We present in this introduction the particular cases of $\BPRn$ and $E\R(n)$, deferring the more general theorem to Section \ref{sec:results}. Let again $\sigma$ denote the non-trivial
representation of $\Ctwo$ on the real line and $\rho =1+\sigma$ the 
real regular representation. Furthermore set $D_n = 2^{n+1}-n-2$ so
that $D_n\rho = |\vb_1|+\cdots + |\vb_n|$. Other terms in the statement will be explained
in Section \ref{sec:AKG}. 
 
\begin{thm}
For each $n\geq 1$ the $\Ctwo$-spectrum $\BPRn$ is Gorenstein
of shift $-D_n\rho -n$, and has Gorenstein duality of shift
$-D_n\rho -n-2(1-\sigma)$. This means that
$$\Z_{(2)}^{\BPRn} \simeq \Sigma^{D_n\rho+n+2(1-\sigma)} \Gamma_{\Jb_n}\BPRn,$$
where $\Jb_n = (\vb_1,\dots, \vb_n)$. This induces a local cohomology
spectral sequence
$$H^*_{\Jb_n}(\BPRn^{\Ctwo}_{\rost})\Rightarrow \pi^{\Ctwo}_{\rost}(\Sigma^{-D_n\rho
  -n-2(1-\sigma)}\Z_{(2)}^{\BPRn}). $$
\end{thm}
\vspace{0.3cm}
\begin{thm}\label{thm:ER(n)}
For each $n\geq 1$ the $\Ctwo$-spectrum $E\R(n)$ has Gorenstein duality of shift
$-D_n\rho -(n-1)-2(1-\sigma)$. This means that
\begin{align*}
 \Z_{(2)}^{E\R(n)} &\simeq \Sigma^{D_n\rho+(n-1)+2(1-\sigma)} \Gamma_{\Jb_{n-1}}E\R(n) \\
 &\simeq \Sigma^{(n+2)(2^{2n+1}-2^{n+2})+n+3}\Gamma_{J_{n-1}}E\R(n) ,
\end{align*}
for $J_{n-1} = \Jb_{n-1}\cap\pi_*^{C_2}E\R(n)$. This induces likewise a local cohomology spectral sequence.
\end{thm}

We note that this has implications for the  $\Ctwo$-fixed point spectrum
$(\BPRn)^{\Ctwo}=BPR\langle n\rangle$. The graded ring
$$\pi_*(BPR\langle n\rangle)=\pi_*^{\Ctwo}(\BPRn)$$
is the integer part of the $RO(\Ctwo)$-graded coefficient ring
$\pi^{\Ctwo}_{\rost}(\BPRn)$. However, since the ideal
$\Jb_n$ is not generated in integer degrees, the statement for $BPR\langle n\rangle$
is usually rather complicated, and one of our main messages is that
working with the equivariant spectra gives more insight. On the other hand, $ER(n) = E\R(n)^{C_2}$ has integral Gorenstein duality because one can use the additional periodicity to move the representation suspension and the ideal $\Jb_n$ to integral degrees. 

We will discuss the general result in more detail later, but the two
first cases are about familiar ring spectra. 

\begin{example} (See Sections \ref{sec:kR} and \ref{sec:kRlcss}.) 
For $n=1$, connective $K$-theory with Reality $\kR$ is $2$-locally a
form of $BP\R\langle 1\rangle$.
 For this example we can work without
2-localization, so that $\Z$ means the integers. Our first theorem states
that $\kR$ is Gorenstein of shift $-\rho-1=-2-\sigma$
and that it has Gorenstein duality of shift $-4+\sigma$.  This just means that 
$$\Z^{k\R} \simeq \Sigma^{4-\sigma} \fib(k\R \to K\R).$$
The local cohomology spectral sequence collapses to a short exact sequence associated to the fibre sequence just mentioned. We
will see in Section \ref{sec:kRlcss} that the sequence is not split, even as abelian groups. 

Theorem \ref{thm:ER(n)} recovers the main result of \cite{H-S14},
i.e.\ that $\Z^{K\R} \simeq \Sigma^4K\R$, which also implies $\Z^{KO}
\simeq \Sigma^4 KO$. It is a special feature of the case $n=1$ that we
also get a nice duality statement for the fixed points in the
connective case. Indeed, by considering the $RO(C_2)$-graded homotopy
groups of $k\R$, one sees \cite[3.4.2]{B-G10} that
$$(\kR \tensor S^{-\sigma})^{\Ctwo}\simeq \Sigma^{1}ko. $$
This implies that connective $ko$ has untwisted Gorenstein duality of shift
$-5$, i.e.\ that 
$$\Z^{ko} \simeq \Sigma^5\fib(ko\to KO).$$ 
This admits a closely related non-equivariant proof, combining 
the fact that $ku$ is Gorenstein (clear from coefficients) and the
fact  that complexification $ko\lra ku$ is relatively  Gorenstein
(connective version of Wood's theorem \cite[4.1.2]{B-G10}). 
\end{example}

\begin{example}
(See Examples \ref{exa:forms} and \ref{exa:tmf} or  Lemma \ref{lem:BPRnGordish} and Corollary \ref{cor:BPRnGorDdish}.) 
The 2-localization of the $C_2$-spectrum $tmf_1(3)$ is a form of $BP\R\langle
2\rangle$, and the theorem is closely related to results in
\cite{HM}. It states
that $tmf_1(3)$ is Gorenstein of shift $-4\rho-2=-6-4\sigma $
and has Gorenstein duality of shift $-8-2\sigma$. We show in Section
\ref{sec:tmfotlcss} that there are non-trivial 
differentials in the local cohomology spectral
sequence.

Passing to fixed points we obtain the 2-local equivalence
$$BPR\langle 2\rangle=(BP\R\langle 2\rangle)^{\Ctwo}=tmf_0(3).$$
By contrast with the $n=1$ case, as observed in \cite{HM}, $tmf_0(3)$ does not have untwisted
Gorenstein duality of any integer degree. 

A variant of Theorem \ref{thm:ER(n)} also computes the
$C_2$-equivariant Anderson dual of $TMF_1(3)$ and the computation of
the Anderson dual of $Tmf_1(3)$ from \cite{HM} follows as well.

The results apply to $tmf_1(3)$ and $TMF_1(3)$ themselves (i.e.,
with just 3 inverted, and not all other odd primes). 
\end{example}

We remark that our main theorem also recovers the main result of \cite{Ricka} about the Anderson self-duality of integral Real Morava K-theory.

\subsection{Guide to the reader}
While the basic structure of this paper is easily visible from the table of contents, we want to comment on a few features. 

The precise statements of our main results can be found in Section \ref{sec:results}.  We will give two different proofs of them. One (Part 3) might be called `the hands on approach' which is
elementary and explicit, and one (Part 2) uses Gorenstein techniques
inspired by commutative algebra. 
The intricacy  and dependence 
on specific calculations in the explicit approach and the make the conceptual approach
valuable. The subtlety of the structural requirements of the
conceptual approach make the reassurance of the explicit approach
valuable. The exact results proved in Parts 2 and 3 are also
slightly different. 

While the Gorenstein approach only relies on the knowledge of the homotopy groups of $\HZu$ and the reduction theorem Corollary \ref{cor:reduction}, we need detailed information about the homotopy groups of quotients of $\BPR$ for the hands-on approach. In Appendix \ref{Appendix}, we give a streamlined account of the computation of $\pi_{\rost}^{C_2}\BPR$ (which appeared first in \cite{HK}). In Section \ref{sec:BPRn}, we give a rather self-contained account of the homotopy groups of $\BPRn$ and of other quotients of $\BPR$, which can also be read independently of the rest of the paper. While some of this is rather technical, most of the time we just have to use Corollary \ref{Cor:crucial} whose statement (though not proof, perhaps) is easy to understand.

We give separate arguments for the computation of the Anderson dual of
$k\R$ so that this easier case might illustrate the more complicated
arguments of our more general theorems. Thus, if the reader is only
interested in $k\R$, he or she might ignore most of this paper. More
precisely, under this assumption one might proceed as follows: First one looks at Section \ref{sec:kRgroups} for a quick reminder on $\pi_{\rost}^{C_2}k\R$, then one skims through Sections \ref{sec:Basics} and \ref{sec:AKG} to pick up the relevant definitions and then one proceeds directly to Section \ref{sec:kR} or Section \ref{sec:kRagain} to get the proof of the main result in the case of $k\R$. Afterwards one may look at the pictures and computations in the rest of Section \ref{sec:kRlcss} to see what happens behind the scenes of Gorenstein duality.

\vspace{1cm}
\part{Preliminaries and results}\vspace{0.5cm}
\section{Basics and conventions about $C_2$-spectra}\label{sec:Basics}
\subsection{Basics and conventions}
We will work in the homotopy category of genuine $G$-spectra  (i.e., stable for suspensions by $S^V$ for any finite
dimensional representation $V$) for $G =\Ctwo$, the group of order $2$. We will denote by $\tensor$ the derived smash product of spectra.

We may combine the equivariant and non-equivariant homotopy groups of
a $C_2$-spectrum into a Mackey functor, which we denote by $\underline{\pi}_*^{C_2}X$ and denote $C_2$-equivariant and underlying homotopy groups correspondingly by $\pi^{C_2}_*X$ and $\pi^e_*X$. For an abelian group $A$, we  write $\underline{A}$ for the constant Mackey functor (i.e., restriction maps are the identity),
and $\underline{A}^*$ for its dual (i.e., induction maps are the identity). 
We write $HM$ for the Eilenberg-MacLane spectrum associated to a
Mackey functor $M$. 

Another $C_2$-spectrum of interest to us is $\kR$, the $C_2$-equivariant connective cover of
Atiyah's $K$-theory with Reality \cite{Ati66}. The term ``Real
spectra'' derives from this example. The examples of Real bordism and the other $C_2$-spectra derived from it will be discussed in Section \ref{sec:BPRn}.

We will usually grade our homotopy groups by the real
representation ring $RO(\Ctwo)$, and we write $M_{\rost}$ for $RO(\Ctwo)$-graded groups. In addition to the real sign representation $\sigma$ and the regular representation $\rho$ the virtual
representation $\pp =1-\sigma$ is also significant. Examples of
$RO(\Ctwo)$-graded homotopy classes are the geometric Euler classes
$a_V\colon S^0 \lra S^V$; in particular,  $a=a_{\sigma}$ will play a central role. In addition to $a$, we will also often have a class $u=u_{2\sigma}$ of degree $2\pp$.  

We often want to be able to discuss gradings by certain subsets of
$RO(\Ctwo)$. To start with we often want to refer to gradings by multiples
of the regular representation (where we write $M_{*\rho}$), but we also
need to discuss gradings of the form $k\rho -1$. For this, we use the notation
$$\rhost =\{ k\rho \st k\in \Z\}\cup \{ k\rho -1\st k\in \Z\}. $$
Following \cite{HM} we call an $RO(\Ctwo)$-graded object $M$ {\em even} if
$M_{k\rho -1}=0$ for all $k$. An $RO(\Ctwo)$-graded Mackey functor is {\em
  strongly even} if it is even and all the Mackey functors in gradings
$k\rho$ are constant. We call a $C_2$-spectrum (strongly) even if its homotopy groups are (strongly) even.

If $X$ is a strongly even $C_2$-spectrum and $x\in \pi_{2k}X$, we denote by $\overline{x}$ its counterpart in $\pi_{k\rho}^{C_2}X$. If we want to stress that we consider a certain spectrum as a $C_2$-spectrum, we will also sometimes indicate this by a bar; for example, we may write $\overline{tmf_1(3)}$ if we want to stress the $C_2$-structure of $tmf_1(3)$.

\subsection{Cellularity} \label{sec:Cell}
\label{subsec:Rcell}
In a general triangulated category, it is conventional to say $M$ is \emph{$K$-cellular} if
$M$ is in the localizing subcategory generated by $K$ (or equivalently
by all integer suspensions of $K$). A reference in the case of spectra is \cite[Sec 4.1]{DGI}. We say that a $C_2$-spectrum $M$ is \emph{$K$-$\R$-cellular} (for a $C_2$-spectrum $K$)
if it is in the localizing subcategory generated by the suspensions
$S^{k\rho}\tensor K$ for all integers $k$. We note that this is the
same as the localizing subcategory generated by integer suspensions of
$K$ and $(\Ctwo)_+\tensor K$ because of the cofibre sequence
$$(C_2)_+ \to S^0 \to S^\sigma.$$
We say that a map $N\to M$ is a \emph{$K$-$\R$-cellularization} if $N$ is $K$-$\R$-cellular and the induced map
$$\Hom(K, N) \to \Hom(K, M)$$
is an equivalence of $C_2$-spectra. The $K$-$\R$-cellularization is clearly unique up to equivalence. 

We note that cellularity and $\R$-cellularity are definitely
different. For example  $(\Ctwo)_+$ is not $S^0$-cellular, but it is
$S^0$-$\R$-cellular. 

In this article, we will only use $\R$-cellularity. 

\subsection{The slice filtration}
\label{subsec:slice}
Recall from \cite[Section 4.1]{HHR} or \cite{SlicePrimer} that the \emph{slice cells} are the $\Ctwo$-spectra of the form 
\begin{itemize}
 \item $S^{k\rho}$ of dimension $2k$,
 \item $S^{k\rho-1}$ of dimension $2k-1$, and
 \item $S^k\tensor (\Ctwo)_+$ of dimension $k$.
\end{itemize}
A $\Ctwo$-spectrum $X$ is $\leq k$ if for every slice cell $W$ of dimension $\geq k+1$ the mapping space $\Omega^\infty \Hom_{\mathbb{S}}(W, X)$ is equivariantly contractible. As explained in \cite[Section 4.2]{HHR}, this leads to the definition of $X \to P^kX$, which is the universal map into a $\Ctwo$-spectrum that is $\leq k$. The fibre of $$X \to P^kX$$
is denoted by $P_{k+1}X$. The $k$\emph{-slice} $P_k^kX$ is defined as the fibre of $$P^kX\to P^{k-1}X$$
or, equivalently, as the cofibre of the map $P_{k+1}X \to P_kX$. We have the following two useful propositions:

\begin{prop}[\cite{SlicePrimer}, Cor 2.12, Thm 2.18]
\label{prop:sliceodd}
 If $X$ is an even $\Ctwo$-spectrum, then $P^{2k-1}_{2k-1}X = 0$ for all $k\in \Z$.
\end{prop}

\begin{prop}[\cite{SlicePrimer}, Cor 2.16, Thm 2.18]
\label{prop:sliceven}
 If $X$ is a $\Ctwo$-spectrum such that the restriction map in
 $\underline{\pi}^{\Ctwo}_{k\rho}$ is injective, then $P^{2k}_{2k}X$ is
 equivalent to the Eilenberg-MacLane spectrum
 $\underline{\pi}^{\Ctwo}_{k\rho} X$.
\end{prop}

This allows us to give a characterization of an Eilenberg-MacLane spectrum based on regular representation degrees.
\begin{cor}
\label{cor:characterizingZ}
Any even $\Ctwo$-spectrum $X$ with 
$$\underline{\pi}^{\Ctwo}_{k\rho}(X)=\begin{cases}\underline{A} & \text{ if } k= 0 \\
0 & \text{ else}\end{cases}$$                          
 for an abelian group $A$ is equivalent to $H\underline{A}$.
\end{cor}
\begin{proof}
By the last two propositions, we have
\[P^k_k X\simeq \begin{cases} H\underline{A} & \text{ if } k=0 \\
                0 & \text{ else}
               \end{cases}
\]
By the convergence of the slice spectral sequence \cite[Theorem 4.42]{HHR}, the result follows.
\end{proof}

\section{Anderson duality, Koszul complexes and Gorenstein duality}\label{sec:AKG}

\subsection{Duality for abelian groups}\label{sec:DAb}
It is convenient to establish some conventions for abelian groups to
start with, so as to fix notation. 

Pontrjagin duality is defined for all graded abelian groups $A$
by 
$$A^{\vee}=\Hom_{\Z}(A, \Q /\Z).  $$
Similarly,  the rational dual is defined by 
$$A^{\vee \Q}=\Hom_{\Z}(A, \Q ). $$

Since $\Q $ and $\Q /\Z$ are injective abelian groups these two
dualities are homotopy invariant and  pass to the
derived category. Finally  the Anderson dual $A^*$ is defined by
applying $\Hom_{\Z}(A, \cdot )$ to  the exact sequence
$$0\lra \Z \lra \Q \lra \Q/\Z\lra 0$$ 
so that we have a triangle 
$$A^*\lra A^{\vee \Q }\lra A^{\vee }. $$

If $M$ is a free abelian group, then the
Anderson dual is simply calculated by 
$$M^*=\Hom_{\Z}(M, \Z)$$
(since $M$ is free, the $\Hom$ need not be derived). 

If $M$ is a graded abelian group  which is an $\Ftwo$-vector
space then up to suspension the Anderson dual is  the vector space dual: 
$$M^{\vee}=\Hom_{\Ftwo}(M, \Ftwo)\simeq \Sigma^{-1} M^*$$
(since vector spaces are free,  $\Hom$ need not be derived).

\subsection{Anderson duality}
\label{subsec:Anderson}
Anderson duality is the attempt to topologically realize the algebraic duality from the last subsection. It appears that it was invented by Anderson (only published in mimeographed notes \cite{Anderson}) and Kainen \cite{Kainen}, with similar ideas by Brown and Comenetz \cite{B-C76}. For brevity and consistency, we will only use the term Anderson duality instead of Anderson--Kainen duality or Anderson--Brown--Comenetz duality throughout. We will work in the category of $\Ctwo$-spectra, where Anderson duality was first explored by Ricka in \cite{Ricka}. 

Let $I$ be an injective abelian group. Then we let $I^{\mathbb{S}}$ denote the $C_2$-spectrum representing the functor 
$$X \mapsto \Hom(\pi_{*}^{C_2}X,I).$$
For an arbitrary $C_2$-spectrum, we define $I^X$ as the function spectrum $F(X, I^{\mathbb{S}})$. For a general abelian group $A$, we choose an injective resolution 
$$A \to I \to J$$
and define $A^X$ as the fibre of the map $I^X \to J^X$. For example, we get a fibre sequence
$$\Z^X\lra \Q^X\lra (\Q/\Z)^X.$$
In general, we get a short exact sequence of homotopy groups
$$0 \to \Ext_{\Z}(\pi_{-k-1}^{C_2}(X), A) \to \pi_k^{C_2} (A^X) \to \Hom(\pi_{-k}^{C_2}(X), A) \to 0.$$
The analogous exact sequence is true for $RO(C_2)$-graded Mackey functor valued homotopy groups by replacing $X$ by $(C_2/H)_+ \sm \Sigma^VX$. 
Our most common choices will be $A = \Z$ and $A=\Z_{(2)}$.

From time to time we we use the following property of Anderson duality: If $R$ is a strictly commutative $C_2$-ring spectrum and $M$ an $R$-module, then $\Hom_R(M, A^R) \simeq A^M$ as $R$-modules as can easily be seen by adjunction.

One of the reasons to consider Anderson duality is that it provides universal coefficient sequences. In the $C_2$-equivariant world, this takes the following form \cite[Proposition 3.11]{Ricka}:
\[0 \to \Ext_\Z^1(E_{\alpha-1}^{\Ctwo}(X), A) \to (A^E)_{\alpha}^{\Ctwo}(X) \to \Hom_\Z(E_{\alpha}^{\Ctwo}(X), A) \to 0,\]
where $E$ and $X$ are $C_2$-spectra, $\alpha\in RO(C_2)$ and $A$ is an abelian group. 

Our first computation is the Anderson dual of the Eilenberg--MacLane
spectrum of the constant Mackey functor $\underline{\Z}$.

\begin{lemma}
\label{lem:Zu}
The Anderson dual of the Eilenberg-MacLane spectrum $\HZu$ (as an
$\bbS$-module) is given by 
$$\Z^{\HZu}\simeq\HZu^* \simeq \Sigma^{2\pp} \HZu, $$
where $\delta = 1-\sigma$.
\end{lemma}

\begin{proof}
The first equivalence follows from the isomorphisms
$$\underline{\pi}^{C_2}_*(\Z^{\HZu}) \cong \Hom_{\Z}(\underline{\pi}^{C_2}_{-*}\HZu, \Z) \cong \Zu^*.$$

Since 
$$\pi_*^{\Ctwo}(S^{2-2\sigma}\tensor \HZu)=H^*_{\Ctwo}(S^{2\sigma -2};
\Zu)=H^*(S^{2\sigma -2}/\Ctwo; \Z), $$
and $S^{2\sigma}=S^0*S(2\sigma)$  is the unreduced suspension of of $S(2\sigma)$, the second equivalence is a calculation of
the cohomology of $\R P^1$ . 
\end{proof}

\begin{remark}
This proof shows that if $\Ctwo$ is replaced by a cyclic group of any order
 we still have
$$\Z^{\HZu}=\HZu^* \simeq \Sigma^{\lambda} \HZu$$
where $\lambda =\eps -\alpha$ (with $\eps$ the trivial one
dimensional complex representation and $\alpha$ a faithful one
dimensional representation). 
\end{remark}

Anderson duality works, of course, also for non-equivariant spectra. We learned the following proposition comparing the equivariant and non-equivariant version in a conversation with Nicolas Ricka.

\begin{prop}\label{prop:AndersonFixed}Let $A$ be an abelian group. We have $(A^X)^{C_2} \simeq A^{(X^{C_2})}$ for every $C_2$-spectrum $X$. 
\end{prop}
\begin{proof}Let $\infl_e^{C_2} Y$ denote the inflation of a spectrum $Y$ to a $C_2$-spectrum with `trivial action', i.e.\ the left derived functor of first regarding it as a naive $C_2$-spectrum with trivial action and then changing the universe. This is the (derived) left adjoint for the fixed point functor \cite[Prop 3.4]{MM02}.

Let $I$ be an injective abelian group. Then there is for every spectrum $Y$ a natural isomorphism 
\begin{align*}
[Y, (I^X)^{C_2}] &\cong [\infl_e^{C_2} Y, I^X]^{C_2} \\
&\cong  \Hom(\pi_0^{C_2}(\infl_e^{C_2} Y \tensor X), I)\\
&\cong \Hom(\pi_0(Y\tensor X^{C_2}), I) \\
&\cong [Y, I^{(X^{C_2})}].
\end{align*}
Here, we use that fixed points commute with filtered homotopy colimits and cofibre sequences and therefore also with smashing with a spectrum with trivial action. Thus, there is a canonical isomorphism in the homotopy category of spectra between $I^{(X^{C_2})}$ and $(I^X)^{C_2}$ that is also functorial in $I$ (by Yoneda). For a general abelian group $A$, we can write $A^{(X^{C_2})}$ as the fibre of $(I^0)^{X^{C_2}} \to (I^1)^{X^{C_2}}$ (and similarly for the other side) for an injective resolution $0\to A\to I^0\to I^1$. Thus, we obtain a (possibly non-canonical) equivalence between $A^{(X^{C_2})}$ and $(A^X)^{C_2}$.
\end{proof}
\begin{remark}An analogous result holds, of course, for every finite group $G$. 
\end{remark} 

\subsection{Koszul complexes and derived power torsion}\label{sec:Koszul}
Let $R$ be a non-equivariantly $E_{\infty}$ $C_2$-ring spectrum and $M$ be an
$R$-module. In this section we will recall two versions of stable
Koszul complexes. Among their merits is that they provide models for
cellularization or $\R$-cellularization in cases of interest for us. A basic reference for the material in this section is \cite{G-M95}.  

As classically, the $r$-power torsion in a module $N$ can be defined as the kernel of $N \to N[\frac1r]$, we define the \emph{derived $J$-power torsion} of $M$ with respect to an ideal $J = (x_1,\dots, x_n) \subseteq \pi^{\Ctwo}_{\rost}(R)$ as
\begin{align*}\Gamma_J M = \fib(R \to R[\frac1{x_1}]) \tensor_R \cdots
  \tensor_R \fib(R \to R[\frac1{x_n}]) \tensor_R M .
\end{align*}
This is also sometimes called the \emph{stable Koszul complex} and also denoted by $K(x_1,\dots, x_n)$. As shown in \cite[Section 3]{G-M95}, this only depends on the ideal $J$ and not on the chosen generators. As algebraically, the derived functors of $J$-power torsion are the local cohomology groups, we might expect a spectral sequence computing the homotopy groups of $\Gamma_J M$ in terms of local cohomology. As in \cite[Section 3]{G-M95}, this takes the form
\begin{align}\label{eqn:localcohomology} H_J^s(\pi_{\rost+V}^{C_2}M) \Rightarrow \pi_{V-s}^{C_2}(\Gamma_JM).\end{align}

Our second version of the Koszul complex can be defined in the
one-generator case as
$$\kappa_R(x) = \hocolim \Sigma^{(1-l)|x|}R/x^l$$
for $x \in \pi^{\Ctwo}_{\rost}(R)$.
Here, the map $R/x^l \to \Sigma^{-|x|}R/x^{l+1}$ is induced by the diagram of cofibre sequences
\[
 \xymatrix{\Sigma^{|x^l|}R \ar[r]^{x^l}\ar[d]^=& R \ar[r]\ar[d]^x & R/x^l\ar@{-->}[d] \\
 \Sigma^{|x^l|}R \ar[r]^{x^{l+1}} & \Sigma^{-|x|}R \ar[r] & \Sigma^{-|x|}R/x^{l+1}
 }
\]

More generally, we can make for a sequence $\xu = (x_1,\dots, x_n)$ in $\pi^{\Ctwo}_{\rost}(R)$ the definition
\begin{align*}\kappa_R(\xu; M) &:= \kappa_R(x_1)\tensor_R\cdots \tensor_R \kappa_R(x_n)\tensor_R M\\
&\simeq \hocolim \Sigma^{-((l_1-1)+\cdots (l_n-1))|x|} M/(x_1^{l_1},\dots, x_n^{l_n})\end{align*}

The spectrum $\kappa_R(x)$ comes with an obvious filtration by $\Sigma^{(1-l)|x|}R/x^l$ with filtration quotients $\Sigma^{-l|x|}R/x$. We can smash these filtrations together to obtain a filtration of $\kappa_R(\xu)$
with filtration quotients wedges of summands of the form $\Sigma^{-l_1|x_1| -\cdots - l_n|x_n|}R/(x_1,\dots, x_n)$ (see \cite[1.3.11-12]{tilson} or \cite[2.8, 2.12]{tilsonArXiv}). Using the following lemma, we obtain also a corresponding filtration on $\Gamma_JR$. 

\begin{lemma}\label{lem:Koszul} For $\xu$ as above, we have 
\begin{align*}
 \kappa_R(\xu) &\simeq \Sigma^{|x_1|+\cdots +|x_n|+n}\Gamma_JR.
\end{align*}
\end{lemma}
\begin{proof}
See \cite[Lemma 3.6]{G-M95}.
\end{proof}

We can also define $\kappa_R(\xu;M)$ (and likewise the other versions of Koszul complexes) for an infinite sequence of $x_i$ by just taking the filtered homotopy colimit over all finite subsequences. Usually Lemma \ref{lem:Koszul} breaks down in the infinite case.

\begin{remark}\label{rmk:hocolim}
 The homotopy colimit defining $\kappa_R(\xu;M)$ has a directed cofinal subsystem, both in the
finite and in the infinite case. Indeed, the colimit ranges over all sequences $(l_1,l_2,\dots)$ with only finitely many entries nonzero. For the directed subsystem, we start with $(0,0,\dots)$ and raise in the $n$-th step the first $n$ entries by $1$. Directed homotopy colimit are well-known to be weak colimits in the homotopy category of $R$-modules, i.e.\ every system of compatible maps induces a (possibly non-unique) map from the homotopy colimit \cite[Sec 3.1]{Mar83} \cite[Sec II.5]{Sch07}.
\end{remark}

One of the reason for introducing $\Gamma_JM$ is that it provides a model for the $\R$-cellularization of $M$ with respect to $R/J = (R/x_1) \tensor_R \cdots \tensor_R (R/x_n)$ in the sense of Section \ref{sec:Cell}.
\begin{prop}\label{prop:cell}
 Suppose that $x_1\dots, x_n \in \pi_{*\rho}^{C_2}R$ and set $J = (x_1,\dots, x_n)$. Then $\Gamma_JM \to M$ is a $\R$-cellularization with respect to $R/J$ in the (triangulated) category of $R$-modules. 
\end{prop}
\begin{proof}
 Clearly, $\kappa_R(x_1,\dots, x_n; M)$ is $\R$-cellular with respect to $M/J$; furthermore $M/J$ is $R/J$-$\R$-cellular as clearly $M$ is $R$-cellular. To finish the proof, we have to show that 
 $$\Hom_R(R/J, \Gamma_JM) \to \Hom_R(R/J, M)$$
 is an equivalence. Note that $\Gamma_JM = \Gamma_{x_n}(\Gamma_{(x_1,\dots, x_{n-1})}M)$. Thus, it suffices by induction to show that 
 $$\Hom_R(A/x, \Gamma_x B) \to \Hom_R(A/x, B)$$
 is an equivalence for all $R$-modules $A,B$. This is equivalent to  
 $$\Hom_R(A/x, B[x^{-1}]) = 0$$
which is true as multiplication by $x$ induces an equivalence
 $$\Hom_R(A,B[x^{-1}]) \xrightarrow{x^*} \Hom_R(\Sigma^{|x|}A, B[x^{-1}]).\qedhere$$
\end{proof}

\begin{cor}\label{cor:cellular}
 Let $M$ be a connective $R$-module and $A$ an abelian group. Then the Anderson dual $A^M$ is $\R$-cellular with respect to $R/J$ for every ideal $J\subset \pi_{\rost}^{C_2}$ finitely generated in degrees $a+b\sigma$ with $a\geq 1$ and $a+b\geq 1$. 
\end{cor}
\begin{proof}
 By the last proposition, we have to show that $\Gamma_JA^M \simeq A^M$. For this it suffices to show that $A^M[x^{-1}]$ is contractible for every generator $x$ of $J$. As $M$ is connective, we know that $\pi_{a+b\sigma}M = 0$ if $a<0$ and $a+b<0$ (this follows, for example, using the cofibre sequence $(C_2)_+ \to S^0 \to S^\sigma$). Thus, $\pi_{a+b\sigma}A^M = 0$ if $a>0$ and $a+b>0$. The result follows. 
\end{proof}

\section{Real bordism and the spectra $BP\mathbb{R}\langle n \rangle$}\label{sec:BPRn}
\subsection{Basics and homotopy fixed points}\label{sec:BPRBasics}
The $\Ctwo$-spectrum $M\R$ of \emph{Real bordism} was originally defined by Araki and Landweber. In the naive model of $\Ctwo$-spectra, 
where a $\Ctwo$-spectrum is just given as a sequence $(X_n)$
 of pointed $\Ctwo$-spaces together with maps 
 $$\Sigma^{\rho}X_n \to X_{n+1}$$
 it is just given by the Thom spaces $M\R_n = BU(n)^{\gamma_n}$ 
 with complex conjugation as $\Ctwo$-action. Defining it as a strictly commutative $\Ctwo$-orthogonal spectrum requires more care and was done 
 in \cite[Example 2.14]{SchEquiv} and \cite[Section B.12]{HHR}. An important fact is that the geometric fixed points of $M\R$ are equivalent to 
 $MO$ (first proven in \cite{A-M} and reproven in \cite[Proposition B.253]{HHR}).
 
As shown in \cite{Ara79} and \cite[Theorem 2.33]{HK}, there is a
splitting 
$$M\R_{(2)} \simeq \bigoplus_{i}\Sigma^{m_i\rho}BP\R,$$
where the underlying spectrum of $BP\R$ agrees with $BP$. This
splitting corresponds on geometric fixed points to the splitting 
$$MO \simeq \bigoplus_{i}\Sigma^{m_i}H\F_2.$$ 
As shown in \cite{HK} (see also
Appendix \ref{Appendix}), the restriction map 
$$\pi_{*\rho}^{\Ctwo}BP\R\to \pi_{2*}BP$$ 
is an isomorphism. Choose now arbitrary
indecomposables $v_i \in \pi_{2(2^i-1)}BP$ and denote their lifts to
$\pi_{(2^i-1)\rho}^{\Ctwo}BP\R$ and their images in $\pi_{(2^i-1)\rho}^{\Ctwo}M\R$
by $\vb_i$. We denote by $\BPRn$ the quotient 
$$BP\R/(\vb_{n+1},\vb_{n+2},\dots)$$ 
in the homotopy category of $M\R$-modules. At least a priori, this depends on the choice of $v_i$. 

We want to understand the homotopy groups of $\BPRn$. This was first
done by Hu in \cite{Hu} (beware though that Theorem 2.2 is not correct
as stated there) and partially redone in   \cite{K-W13}. For the
convenience of the reader, we will give the computation again. Note
that our proofs are similar but not identical to the ones in the
literature. The main difference is that we do not use ascending
induction and prior knowledge of $H\Z$ to compute $\Phi^{C_2}\BPRn$,
but precise knowledge about $\pi_{\rost}^{C_2}BP\R$ -- this is not
simpler than the original approach, but gives extra information about
other quotients of $BP\R$, which we will need later. We recommend that
the reader looks at Appendix A for a complete understanding of the
results that follow. 

We will use the
Tate square \cite{GMTate} and consider the  following diagram in which
the rows are cofibre sequences: 

 \[\xymatrix@C=0.7cm{
  \BPRn \tensor (EC_2)_+ \ar[r]\ar[d]^\simeq & \BPRn \ar[r]\ar[d]& \BPRn\tensor \tilde{E}C_2 \ar[d]\ar[r] &\Sigma\BPRn\tensor (EC_2)_+  \ar[d]\\
  \BPRn^{(EC_2)_+} \tensor (EC_2)_+ \ar[r] & \BPRn^{(EC_2)_+} \ar[r]& \BPRn^{(EC_2)_+}\tensor \tilde{E}C_2 \ar[r] & \Sigma\BPRn\tensor (EC_2)_+ 
  }
 \]

After taking fixed points this becomes
 \[\xymatrix{
  \BPRn_{h\Ctwo} \ar[r]\ar[d]^= & \BPRn^{\Ctwo} \ar[r]\ar[d]& \BPRn^{\Phi \Ctwo} \ar[d]\ar[r] &\Sigma\BPRn_{h\Ctwo} \ar[d]\\
  \BPRn_{h\Ctwo} \ar[r] & \BPRn^{h\Ctwo} \ar[r]& \BPRn^{t\Ctwo} \ar[r] & \Sigma\BPRn_{h\Ctwo} 
  }
 \]

 First, we compute the homotopy groups of the homotopy fixed points. For this we use the $RO(C_2)$-graded homotopy fixed point spectral sequence, described for example in \cite[Section 2.3]{HM}.
 
 \begin{prop}\label{prop:BPRnHomotopy}
  The $RO(\Ctwo)$-graded homotopy fixed point spectral sequence
  \[E_2=H^*(\Ctwo; \pi^e_{\rost}\BPRn) \cong \Z_{(2)}[\vb_1,\dots, \vb_n, u^{\pm 1}, a]/2a \Rightarrow \pi^{\Ctwo}_\bigstar (\BPRn^{(E\Ctwo)_+})\]
  has differentials generated by $d_{2^{i+1}-1}(u^{2^{i-1}}) = a^{2^{i+1}-1}\vb_i$ for $i\leq n$ and $E_{2^{n+1}} = E_\infty$.
 \end{prop}
 \begin{proof}
  The description of $E_{2^{n+1}}$ is entirely analogous to the proof of \ref{prop:Differentials}, using that $a^{2^{i+1}-1}\vb_i = 0$ in $\pi_\bigstar^{\Ctwo}\BPRn^{(E\Ctwo)_+}$. Now we need to show that there are no further differentials: As every element in filtration $f$ is divisible by $a^f$ in $E^{2^{n+1}}$, the existence of a nonzero $d_m$ (with $m\geq 2^{n+1}$) implies the existence of a nonzero $d_m$ with source in the $0$-line. Moreover, a nonzero $d_m$ of some element $u^{l}\vb$ (for $\vb$ a polynomial in the $\vb_i$) on the $0$-line implies a nonzero $d_m$ on $u^{l}$ as $\vb$ is a permanent cycle (in the image from $BP\R$). The image of such a differential must be of the form $a^mu^{l'}\vb'$, where $\vb'$ is a polynomial in $\vb_1,\dots, \vb_n$. As $a^m\vb_i = 0$ for $1\leq i\leq n$ in $E^{2^{n+1}}$, the polynomial $\vb'$ must be constant. Counting degrees, we see that 
 \[(2l-1)-2l\sigma = |u^l|-1 = |a^mu^{l'}| = 2l' -(2l'+m)\sigma\]
 and thus $m = 2l-2l' = 1$. This is clearly a contradiction. 
 \end{proof}
 
 \begin{cor}
  We have 
  $$\pi^{C_2}_{\rost}(\BPRn^{(EC_2)_+}\tensor \tilde{E}C_2) \cong \F_2[u^{\pm 2^n}, a^{\pm 1}].$$
   In particular, we get $\pi_*\BPRn^{tC_2} \cong \F_2[x^{\pm 1}]$, where $x = u^{2^n}a^{-2^{n+1}}$ and $|x| = 2^{n+1}$. These are understood to be additive isomorphisms.
 \end{cor}
 \begin{proof}
  Recall that 
  $$\pi^{C_2}_{\rost}(\BPRn^{(EC_2)_+}\tensor \tilde{E}C_2) = \pi^{C_2}_{\rost}(\BPRn^{(EC_2)_+})[a^{-1}].$$
  as $S^{\infty\sigma}$ is a model of $\tilde{E}C_2$. The result follows as all $\vb_i$ are $a$-power torsion, but $u^{2^nm}$ is not. 
 \end{proof}
 
 \subsection{The homotopy groups of $BP\mathbb{R}\langle n\rangle$}\label{sec:BPRnC2}
 
 Computing the homotopy groups of the fixed points is more delicate
 than the computation of the homotopy fixed points. We first have to
 use our detailed knowledge about the homotopy groups of $\BPR$.
 Given a sequence $\underline{l} = (l_1,\dots)$, we denote by
 $BP\R/\underline{\vb}^{\underline{l}}$ the spectrum
 $BP\R/(\vb^{l_{i_1}}_{i_1}, \vb^{l_{i_2}}_{i_2},\dots)$, where $i_j$
 runs over all indices such that $l_{i_j}\neq 0$. Similarly
 $BP\R/\vb_i^j$ is understood to be $BP\R$ if $j=0$. We use the
 analogous convention when we have algebraic quotients of homotopy
 groups. 

We recommend the reader  skips the proof of the following result for
first reading, as the technical detail is not particularly
illuminating. 
 
 \begin{prop}\label{prop:BPBound}
Let $k\geq 1$ and $\underline{l} = (l_1,l_2, \dots)$ be a sequence of nonnegative integers with $l_k=0$. Then the element $\vb_k$ acts injectively on $(\pi_{*\rho-c}^{\Ctwo}\BPR)/\underline{\vb}^{\underline{l}}$ if $0\leq c \leq 2^{k+1}+1$, with a splitting on the level of $\Z_{(2)}$-modules. 
\end{prop}
\begin{proof}
Recall from Appendix \ref{Appendix} that $\pi_\bigstar^{\Ctwo}BP\R$ is isomorphic to the subalgebra of
 $$P/(2a, \vb_ia^{2^{i+1}-1})$$
 (where $i$ runs over all positive integers) generated by $\vb_i(j) = u^{2^ij}\vb_i$ (with $i,j\in\Z$ and $i\geq 0$) and $a$, where $P = \Z_{(2)}[a, \vb_i, u^{\pm 1}]$.  The degrees of the elements are $|a| = 1-\rho$ and $$|\vb_i(j)| = (2^i-1)\rho + 2^ij(4-2\rho) = (2^i-1-2^{i+1}j)\rho + 2^{i+2}j.$$
 We add the relations $\vb_i^{l_i} = 0$ if $l_i\neq 0$.
 
We will first show that the ideal of $\vb_k$-torsion elements in
$(\pi_\bigstar^{\Ctwo}BP\R)/\underline{\vb}^{\underline{l}}$ is
contained in the ideal generated by $a^{2^{k+1}-1}$ and
$\vb_s^{l_s-1}\vb_s(j)$ for $s$ with $l_s \neq 0$ and $j$ divisible by
$2^{k-s}$ if $s<k$. Indeed, because the ideal $(2a, \vb_ia^{2^{i+1}-1}, \underline{\vb}^{\underline{l}}) \subset P$ is generated by monomials, a polynomial in $P$ defines a $\vb_k$-torsion element in $(\pi_\bigstar^{\Ctwo}BP\R)/\underline{\vb}^{\underline{l}}$ if and
only if each of its monomials define $\vb_k$-torsion elements. A
monomial $x_P$ in $P$ can only define a nonzero $\vb_k$-torsion element in $(\pi_\bigstar^{\Ctwo}BP\R)/\underline{\vb}^{\underline{l}}$ if it
is divisible by $a^{2^{k+1}-1}$ or $\vb_s^{l_s}$. In the latter case,
$x_P$ is of the form $\vb \vb_s^{l_s}u^m$, where $\vb$ is a polynomial
in the $\vb_i$. This is divisible by $\vb_s^{l_s}$ in
$\pi_\bigstar^{\Ctwo}BP\R$ if and only if  $m$ is divisible by $2^i$ for some $\vb_i$ in $\vb$. Thus, $x_P$ defines a nonzero element $x$ in $(\pi_\bigstar^{\Ctwo}BP\R)/\underline{\vb}^{\underline{l}}$ such that $\vb_kx$ defines $0$ only if $2^k|m$, which corresponds to the condition above.  

Let $x$ be a nonzero $\vb_k$-torsion element in $(\pi_\bigstar^{\Ctwo}BP\R)/\underline{\vb}^{\underline{l}}$, represented by a monomial in $P$. First assume that $x$ is divisible by $a^n$ with $n\geq 2^{k+1}-1$, but not by $a^{n+1}$. Then, $x$ is not divisible by any $\vb_i(j)$ with $i\leq k$ as $a^n\vb_i(j) = 0$. We demand that $x$ is in degree $*\rho-c$ with $c\geq 0$; in particular, $x\neq a^n$. Let $\vb_i(j)$ a divisor of $x$ with minimal $i$. Thus, the degree of $x$ must be of the form $*\rho + 2^{i+2}m +n$. We know that $n \leq 2^{i+1}-2$. The largest negative value the non-$\rho$-part can take is $-2^{i+2}+2^{i+1}-2=-2^{i+1}-2$. The statement about injectivity follows in this case as $i>k$. 

Now assume that $x$ is a $\vb_k$-torsion element not divisible by $a^n$ for $n\geq 2^{k+1}-1$. Then $x$ must be of the form $\vb_s^{l_s-1}\vb_s(j)<$ where $j$ is divisible by $2^{k-s}$ if $s<k$. Observe that 
\[\vb_s^{l_s-1}\vb_s(j)\vb_t(m) = \vb_s^{l_s}\vb_t(2^{s-t}j+m) = 0 \in (\pi_\bigstar^{\Ctwo}BP\R)/\underline{\vb}^{\underline{l}}\]
 for $t<s$, so that $y$ is not divisible by any $\vb_t(m)$ for $t<s$. Likewise observe that if $s\leq t\leq k$, then 
 \[\vb_s^{l_s-1}\vb_s(j)\vb_t(m) = \vb_s^{l_s}\vb_t(m+2^{k-t}j') = 0 \in (\pi_\bigstar^{\Ctwo}BP\R)/\underline{\vb}^{\underline{l}},\]
 where $j = 2^{k-s}j'$. Thus, $y$ is also not divisible by any $\vb_t(m)$ with $s\leq t\leq k$. As $|\vb_s(j)| = *\rho +d$, where $d$ is divisible by $2^{k+2}$, and the same is true for $|\vb_t(j)|$ with $t>k$, we see that if $|x|$ is of the form $*\rho - c$ with $c\geq 0$, then we have $$c \geq 2^{k+2}-(2^{k+1}-2) = 2^{k+1}+2.$$ The statement about injectivity follows also in this case. 

We still have to show the split injectivity. 
If $\vb_k y = 2z$, but $y$ not divisible by $2$, then $y$ must be of the form $2\vb u^{2^kj}$ in $P$, where $\vb$ is a polynomial in the $\vb_i$. Thus, $|y| = 2^{k+2}j +*\rho$, so we are fine in degree $*\rho -c$ for $0\leq c\leq 2^{k+1}+1 \leq 2^{k+2}-1$. 
\end{proof}

\begin{remark}
 The exact bounds in the preceding proposition are not very important. The only important part for later inductive arguments is that the bound grows with $k$. 
\end{remark}

\begin{cor}\label{Cor:QuotientBP}
Let $\underline{l} = (l_1,l_2, \dots)$ be a sequence with only finitely many nonzero entries and let $j$ be the smallest index such that $l_j \neq 0$. Then the map 
\[
(\pi_{*\rho-c}^{\Ctwo}BP\R)/\underline{\vb}^{\underline{l}} \to \pi_{*\rho-c}^{\Ctwo}(BP\R/\underline{\vb}^{\underline{l}})
\] 
is an isomorphism for $0\leq c\leq 2^{j+1}$.
\end{cor}
\begin{proof}
We use induction on the number $n$ of nonzero indices in $\underline{l}$. If $n=0$ (and $j=\infty$), the statement is clear.

For the step, define $\underline{l}'$ to be the sequence obtained from $\underline{l}$ by setting $l_j = 0$. Consider the short exact sequence
\[0 \to (\pi_{*\rho-c}^{\Ctwo}(B/\underline{\vb}^{\underline{l}'}))/\vb_j^{l_j} \to \pi_{*\rho-c}^{\Ctwo}(B/\underline{\vb}^{\underline{l}}) \to \left\{\pi_{*\rho-(c+1)}^{\Ctwo}(B/\underline{\vb}^{\underline{l}'})\right\}_{\vb_j^{l_j}} \to 0.\]
Here, the notion $\{X\}_z$ denotes the elements in $X$ killed by $z$. 

Assume $c\leq 2^{j+1}$. By the induction hypothesis, $\pi_{*\rho-c}^{\Ctwo}(B/\underline{\vb}^{\underline{l}'})\cong (\pi_{*\rho-c}^{\Ctwo}B)/\underline{\vb}^{\underline{l}'}$ as $c\leq 2^{j+2}$, so that $(\pi_{*\rho-c}^{\Ctwo}(B/\underline{\vb}^{\underline{l}'}))/\vb_j^{l_j} \cong (\pi_{*\rho-c}^{\Ctwo}B)/\underline{\vb}^{\underline{l}}$. Furthermore,
\begin{align*}
\left\{\pi_{*\rho-(c+1)}^{\Ctwo}(B/\underline{\vb}^{\underline{l}'})\right\}_{\vb_j^{l_j}} &\cong \left\{(\pi_{*\rho-(c+1)}^{\Ctwo}B)/\underline{\vb}^{\underline{l}'}\right\}_{\vb_j^{l_j}} \\
&\cong 0
\end{align*}
as follows from $c+1 \leq 2^{j+2}$ and $c+1 \leq 2^{j+1}+1$ by the induction hypothesis and Proposition \ref{prop:BPBound}. Thus, we see that $(\pi_{*\rho-c}^{\Ctwo}B)/\underline{\vb}^{\underline{l}} \to \pi_{*\rho-c}^{\Ctwo}(B/\underline{\vb}^{\underline{l}})$ is an isomorphism.
\end{proof}

The following corollary is crucial:

\begin{cor}\label{Cor:crucial}
 Let $I \subset \Z_{(2)}[\vb_1,\dots]$ be an ideal generated by powers of the $\vb_i$. Then $BP\R/I$ is strongly even. 
\end{cor}
\begin{proof}
 As being strongly even is a property closed under filtered homotopy colimits, we are reduced to the case that $I$ is finitely generated. By the last corollary, it suffices to show that $BP\R$ itself is strongly even. That the Mackey functor $\underline{\pi}_{*\rho}^{\Ctwo}(BP\R)$ is constant is clear from Theorem \ref{thm:BPR}. 
 
 Assume that $x$ is a nonzero element in $\pi_{*\rho-1}^{\Ctwo}BP\R$. We can assume that $x$ is represented by $a^ku^l\vb$ in the $E_2$-term of the homotopy fixed point spectral sequence for $BP\R$, where $\vb$ is a monomial in the $\vb_i$ (with $\vb_0 =2$), $k\geq 0$ and $l\in\Z$. The element $x$ is in degree $k+4l + *\rho$. Let $j\geq 0$ be the minimal number such that $\vb_j|\vb$. Then $2^j|l$ and $k\leq 2^{j+1}-2$. This is clearly in contradiction with $k+4l = -1$. 
\end{proof}

We recover the $C_2$-case of the reduction theorem of \cite[Prop 4.9]{HK} and \cite[Thm 6.5]{HHR}.
\begin{cor}\label{cor:reduction}
 There is an equivalence $\BPR/(\vb_1,\vb_2,\dots) \simeq H\Zu_{(2)}$.
\end{cor}
\begin{proof}
 This follows directly from the last corollary and Corollary \ref{cor:characterizingZ}. 
\end{proof}

\begin{cor}\label{cor:rho+}
   Let $I \subset \Z_{(2)}[\vb_1,\dots]$ be an ideal generated by powers of the $\vb_i$. Then 
   $$\pi_{*\rho+1}^{C_2}BP\R/I \cong \F_2\{a\}\tensor \Z_{(2)}[\vb_1, \vb_2,
   \dots]/I.$$
 \end{cor}
 \begin{proof}
  As $BP\R/I$ is strongly even, this follows from \cite[Lemma 2.15]{HM}.
 \end{proof}

This allows us to compute $\pi_\bigstar^{\Ctwo}\BPRn$.
 
 \begin{prop}\label{prop:BPRnFixedPoints}
  The spectrum $\BPRn$ is the connective cover of its Borel completion $\BPRn^{(E\Ctwo)_+}$. The cofibre $C$ of $\BPRn \to \BPRn^{(E\Ctwo)_+}$ has homotopy groups $$\pi_\bigstar^{\Ctwo}C \cong \F_2[a^{\pm 1}, u^{-2^n}]u^{-2^n},$$
  with the naming of the elements indicating the map $\pi_\bigstar^{\Ctwo}\BPRn^{(E\Ctwo)_+} \to \pi_\bigstar^{\Ctwo}C$.
 \end{prop}
 \begin{proof}
  This is clear on underlying homotopy groups. Thus, we have only to show that $\BPRn^{\Ctwo} \to \BPRn^{h\Ctwo}$ is a connective cover. For that purpose, it is enough to show that $\BPRn^{\Phi \Ctwo}$ is connective and that the fibre of $\BPRn^{\Phi \Ctwo} \to \BPRn^{t\Ctwo}$ has homotopy groups only in negative degrees. 
  
  We have $\BPRn^{\Phi \Ctwo} \simeq \BPR^{\Phi \Ctwo}/(\vb_{n+1},\dots)$. As noted in the proof of Proposition \ref{prop:vbn}, the image of $\vb_i$ in $M\R^{\phi \Ctwo}$ is $0$. As the quotient $\BPR^{\Phi \Ctwo}/\vb_{n+1},\dots$ can be taken in the category of $M\R^{\Phi \Ctwo}$-modules, we are only quotiening out by $0$. It follows easily that $(BP\R/(\vb_{n+1},\dots, \vb_{n+m}))^{\Phi \Ctwo}$ has in the homotopy groups an $\F_2$ in all degrees of the form $\sum_{i=n+1}^{n+m}\varepsilon_i(|v_i|+1) = \sum_{i=n+1}^{n+m}\varepsilon_i2^i$ with $\varepsilon_i = 0$ or $1$. As geometric fixed point commute with homotopy colimits, we see that $\pi_*\BPRn^{\Phi \Ctwo} \cong \F_2[y]$ (additively) with $|y| = 2^{n+1}$. It remains to show that $y^k$ maps nonzero to $\pi_*\BPRn^{t\Ctwo}$ (and hence maps to $x^k)$. 
  
  It is enough to show that $a^{-|y^k|-1}y^k$ maps nonzero to $\pi_\bigstar^{\Ctwo} \Sigma \BPRn\tensor (E\Ctwo)_+$ in the sequence coming from the Tate square, i.e.\ that $a^{-|y^k|-1}y^k$ is not in the image from (the fixed points of) $\BPRn$. But $a^{-|y^k|-1}y^k$ is in degree $(|y^k|+1)\rho-1$ and $\pi_{(|y^k|+1)\rho-1}^{\Ctwo}\BPRn = 0$ by Corollary \ref{Cor:crucial}. 
 \end{proof}
 
 Let us describe the homotopy groups of $\BPRn$ in more detail. We set $\vb_0 = 2$ for convenience. Denote by $BB$ (for \emph{basic block}) the $\Z_{(2)}[a, \vb_1,\dots,\vb_n]/2a$-submodule of 
 $$\Z_{(2)}[\vb_1,\dots, \vb_n]/(a^{2^{k+1}-1}\vb_k)_{0\leq k \leq n}$$
 generated by $1$ and by $\vb_k(m) = u^{2^km}\vb_k$ for $0\leq k < n$ and $0< m <2^{n-k}$.
By Proposition \ref{prop:BPRnHomotopy}, we see that 
 $$\pi_{\rost}^{C_2}\BPRn^{(EC_2)_+} \cong BB[U^{\pm 1}]$$
 with $U = u^{2^n}.$ Note that this isomorphism is not  claimed to be multiplicative; in general, $\BPRn$ is not even known to have any kind of (homotopy unital) multiplication.
 
 Define $BB'$ to be the kernel of the map $BB \to \F_2[a]$ given by sending all $\vb_k$ and $\vb_k(m)$ to zero. Set $NB = \Sigma^{\sigma-1}\F_2[a]^\vee \oplus BB'$, where $NB$ stands for \emph{negative block}. Then it is easy to see from the last proposition that 
 $$\pi_{\rost}^{C_2}\BPRn \cong BB[U] \oplus U^{-1}NB[U^{-1}],$$ 
 where this isomorphism is again only meant additively. We will be a little bit more explicit about the homotopy groups of $\BPRn$ in the cases $n=1$ and $2$ in Part \ref{part:LocalCohomology}. 

\subsection{Forms of $BP\mathbb{R}\langle n\rangle$}Our next goal is
to identify certain spectra as forms of  $\BPRn$. We take the following definition from \cite{HM}:

\begin{defn}
Let $E$ be an even $2$-local commutative and associative $\Ctwo$-ring spectrum up to homotopy. By \cite[Lemma 3.3]{HM}, $E$ has a Real orientation and after choosing one, we obtain a formal group law on $\pi_{*\rho}^{\Ctwo}E$. The $2$-typification of this formal group law defines a map $\pi^e_{2*}BP \cong \pi_{*\rho}^{C_2}BP\R \to \pi_{*\rho}^{C_2}E$. We call $E$ a \emph{form of $BP\R\langle n\rangle$} if the map
\[\underline{\Z_{(2)}[\vb_1,\dots, \vb_n]} \subset \underline{\pi}_{*\rho}BP\R \to \underline{\pi}_{*\rho} E\]
is an isomorphism of constant Mackey functors. 

This depends neither on the choice of $\vb_i$ nor on the chosen Real orientation, as can be seen using that $\vb_i$ is well-defined modulo $(2, \vb_1,\dots, \vb_{i-1})$. 
\end{defn}

Equivalently, one can say that $E$ is a form of $\BPRn$ if and only if
$E$ is strongly even and its underlying spectrum is a form of
$BP\langle n\rangle$. We want to show that every form of $BP\R\langle
n\rangle$ is also of the form  $\BPR /\vb_{n+1}, \vb_{n+2}, \ldots $
for some choice of elements $\vb_i$. For this, we need the following lemma from \cite[Lemma 3.4]{HM}:
\begin{lemma}\label{lem:regrep}
Let $f\colon E\to F$ be a map of $\Ctwo$-spectra. Assume that f induces isomorphisms 
\[\pi^{\Ctwo}_{k\rho}E \to \pi^{\Ctwo}_{k\rho}E \quad \text{and} \quad \pi_kE \to \pi_kF\]
for all $k\in\Z$. Assume furthermore that $\pi^{\Ctwo}_{k\rho-1}E \to \pi^{\Ctwo}_{k\rho-1}F$ is an injection for all $k\in\Z$ (for example, if $\pi^{\Ctwo}_{k\rho-1}E =0$). Then $f$ is an equivalence of $\Ctwo$-spectra.
\end{lemma}
\begin{prop}\label{prop:FormBPRn}
 Let $E$ be a form of $BP\R\langle n\rangle$. Then one can choose
 indecomposables $\vb_i\in \pi_{(2^i-1)\rho}^{\Ctwo}\BPR$ for $i\geq
 n+1$ such that $E \simeq \BPR /(\vb_{n+1},\vb_{n+2},\dots)$. 
\end{prop}
\begin{proof}
 First choose any system of $\vb_i$. Choose furthermore a Real orientation $f\colon BP\R \to E$ and denote $f(\vb_i)$ by $x_i$. Define a multiplicative section 
 $$s\colon \pi_{*\rho}^{\Ctwo}E \to \pi_{*\rho}^{\Ctwo}\BPR$$ by $s(x_i) = \vb_i$ for $1\leq i \leq n$. 
 
 Now define a new system of $\vb_i$ by 
 \[\vb_i^{\mathrm{new}} = \vb_i - s(f_*(\vb_i))\]
 for $i\geq n+1$. As these agree with $\vb_i$ mod $(\vb_1,\dots, \vb_n)$, they are still indecomposable. Furthermore, the $\vb_i^{\mathrm{new}}$ are for $i\geq n+1$ clearly in the kernel of $f_*$. Thus, we obtain a map $\BPRn/(\vb_{n+1}^{\mathrm{new}},\vb_{n+2}^{\mathrm{new}},\dots) \to E$ that is an isomorphism on $\pi_{*\rho}^{\Ctwo}$. By Corollary \ref{Cor:crucial}, the source is strongly even. By Lemma \ref{lem:regrep}, the map is an equivalence. 
\end{proof}

\begin{examples}\label{exa:forms}We consider Real versions of the classical examples $ku$ and $tmf_1(3)$.
 \begin{enumerate}
  \item The connective Real K-theory spectrum $k\R_{(2)}$ is a form of $BP\R\langle 1\rangle$. Indeed, the underlying spectrum $ku_{(2)}$ is well known to be a form of $BP\langle 1\rangle$ and $k\R_{(2)}$ is also strongly even (as can be seen by the results from \cite[3.7D]{B-G10} or from the computation in Section \ref{sec:kRlcss}). 
  \item Define $\overline{tmf_1(3)}$ as the equivariant connective cover of the spectrum $\overline{Tmf_1(3)}$, i.e.\ $Tmf_1(3)$ with the algebro-geometrically defined $\Ctwo$-action (see \cite[Section 4.1]{HM} for details). As shown in \cite[Corollary 4.17]{HM}, $\overline{tmf_1(3)}_{(2)}$ is a form of $BP\R\langle 2\rangle$. By Proposition \ref{prop:FormBPRn}, we can construct $\overline{tmf_1(3)}_{(2)}$ by killing a sequence $\vb_2, \vb_3,\dots$ in $BP\R$. This construction is used in \cite{L-OString} to define a $\Ctwo$-equivariant version of $tmf_1(3)_{(2)}$. In particular, we see (using the discussion before Proposition 4.23 in \cite{HM}) that $\overline{TMF_1(3)}_{(2)}$ (with the algebro-geometrically defined $\Ctwo$-action) agrees with the $\mathbb{TMF}_1(3)_{(2)}$ of \cite{L-OString}.
 \end{enumerate}
\end{examples}

\section{Results and consequences}\label{sec:results}
In this section, we want to discuss our main results in more detail than in the introduction and we will also derive some consequences and give some examples. Recall to that purpose the notation from Sections \ref{sec:Koszul} and \ref{sec:BPRBasics}. Furthermore, we will implicitly localize everything at $2$ so that $\Z$ means $\Z_{(2)}$ etc. Our main theorem is the following:
\begin{thm}\label{thm:main}
Let $(m_1,m_2,\dots)$ be a sequence of nonnegative integers with only finitely many entries bigger than $1$ and let $M$ be the quotient $\BPR/(\vb_1^{m_1},\vb_2^{m_2},\dots)$, where we only quotient by the positive powers of $\vb_i$. Denote by $\underline{\vb}$ the sequence of $\vb_i$ in $\pi^{\Ctwo}_{\rost}M\R$ such that $m_i = 0$, by $|\underline{\vb}|$ the sum of their degrees and by $m'$ the sum of all $(m_i-1)|\vb_i|$ for $m_i> 1$. Then 
$$\Z^M \simeq \Sigma^{-m'+4-2\rho}\kappa_{M\R}(\underline{\vb}; M).$$

The most important case is that $m_{n+1} = m_{n+2} = \cdots = 1$ so that 
$$M = \BPRn/(\vb_1^{m_1},\dots,\vb_n^{m_n}).$$
If $k$ is the number of elements in $\underline{\vb}$, we also get 
\begin{align*}
\Z^M \simeq \Sigma^{-m'+k+|\underline{\vb}| +4-2\rho}\Gamma_{\underline{\vb}}M,
\end{align*}
where we view $M$ as an $M\R$-module. 
\end{thm}
The first form will be proved as Theorem \ref{Thm:QuotientDuality} and the second follows from it using Lemma \ref{lem:Koszul}. The second form also follows from Corollary \ref{cor:BPRnGorDdish} (using that $\Gamma_{\underline{\vb}}$ preserves cofibre sequences to pass to quotients of $\BPRn$). 

\begin{example}\label{ex:BPRn}
$\Z^{\BPRn} \simeq \Sigma^{n+D_n\rho +4-2\rho}\Gamma_{(\vb_1,\dots, \vb_n)}\BPRn$ for $D_n = |v_1|+\cdots +|v_n|$. This says that $\BPRn$ has Gorenstein duality with respect to $\HZu \simeq \BPRn/(\vb_1,\dots, \vb_n)$. (The last equivalence follows from Corollary \ref{cor:reduction}.)
\end{example}

\begin{example}
Set $k\R(n) = \BPRn/(\vb_1,\dots, \vb_{n-1})$ to be connective
integral Real Morava $K$-theory and $K\R(n) = k\R(n)[\vb_n^{-1}]$ its periodic version. Then 
\begin{align*}\Z^{k\R(n)} &\simeq \Sigma^{1+|\vb_n|+4-2\rho}\Gamma_{\vb_n}k\R(n)\\
&\simeq \Sigma^{(2^n-3)\rho+4}\cof(k\R(n) \to K\R(n))\end{align*}
This includes for $n=1$ the case of usual ($2$-local) connective Real K-theory. 
\end{example}
\begin{example}
 To have a slightly stranger example, take $M = \BPR\langle 3\rangle/(\vb_1^4, \vb_3^2)$. Then 
 $$\Z^M \simeq \Sigma^{5-9\rho}\Gamma_{\vb_2}M.$$
\end{example}

\vspace*{0.5cm}

So far, we have only talked about \emph{quotients} of $\BPR$. This does not include important Real spectra like the Real Johnson--Wilson theories $E\R(n) = \BPRn[\vb_n^{-1}]$ or the (integral) Real Morava K-theories $K\R(n)$. For this, we have to study the behaviour of our constructions under localizations. 

Let $M$ be an $RO(C_2)$-graded $\Z[v]$-module, where $v$ has some degree $|v| \in RO(C_2)$. We say that $M$ has \emph{bounded $v$-divisibility} if for every degree $a+b\sigma$, there is a $k$ such that 
$$v^k\colon M_{a+b\sigma-|v^k|} \to M_{a+b\sigma}$$
is zero. We will also apply the concept to modules that are just $\Z|v|$-graded.
\begin{lemma}The class of $RO(C_2)$-graded $\Z[v]$-modules of bounded $v$-divisibility is closed under submodules, quotients and extensions. 
\end{lemma}
\begin{proof}
 This is clear for submodules and quotients. Let 
 $$0 \to K \to M \to N \to 0$$
 be a short exact sequence of $\Z[v]$-modules where $K$ and $N$ are of bounded $v$-divisibility. For a given degree $\alpha \in RO(C_2)$, we know that there is a $k$ such that $v^k$ maps trivially into $K_\alpha$. Furthermore, there is an $n$ such that $v^n$ maps trivially into $N_{\alpha-k|v|}$. Thus, multiplication by $v^{n+k}$ is the zero map $M_{\alpha-(k+n)|v|} \to M_\alpha$.
\end{proof}

Let $M$ be an $M\R$-module. We say that $M$ is of \emph{bounded $\vb_n$-divisibility} if both $\pi^{\Ctwo}_{\rost}M$ and $\pi^e_*M$ are of bounded $\vb_n$-divisibility. This is, for example, true if $M$ is connective. 

\begin{lemma}\label{lem:boundedrho}We have the following two properties of $\vb_n$-divisiblity.
\begin{enumerate}
\item Being of bounded $\vb_n$-divisibility is closed under cofibres and suspensions. 
\item An $M\R$-module $M$ is of bounded $\vb_n$-divisibility if and only if $\pi^{\Ctwo}_{*\rho}M$ and $\pi^e_*M$ are of bounded $\vb_n$-divisibility. 
\end{enumerate}
\end{lemma}
\begin{proof}
Both statements follow from the last lemma. For the second item, we additionally use the exact sequence
$$\pi^e_{a+b+1}M \to \pi^{C_2}_{a+(b+1)\sigma}M \to \pi^{C_2}_{a+b\sigma}M \to \pi^e_{a+b}M$$
induced by the cofibre sequence
$$(C_2)_+ \to S^0\to S^{\sigma}.\qedhere$$
\end{proof}

\begin{lemma}
If $M$ has bounded $\vb_n$-divisibility, then there is a natural equivalence 
$$M[\vb_n^{-1}] \simeq \Sigma \holim\left( \cdots \to \Sigma^{|\vb_n|}\Gamma_{\vb_n}M \xrightarrow{\vb_n} \Gamma_{\vb_n} M\right)$$
of $M\R$-modules.
\end{lemma}
\begin{proof}
We apply the endofunctor $H\colon N\mapsto \holim(\cdots \to \Sigma^{|\vb_n|}N \xrightarrow{\vb_n} N)$ of $M\R$-modules to the cofibre sequence
$$\Gamma_{\vb_n}M \to M\to M[\vb_n^{-1}].$$
Clearly $H(M[\vb_n^{-1}])\simeq M[\vb_n^{-1}]$. Thus, we just have to show that $H(M)\simeq 0$. This follows by the $\lim^1$-sequence and bounded $\vb_n$-divisibility. 
\end{proof}

\begin{lemma}
Let $B$ be a quotient of $BP\R$ by powers of the $\vb_i$. Then
$B[\vb^{-1}]$ has bounded $\vb_n$-divisibility if $\vb$ is a product
of $\vb_i$ not containing $\vb_n$. Hence, the same is also true for
the stable Koszul complex $\Gamma_{\underline{\vb}}B$, where $\underline{\vb}$ is a sequence of $\vb_i$ not containing $\vb_n$.
\end{lemma}
\begin{proof}
By Lemma \ref{lem:boundedrho}, it is enough to check the first statement on $\pi_{*\rho}^{C_2}$ and on $\pi^e_*$. On the latter, it is clear and the former is isomorphic to it by Corollary \ref{Cor:crucial}. For the second statement we use that $\Gamma_{\underline{\vb}}B$ is the fibre of $B \to \check{C}(\underline{\vb};B)$, where $\check{C}(\underline{\vb};B)$ has a filtration with subquotients $M\R$-modules of the form $\Sigma^?B[x^{-1}]$ for some $x\in \pi_{\rost}^{C_2}M\R$ \cite[Lemma 3.7]{G-M95}. Thus, the second statement follows from Lemma \ref{lem:boundedrho}. 
\end{proof}

\begin{thm}
Let the notation be as in Theorem \ref{thm:main} and assume for simplicity that only finitely many $m_i$ are zero and that $m_n = 0$. Then 
$$\Z^{M[\vb_n^{-1}]} \simeq \Sigma^{-m'+|\underline{\vb}|+(k-1)+4-2\rho}\Gamma_{\underline{\vb}\setminus \vb_n} M.$$
Here $\underline{\vb} \setminus \vb_n$ denotes the sequence of all $\vb_i$ such that $m_i = 0$ and $i\neq n$.  
\end{thm}
\begin{proof}
The preceding lemmas imply the following chain of equivalences:
\begin{align*}
\Z^{M[\vb_n^{-1}]} &\simeq \Z^{\hocolim (M \xrightarrow{\vb_n} \Sigma^{-|\vb_n|}M \xrightarrow{\vb_n} \cdots)} \\
&\simeq \holim (\cdots \xrightarrow{\vb_n} \Z^M) \\
&\simeq \Sigma^{-m'+|\underline{\vb}|+k+4-2\rho}\holim \left(\cdots \xrightarrow{\vb_n} \Gamma_{\underline{\vb}}M\right)\\
&\simeq \Sigma^{-m'+|\underline{\vb}|+k+4-2\rho}\holim \left(\cdots \xrightarrow{\vb_n} \Gamma_{\vb_n}(\Gamma_{\underline{\vb}\setminus \vb_n}M) \right) \\
&\simeq \Sigma^{-m'+|\underline{\vb}|+(k-1)+4-2\rho}(\Gamma_{\underline{\vb}\setminus \vb_n} M)[\vb_n^{-1}] \\
&\simeq \Sigma^{-m'+|\underline{\vb}|+(k-1)+4-2\rho}\Gamma_{\underline{\vb}\setminus \vb_n} (M[\vb_n^{-1}])
\end{align*}
\end{proof}

\begin{example}
We recover the following result by Ricka \cite{Ricka}: 
$$\Z^{K\R(n)} \simeq \Sigma^{4-2\rho}K\R(n).$$
Here, $K\R(n)$ denotes integral Morava K-theory $E\R(n)/(\vb_1,\dots, \vb_{n-1})$. 
\end{example}

\begin{example}
In the following, we will use that there are invertible classes $x,\vb_n\in\pi_\bigstar^{\Ctwo}E\R(n)$ of degree $-2^{2n+1}+2^{n+2}-\rho$ and $(2^n-1)\rho$ respectively, where $x = \vb_n^{1-2^n}u^{2^n(1-2^{n-1})}$.
\begin{align*}
\Z^{E\R(n)} &\simeq \Sigma^{D_{n-1}\rho + (n-1)+4-2\rho} \Gamma_{(\vb_1,\dots, \vb_{n-1})}E\R(n) \\
&\simeq \Sigma^{-(n+2)\rho+(n+3)} \Gamma_{(\vb_1,\dots, \vb_{n-1})}E\R(n) \\
&\simeq \Sigma^{(n+2)(2^{2n+1}-2^{n+2})+n+3} \Gamma_{(\vb_1,\dots, \vb_{n-1})}E\R(n).
\end{align*}
This says that $E\R(n)$ has Gorenstein duality with respect to $E\R(n)/(\vb_1,\dots, \vb_{n-1}) = K\R(n)$. Note that we can replace the ideal $(\vb_1,\dots, \vb_{n-1})$ by an ideal generated in integral degrees, namely $(\vb_1x, \dots, \vb_{n-1}x^{2^{n-1}-1})$. 
\end{example}

\begin{example}\label{exa:tmf}
Recall from \cite{HM} the spectra $tmf_1(3)$, $Tmf_1(3)$ and $TMF_1(3)$ and the corresponding $C_2$-spectra $\overline{tmf_1(3)}$, $\overline{Tmf_1(3)}$ and $\overline{TMF_1(3)}$. Recall that we have $\pi_*tmf_1(3) = \Z[a_1,a_3]$, where $a_1$ and $a_3$ can be identified with the images of the Hazewinkel generators $v_1$ and $v_2$, and that $\overline{tmf_1(3)}$ is a form of $BP\R\langle 2\rangle$ (as already discussed in Example \ref{exa:forms}). This gives the Anderson dual of $\overline{tmf_1(3)}$. Tweaking the last theorem a little bit, allows also to show that 
$$\Z^{\overline{TMF_1(3)}} \simeq \Sigma^{5+2\rho}\Gamma_{\vb_1} \overline{TMF_1(3)}.$$
We can also recover one of the main results of \cite{HM}, namely that $\Z^{\overline{Tmf_1(3)}} \simeq \Sigma^{5+2\rho} \overline{Tmf_1(3)}$. Indeed, $Tmf_1(3)$ is by \cite[Section 4.3]{HM} the cofibre of the map 
$$\Gamma_{\vb_1,\vb_2}\overline{tmf_1(3)} \to \overline{tmf_1(3)}.$$
As the source is equivalent to $\Sigma^{-6-2\rho}\Z^{\overline{tmf_1(3)}}$, applying Anderson duality shows that $\Z^{\overline{Tmf_1(3)}}$ is the fibre of
$$\Sigma^{6+2\rho}\overline{tmf_1(3)} \to \Sigma^{6+2\rho} \Gamma_{\vb_1,\vb_2} \overline{tmf_1(3)}.$$
This is equivalent to $\Sigma^{5+2\rho}\overline{Tmf_1(3)}$. 
This example does not require 2-localization, only that $3$ is inverted.
\end{example}

\begin{remark}By Proposition \ref{prop:AndersonFixed}, all the results in this section have direct implications for the Anderson duals of the fixed point spectra. These are easiest to understand in the case of $ER(n) = (E\R(n))^{C_2}$, where we get
$$\Z^{ER(n)} \simeq \Sigma^{(n+2)(2^{2n+1}-2^{n+2})+n+3} \Gamma_{(\vb_1x,\dots, \vb_{n-1}x^{2^n-1})}ER(n).$$
\end{remark}

\vspace{0.8cm}
\part{The Gorenstein approach}
In this part, we explain the Gorenstein approach to prove Gorenstein duality, first for $\kR$ and then for $\BPRn$.
\section{Connective $K$-theory with Reality}
\label{sec:kR}
The present section considers $K$-theory with reality, which is more
familiar than $\BPRn$ for general $n$, and no 2-localization is
necessary. The arguments are especially
simple, firstly because $\kR$ is a commutative ring spectrum, and
secondly becaue we only need to
consider principal ideals. Simple as the argument is, we see in
Section \ref{sec:kRlcss} that the consequences for coefficient rings
are interesting.  

\subsection{Gorenstein condition and Matlis lift}
It is well known that there is a cofibre sequence
$$\Sigma^{\eps}ku\stackrel{v}\lra ku \lra H\Z.   $$
If one knows the coefficient ring $ku_*=\Z [v]$, this is easy 
to construct, since we can identify $ku/v$ as the Eilenberg-MacLane
spectrum from its homotopy groups. 

There is a version with Reality \cite{Dugger}. Indeed, we may
construct  the cofibre sequence
$$\Sigma^{\rho} \kR \stackrel{\vb}\lra \kR \lra \HZu,  $$
where $\kR /\vb$ is identified using Corollary \ref{cor:characterizingZ}
 
Since the Dugger sequence is self dual we immediately deduce that
$\kR$ is Gorenstein. 

\begin{lemma}
\label{lem:kRGor}
$$\Hom_{\kR}(\HZu, \kR)=\Sigma^{-\rho-1}\HZu$$
and $\kR \lra \HZu$ is Gorenstein. 
\end{lemma}

\begin{proof}
Apply $\Hom_{\kR}(\cdot , \kR)$ to the Dugger sequence. 
\end{proof}

To actually get Gorenstein duality we need to construct a Matlis
lift (adapted from \cite[Section 6]{DGI}), which is a counterpart in topology of the injective hull of the residue
field. 
\begin{defn}
If $M$ is an $\HZu$-module, we say that a $\kR$-module $\tilde{M}$ is
a {\em Matlis lift} of $M$ if $\tilde{M}$ is $\HZu$-$\R$-cellular and 
$$\Hom_{\kR}(T, \tilde{M})\simeq \Hom_{\HZu}(T, M)$$
for all $\HZu$-modules $T$.
\end{defn}

The Anderson dual provides one such example. 

\begin{lemma}
\label{lem:ML}
The $\kR$-module $\Sigma^{-2(1-\sigma)}\Z^{\kR}$  is a Matlis lift of
$\HZu$. Indeed, 

(i)  $\Zu^{\kR}$ is $\HZu$-$\R$-cellular and 

(ii) There is an equivalence 
$$\Sigma^{2\pp} \HZu\simeq \HZu^*=\Hom_{\kR} (\HZu, \Z^{\kR}), $$
where $\delta =1-\sigma$. 
\end{lemma}

\begin{proof}
One could prove the first part from the slice tower, but it also follows directly from Corollary \ref{cor:cellular}. 

The second statement is immediate from Lemma \ref{lem:Zu}. 
\end{proof}

\subsection{Gorenstein duality}
 We next want to move on to Gorenstein duality, so we write
$$\cE=\Hom_{\kR}(\HZu, \HZu). $$

Combining Lemmas \ref{lem:kRGor} and \ref{lem:ML}, we have 
\begin{eqnarray}\label{eq:kReq}\Hom_{\kR}(\HZu, \kR)\simeq \Sigma^{-\rho-1}\HZu \simeq
\Hom_{\kR}(\HZu, \Sigma^{-4+\sigma}\Z^{\kR})\end{eqnarray}

We now want to remove the $\Hom_{\kR}(\HZu , \cdot )$ from this
equivalence. 

\begin{lemma} {\em (Effective constructibility)}
\label{lem:effective}
The evaluation map 
$$\Hom_{\kR}(\HZu, M)\tensor_{\cE}\HZu \lra M$$
is $\HZu$-$\R$-cellularization for every left $\kR$-module $M$.
\end{lemma}

\begin{proof}
Since the domain is clearly $\HZu$-$\R$-cellular, it is enough to show the map is an
equivalence for all cellular modules $M$.

This is clear for $M=\HZu$. The class of $M$ for which the statement is true is closed under (i) triangles, (ii)
coproducts (since $\HZu$ is small) and (iii) suspensions by
representations. This gives all $\R$-cellular modules.  
\end{proof}

Local cohomology gives an alternative approach to
cellularization. Recall that we define the $\vb$-power torsion of a $\kR$-module $M$
by the fibre sequence
$$\Gamma_{\vb}M \lra M \lra M[1/\vb]. $$

The following lemma is a special case of Proposition \ref{prop:cell}. 
\begin{lemma}
\label{lem:Gammacell}
The map 
$$\Gamma_{\vb}M\lra M$$ 
is $\HZu$-$\R$-cellularization.
\end{lemma}

It remains to check that the two $\cE$-actions on $\HZu$ coincide.

\begin{lemma}\label{lem:UniquenesskR}
There is a unique right $\cE$-module structure on $\HZu$.
\end{lemma}

\begin{proof}
Suppose that $\HZu'$ is a right $\cE$-module whose underlying
$C_2$-spectrum is equivalent to the
Eilenberg-MacLane spectrum $\HZu$. 
We first claim that $\HZu'$ can be constructed as an
$\cE$-module with cells in degrees $k\rho$ for $k\leq 0$:
$$\HZu'\simeq_{\cE} S^0_{\cE}\cup e^{-\rho}_{\cE}\cup e^{-2\rho}_{\cE} \cup \cdots $$

Once that is proved, we argue as follows.  If $\HZu''$ is another right
$\cE$-module with underlying $C_2$-spectrum $\HZu$, we may construct a
map $\HZu'\lra \HZu''$ skeleton by skeleton in the usual way. 
We start with the $\cE$-module map $\cE =(\HZu')^{(0)}\lra \HZu'$ giving
the unit, and successively extend the map over the
cells of $\HZu'$. At each stage the obstruction to the existence of an 
extension over $(\HZu')^{-k\rho}$ lies in $\pi^{\Ctwo}_{-k\rho-1}(\HZu'')$. 
These groups are zero. We end with a map which is an isomorphism on 
0th homotopy Mackey functors and therefore an equivalence.

For the cell-structure, it is enough to show that for every right
$\cE$-module $\HZu'$ of the homotopy type of the Eilenberg--MacLane
spectrum $\HZu$, there is a map $\cE \to \HZu'$ of right $\cE$-modules
whose fibre has the homotopy type of $\Sigma^{-\rho-1}\HZu$. Indeed, suppose we have already constructed a right $\cE$-module $(\HZu')^{(n)}$ with an $\cE$-map to $\HZu'$ with fibre of the homotopy type $\Sigma^{-(n+1)\rho-1}\HZu$. Then it is easy to see that the cofibre $(\HZu')^{(n+1)}$ of the map $\Sigma^{-(n+1)\rho -1}\cE \to \Sigma^{-(n+1)\rho-1}\HZu \to (\HZu')^{(n)}$ has the analogous property. Taking the homotopy colimit, we get a map $\hocolim (\HZu')^{(n)} \to \HZu'$ with fibre $\hocolim \Sigma^{-(n+1)\rho-1}\HZu$, which is clearly zero (e.g.\ by Lemma \ref{lem:regrep} and the fact that $\HZu$ is even; we refer to \cite[Section 3.4]{Ricka} for a table of $\underline{\pi}_{\rost}^{C_2}\HZu$). 

We choose the map $f\colon \cE \to \HZu'$ representing $1\in \pi_0^{\Ctwo}\HZu'$ and call the fibre $F$. We want to show that $f$ agrees with the canonical map $\cE \to \HZu$ on homotopy groups of the form $\pi_{k-\sigma}^{\Ctwo}$ for $k\in\Z$. Indeed, the only nonzero class in $\HZu'$ in these degrees is $a\in\pi_{-\sigma}^{\Ctwo}\HZu'$, which has to be hit by $a\in \pi_{-\sigma}^{\Ctwo}\cE$ as it comes from the sphere. Thus, $\pi_{k-\sigma}^{\Ctwo} F \cong \pi_{k-\sigma}^{\Ctwo} \Sigma^{-1-\rho}\HZu$ for all $k$ and hence $F\simeq \Sigma^{-1-\rho}\HZu$ as $C_2$-spectra, as we needed to show. 
\end{proof}

From this the required statement follows. 

\begin{cor} {\em (Gorenstein duality)} 
\label{cor:kRGorD}
There is an equivalence of $\kR$-modules
$$\Gamma_{\vb} \kR \simeq \Sigma^{-4+\sigma} \Z^{\kR}. \qqed$$
\end{cor}
\begin{proof}
 By \eqref{eq:kReq} and Lemma \ref{lem:UniquenesskR}, we know that 
 $$ \Hom_{\kR}(\HZu, \kR)\tensor_{\cE} \HZu\simeq 
\Hom_{\kR}(\HZu, \Sigma^{-4+\sigma}\Z^{\kR})\tensor_{\cE}\HZu.$$
By Lemma \ref{lem:effective}, the two sides are the cellularizations of $\kR$ and $\Sigma^{-4+\sigma}\Z^{\kR}$ respectively. By Lemmas \ref{lem:Gammacell} and \ref{lem:ML}, the former is $\Gamma_{\vb}\kR$ and the latter is $\Sigma^{-4+\sigma}\Z^{\kR}$ itself. 
\end{proof}

The implications of this equivalence for the coefficient ring are
investigated in Section \ref{sec:kRlcss}. 

\section{\texorpdfstring{$\protect BP\langle n \rangle$}{BP<n>} with Reality}
\label{sec:dishonest}

We now turn to the case of $\BPRn$ for a general $n$.  The counterpart
of the argument of Section \ref{sec:kR} is a little simpler when  $\BPRn$ is a commutative ring
spectrum.  For $n=1$ and $n=2$, the spectra $\kR$, and  $tmf_1(3)$,  are both known to be a
commutative ring spectra, and their 2-localizations give $\BPRn$
when $n=1$ and $n=2$ respectively.  However for higher
$n$  it is not known that $\BPRn$ is a commutative ring spectrum. 
This is a significant technical issue, but
one that is familiar when working with non-equivariant $BP$-related theories
since $BP$ is not known to be a commutative ring. The established
method for getting around this is to use the fact that 
$BP$ and $\BPn$ are modules over the commutative ring  $MU$.
We will adopt precisely the same method by working with $\MUR$-modules.
The only real complication is that 
we are forced to work with spectra whose homotopy groups are bigger
than we might like,  but if we focus on the relevant part, it causes
no real difficulties.

\subsection{Gorenstein condition and Matlis lift}
As mentioned in the introduction of this section, we will work in the setting of $\MUR$-modules. More precisely, we will always (implicitly) localize at $2$ and set $S = \MUR_{(2)}$.  As discussed in Section \ref{sec:BPRBasics}, we can define $S$-modules $\BPRn$, once we have chosen a sequence of $\vb_i$ (for example, the Hazewinkel or Araki generators). 

The ideal 
$$\Jb_n=(\vbn{1}, \ldots, \vbn{n})$$
plays a prominent role, and we will abuse notation by writing 
$$S/\Jbn:=\cof(S\stackrel{\vbn{1}}\lra S)\tensor_S
\cof(S\stackrel{\vbn{2}}\lra S)\tensor_S \cdots \tensor_S
\cof(S\stackrel{\vbn{n}}\lra S),  $$
and then 
$$M/\Jbn:=M\tensor_S S/\Jbn.$$
In particular, 
$$\BPRn/\Jbn=\BPRn/\vbn{n}/\vbn{n-1}/\cdots /\vbn{1}\simeq \HZu  $$
by the $C_2$-case of the reduction theorem, here proved as Corollary \ref{cor:reduction}. 

If $\BPRn$ is a ring spectrum
$$\Hom_{\BPRn}(\HZu, M)=\Hom_{\BPRn}(\BPRn\tensor_S S/\Jbn,
M)=\Hom_{S}(S/\Jbn, M), $$
so that the right hand side gives a way for us to express the
fact that certain $\BPRn$-modules (such as $\BPRn$ and $\Z^{\BPRn}$)
are Matlis lifts, using only  module structures over $S$.

Applying this when $M=\BPRn$, we obtain the Gorenstein condition. 

\begin{lemma}
\label{lem:BPRnGordish}
The map $\BPRn \lra \HZu$ is Gorenstein of shift $-D_n\rho -n$ 
in the sense that 
$$\Hom_S(S/\Jbn, \BPRn)\simeq \Sigma^{-D_n\rho -n}\HZu, $$
where
$$D_n\rho =|\vbn{n}|+|\vbn{n-1}|+\cdots
+|\vbn{1}|=\left[2^{n+1}-n-2 \right]\rho . $$
\end{lemma}

\begin{proof}
Since each of the maps $\vbn{i}: \Sigma^{|\vbn{i}|}S\lra S$ is self-dual,  
for any $S$-module $M$, we have 
$$\Hom_S(S/\Jbn, M)\simeq \Sigma^{-D_m\rho-n}S/\Jbn \tensor_S M. $$
\end{proof}

Applying this when $M=\Z^{\BPRn}$, we obtain the Anderson Matlis lift.

\begin{lemma}
\label{lem:MLB}
The Anderson dual of $\BPRn$ is a Matlis lift of $\HZu^*$ in the sense that 

(i)  $\Z^{\BPRn}$ is $\HZu$-$\R$-cellular and 

(ii) There is an equivalence 
$$\Sigma^{2-2\sigma}\HZu \simeq \HZu^* \simeq \Hom_{S} (S/\Jbn, \Z^{\BPRn}). $$
\end{lemma}

\begin{proof}
One could prove the first part from the slice tower, but it also follows directly from Corollary \ref{cor:cellular}. 

For the second statement observe that 
$$\Hom_{S} (S/\Jbn, \Z^{\BPRn}) \simeq \Hom_S(S/\Jbn\tensor_S \BPRn, \Z^S)\simeq \Z^{\HZu}.$$ 
Thus, Lemma \ref{lem:Zu} implies the statement. 
\end{proof}

\subsection{Gorenstein duality}
Throughout this section, we will write $R = \BPRn$ for brevity. 
Combining Lemmas \ref{lem:BPRnGordish} and \ref{lem:MLB}, we have an
equivalence of $S$-modules
$$\Hom_{S}(S/\Jbn, R)\simeq \Sigma^{-D_n \rho-n}\HZu \simeq
\Hom_{S}(S/\Jbn, \Sigma^{-(D_n+n+2)-(D_n-2)\sigma}\Z^R)$$

We now want to remove the $\Hom_{S }(S/\Jbn , \cdot )$ from this
equivalence. The endomorphism ring 
$$\cEtn =\Hom_S(S/\Jbn, S/\Jbn)$$
of the small $S$-module $S/\Jbn$, replaces $\cE_n=\Hom_{R}(\HZu,
  \HZu)$ from the case that $R =\BPRn$ is a ring spectrum.  We note that 
$$\cEtn\tensor_S R=\Hom_S(S/\Jbn, S/\Jbn )\tensor_S R\simeq 
\Hom_S(S/\Jbn, S/\Jbn)\tensor_S R). $$
If $R = \BPRn$ were a commutative ring, this would be a ring equivalent to
$\Hom_R(\HZu, \HZu)$.

In any case, the following is proved exactly like Lemma \ref{lem:effective}

\begin{lemma} {\em (Effective constructibility)}\label{lem:effective2}
The evaluation map 
$$\Hom_{S }(S/\Jbn, M)\tensor_{\cEtn}S/\Jbn \lra M$$
is $S/\Jbn$-$\R$-cellularization.\qqed 
\end{lemma}

Of course local cohomology gives an alternative approach to
cellularization. Recall that we define
$$\Gamma_{\Jb_n}M =\Gamma_{\vbn{1} }S
\tensor_{S}\Gamma_{\vbn{2}}S \tensor_{S}\cdots \tensor_{S}
\Gamma_{\vbn{n} }S  \tensor_{S}M. $$
Then Proposition \ref{prop:cell} gives the following lemma. 

\begin{lemma}
$$\Gamma_{\Jb_n}M\lra M$$ 
is $\HZu$-$\R$-cellularization.
\end{lemma}

It remains to check that the two $\cEtn$ actions on $\HZu$ coincide. For
$\kR$ (i.e., $n=1$) we showed there was a unique right $\cE_n$-module
structure on $\HZu$. This may be true for $\cEtn$-module structures, but we will instead
just prove in the next subsection that the two particular $\cEtn$-modules that
arose from the left and right hand ends of the first display of this subsection are equivalent. 

The required Gorenstein duality statement follows. Its implications  for the coefficient ring for
$n=2$ are investigated explicitly in Section \ref{sec:tmfotlcss}. 

\begin{cor} 
\label{cor:BPRnGorDdish}
{\em (Gorenstein duality)} There is an equivalence of $\MUR$-modules
$$\Gamma_{\Jb_n} R \simeq \Sigma^{-(D_n+n+2)-(D_n-2)\sigma} \Z^R$$
with $R = \BPRn$.
\end{cor}
\begin{proof} 
We will argue in Subsection \ref{subsec:Eequiv} that the equivalence
$$\Hom_S(S/\Jbn, R) \simeq \Hom_S(S/\Jbn, \Sigma^{-D_n\rho
-n-2\delta} \Z^R), $$
is in fact an equivalence of right
modules over $\cEt_n$. By Lemma \ref{lem:effective2}, $R$ and 
$\Sigma^{-(D_n+n+2)-(D_n-2)\sigma}
\Z^R$ have equivalent $S/\Jbn$ cellularizations. We have seen above that the cellularization of $R$ is $\Gamma_{\Jb_n} \BPRn $ and that $\Sigma^{-D_n\rho -n-2\delta}
\Z^R$ itself is cellular. 
\end{proof}

\subsection{The equivalence of induced and coinduced Matlis lifts of
 $\protect \HZu$}
\label{subsec:Eequiv}
For brevity we will still write $R=\BPRn$, and note that we have a map
$S=\MUR \lra \BPRn=R$. The two $S$-modules that concern us are of a
very special sort, one looks as if it is obtained from an $S$-module
by  `extension of scalars from $S$ to $R$' and one
looks as if it is obtained by  `coextension of scalars from $S$ to $R$'.

\begin{lemma}
\label{lem:resEtnEn}
We have equivalences of right $\cEt_n$-modules
$$\Hom_S(S/\Jbn, R)\simeq  \Hom_S(S/\Jbn, S)\tensor_SR.$$
$$\Hom_S(S/\Jbn, \Z^R)=\Hom_S(R, \Hom_S(S/\Jbn, \Z^S))$$
\end{lemma}

\begin{proof}
The first equivalence is immediate from the smallness of
$S/\Jbn$. 

The second equivalence follows from the equivalence 
$$\Z^R\simeq \Hom_S(R, \Z^S)$$
of $S$-modules. 
\end{proof}

Suspending the equivalences from Lemma \ref{lem:resEtnEn} so that we are comparing two $\cEt_n$-modules
equivalent to $\HZu$ (see Lemma \ref{lem:MLB}) we have
$$Y_1=\Hom_S(S/\Jbn, \Sigma^{D_n\rho+n}R)\simeq \Hom_S(S/\Jbn,\Sigma^{D_n\rho+n}S)\tensor_SR=X_1\tensor_SR$$
and 
$$Y_2=\Hom_S(S/\Jbn, \Sigma^{2\pp} \Z^R)\simeq \Hom_S(S,
\Hom_S(S/\Jbn, \Sigma^{2\pp} \Z^S)) =\Hom_S(R,X_2). $$

In Subsection \ref{subsec:alpha} we will construct an $\cEt_n$-map $\alpha: X_1\lra Y_2$ and then argue
in Subsection \ref{subsec:alphat} that this extends along $X_1=X_1\tensor_SS\lra X_1\tensor_SR =Y_1$ to
give a map $\alphat: Y_1\lra Y_2$ which is easily seen to be an
equivalence: it is clearly a $\rhost$ isomorphism
and hence an equivalence by Lemma \ref{lem:regrep}. 

To see our strategy, note that  the extension problem 
$$\diagram
X_1\dto \rto^-{\alpha} & \Hom_S(S/\Jbn, \Hom_S(R,\Z^S))\\
X_1\tensor_SR\ar@{-->}[ur]_-{\alphat}&
\enddiagram$$
in the category of $\cEt_n$-modules is equivalent to the extension problem
$$\diagram
X_1 \tensor_{\cEt_n} S/\Jbn \tensor_S R
\dto \rto^-{\alpha'} & \Z^S\\
X_1 \tensor_{\cEt_n} S/\Jbn \tensor_S R\tensor_SR\ar@{-->}[ur]_-{\alphat'}&
\enddiagram$$
in the category of $S$-modules. The point is that by the defining property of the Anderson dual, this
latter extension problem can be tackled by looking in
$\pi^{\Ctwo}_0$. The 0th homotopy groups of the spectra on the left are
easily calculated from the known ring $\pi^{\Ctwo}_{\rost}(\HZu)$. 

\subsection{Construction of the map $\alpha$}
\label{subsec:alpha}

We construct the map $\alpha$ using a similar method as in the proof of Lemma \ref{lem:UniquenesskR}.

\begin{lemma}
There is a map 
$$\alpha : X_1 \lra Y_2$$
of right $\cEt_n$-modules that takes the image of $1\in
\pi^{\Ctwo}_0(S)$ to a generator of $\pi^{\Ctwo}_0(\HZu)=\Z$. 
\end{lemma}

\begin{proof}
First we claim that $X_1$ has a $\cEt_n$-cell structures 
with one 0-cell and other cells in dimensions which are negative
multiples of $\rho$. More precisely, there is 
 a filtration 
$$\cEt_n\simeq X_1^{[0]}\to X_1^{[1]}\to
X_1^{[2]}\to \cdots \to X_1$$
so that $X_1\simeq \hocolim_dX_1^{[d]}$ and there are cofibre sequences
$$X_1^{[d-1]}\lra X_1^{[d]}\lra \bigvee \Sigma^{-d\rho} \cEt_n. $$

By definition $X_1 = \Hom_S(S/\Jbn,\Sigma^{D_n\rho+n}S)$. By Proposition \ref{prop:cell} and Lemma \ref{lem:Koszul}, this is equivalent to 
$$Hom_S(S/\Jbn,\Sigma^{D_n\rho+n}\Gamma_{\Jbn}S) \simeq \Hom_S(S/\Jbn, \kappa_S(\vb_1,\dots, \vb_n))$$
because $\Gamma_{\Jbn}S \to S$ is $S/\Jbn$-$\R$-cellularization. The usual construction of the stable Koszul complex from the unstable
Koszul complex recalled in Subsection \ref{sec:Koszul}, shows that 
$$\kappa_S(\vb_1,\dots, \vb_n)$$ 
has a filtration with subquotients sums of $(-k\rho)$-fold suspensions of $S/\Jbn$. This induces a corresponding filtration on $X_1$.

As in Lemma \ref{lem:UniquenesskR} we may construct $\alpha$ by obstruction
theory.  Indeed, we start
by choosing a map $\cEt_n=X_1^{[0]}\lra Y_2^{[0]}$ taking the unit to
a generator.  At the $d$th stage we have a
problem 
$$\diagram
X_1^{[d-1]} \rto \dto & Y_2\\
X_1^{[d]} \ar@{-->}[ur]&
\enddiagram$$
The obstruction to extension is in a finite product of groups
$$[\Sigma^{-d\rho-1}\cEt_n, Y_2]^{\cEt_n}=\pi^{\Ctwo}_{-d\rho -1}(\HZu)=0 $$
where the vanishing is from the known value of $\pi^{\Ctwo}_{\rost}(\HZu)$. 
\end{proof}

\subsection{The map $\alphat$}
\label{subsec:alphat}
Referring to the second extension problem  diagram above, we note $S/\Jbn\tensor_SR\simeq
\HZu$ as $S$-modules. Thus, we have to solve the lifting problem
$$\diagram
X_1 \tensor_{\cEt_n} \HZu \tensor_S S \dto_{1\tensor1\tensor \pi} \rto^-{\alpha'} & \Z^S\\
X_1 \tensor_{\cEt_n} \HZu \tensor_SR\ar@{-->}[ur]_-{\alphat'}&
\enddiagram$$
where $\HZu$ is equipped with some $\cEt_n$-module structure. Denote the upper left corner by $T$. The map $T \to T\tensor_S R$ is a split inclusion on underlying $MU$-modules. Indeed, 
$$T\simeq X_1 \tensor_{\cEt_n} S/\Jbn \tensor_S R$$
 and the map $R \to R\tensor_S R$ is a split inclusion on underlying spectra because $BP\langle n\rangle$ has the structure of a homotopy unital $MU$-algebra \cite[V.2.6]{EKMM}. 

By the definition of Anderson duals, we have a diagram of short exact sequences:
\[\xymatrix{
 0 \ar[r]& \Ext_\Z^1(\pi^{\Ctwo}_{-1}(T\tensor_SR),\Z)\ar[d] \ar[r]& [T \tensor_SR, \Z^S]^S \ar[d]\ar[r]& \Hom_\Z(\pi^{\Ctwo}_0(T \tensor_SR),\Z)\ar[d] \ar[r]& 0 \\
 0\ar[r]&\Ext_\Z^1(\pi^{\Ctwo}_{-1}(T),\Z)\ar[r]& [T, \Z^S]^S \ar[r]& \Hom_\Z(\pi^{\Ctwo}_0(T),\Z) \ar[r]& 0
 }
\]

We want to show that the maps $\pi_k^{C_2}T \to \pi_k^{C_2}T\tensor_S R$ are split injections for $k=0,-1$, which solves the problem. For the computation of $\pi_*^{C_2}T$ recall from the last section that $X_1$ has a filtration starting with $X_1^{[x]} = \cEt_n$ and with subquotients sums of terms of the form $\Sigma^{-d\rho}\cEt_n$. Thus, $T$ obtains a filtration starting with $T^{[1]} = \HZu$ and with subquotients sums of terms of the form $\Sigma^{-d\rho}\HZu$. The map $\HZu = T^{[1]} \to T$ clearly induces isomorphisms on $\underline{\pi}_k^{\Ctwo}$ for $k = 0,-1$ by the known homotopy groups of $\HZu$ (see e.g.\ \cite[Section 3.4]{Ricka} for a table). Thus, $\underline{\pi}^{C_2}_{-1}T = 0$ and $\underline{\pi}^{C_2}_0T = \Zu$. 

If we have a map $\Zu \to M$ from the constant Mackey functor, it is a split injection on $(C_2/C_2)$ if it is one on $(C_2/e)$. But we have already seen above that on underlying spectra $T\to T\tensor_SR$ is a split inclusion. Thus, we have shown that $\pi_k^{C_2}T \to \pi_k^{C_2}(T\tensor_S R)$ is split injective, which provides the map $\tilde{\alpha}'$.

\vspace{1cm}
\part{The hands-on approach}
In this part, we give a different way to compute the Anderson dual of $\BPRn$ by first computing the Anderson dual of $BP\R$ itself. Again, we will first do the case of $\kR$. 

\section{The case of $k\mathbb{R}$ again}\label{sec:kRagain}
To illustrate our strategy, we give an alternative calculation of the
Anderson dual of $k\R$. This can also be deduced from our main theorem
below, but it might be helpful to see the proof in this  simpler case
first. General references for the $RO(C_2)$-graded homotopy groups of $k\R$ are \cite[Section 3.7]{B-G10} or Section \ref{sec:kRgroups}.

We want to show the following proposition:

\begin{prop}
There is an equivalence $\kappa_{k\R}(\vb) \to \Sigma^{2\rho -4}\Z^{k\R}.$
\end{prop}
Recall here that $\vb \in \pi^{C_2}_{\rho}\kR$ is the Bott element for Real K-theory and 
$$\kappa_{k\R}(\vb) = \hcolim_n \Sigma^{-(n-1)\rho}k\R/\vb^n.$$
 Our idea is simple: To obtain a map from the homotopy colimit, we have just to give maps 
 $$\Sigma^{-(n-1)\rho}k\R/\vb^n \to \Sigma^{2\rho -4}\Z^{k\R}$$
  that are compatible in the homotopy category (see Remark \ref{rmk:hocolim}). We will show in the next lemma that these maps are essentially unique: The Mackey functor of homotopy classes of $k\R$-linear maps $\Sigma^{-(n-1)\rho}k\R/\vb^n \to \Sigma^{2\rho -4}\Z^{k\R}$ is isomorphic to $\Zu$ and the precomposition with the map $\Sigma^{-(n-1)\rho}k\R/\vb^n  \to \Sigma^{-n\rho}k\R/\vb^{n+1}$ induces the identity on $\Zu$. 

Choosing the $C_2$-equivariant map $\kappa_{k\R}(\vb) \to \Sigma^{2\rho -4}\Z^{k\R}$ that corresponds to $1\in \Z$ for every $n$ induces an equivalence on underlying homotopy groups. By Lemma \ref{lem:regrep} the result follows as soon as we have established that $\kappa_{k\R}(\vb)$ is strongly even and that the Mackey functor $\underline{\pi}_{*\rho}\Sigma^{2\rho-4}\Z^{k\R}$ is constant. These two facts will also be shown in the following lemma, finishing the proof of the proposition.

\begin{lemma}Denote for a $\underline{\Z}[\vb]$ module $M$ by $\{M\}_{\vb^n}$ the $\vb^n$-torsion in it. Then we  have:
 \begin{enumerate}
  \item $k\R/\vb^n$ is strongly even and hence the same is true for $\kappa_{k\R}(\vb)$.
  \item $\underline{\pi}_{n\rho}^{\Ctwo}\Sigma^{2\rho-4}\Z^{k\R} \cong \underline{\pi}^{\Ctwo}_{(n-2)\rho+4}\Z^{k\R}$ is constant for all $n\in\Z$.
  \item $\underline{[\Sigma^{-(n-1)\rho}k\R/\vb^n, \Sigma^{2\rho-4}\Z^{k\R}]}^{C_2}_{k\R} \cong \left\{\underline{\pi}^{\Ctwo}_{-(n-1)\rho}\Sigma^{2\rho-4}\Z^{k\R}\right\}_{\vb^n} \cong \underline{\Z}$
 \end{enumerate}
\end{lemma}
\begin{proof}
 The first part follows as 
 $$\underline{\pi}^{\Ctwo}_{k\rho -i}(k\R/\vb^n) =\underline{\pi}^{\Ctwo}_{k\rho -i}(k\R)/\vb^n$$
  for $i=0,1$ because $\pi_{k\rho-i}^{C_2}k\R = 0$ for $i=1,2$. 
 
 For the second part consider the short exact sequence
 \[0\to \Ext(\underline{\pi}^{C_2}_{k\rho-5}k\R, \Z) \to \underline{\pi}^{C_2}_{-k\rho +4}\Z^{k\R} \to \Hom(\underline{\pi}^{C_2}_{k\rho -4}k\R,\Z) \to 0.\]
  We have $\underline{\pi}^{\Ctwo}_{k\rho-5}k\R = 0$ for all $k\in\Z$. For $k<2$, the Mackey functor $\underline{\pi}^{\Ctwo}_{k\rho-4}k\R$ vanishes as well and for $k\geq 2$, we have $\underline{\pi}^{\Ctwo}_{k\rho-4}k\R \cong \Zu^*$, generated by $v^{k-2}$ and $2\vb^{k-2}u$. Thus, 
  $$\underline{\pi}^{\Ctwo}_{-k\rho+4}\Z^{k\R}  \cong \begin{cases} 0 & \text{ if }k<2 \\
   \underline{\Z} & \text{ if }k\leq 2 \end{cases}$$
   This shows part (2).  As multiplication by $\vb^n$ does not hit $\underline{\pi}^{\Ctwo}_{(n+1)\rho-4}k\R$, the whole Mackey functor $\underline{\pi}^{\Ctwo}_{-(n+1)\rho+4}\Z^{k\R}$ is $\vb^n$-torsion. This gives the second isomorphism of the third part.

 For the remaining isomorphism, note that the cofibre sequence
 \[
  \Sigma^\rho k\R \xrightarrow{\vb^n} \Sigma^{-(n-1)\rho}k\R \to \Sigma^{-(n-1)\rho}k\R/\vb^n \to \Sigma^{\rho+1} k\R
 \]
induces a short exact sequence
 \[0 \to (\underline{\pi}^{\Ctwo}_{\rho +1}\Sigma^{2\rho-4}\Z^{k\R} )/ \vb_n \to \underline{[\Sigma^{-(n-1)\rho}k\R/\vb^n, \Sigma^{2\rho-4}\Z^{k\R}]}^{C_2}_{k\R} \to \left\{\underline{\pi}^{\Ctwo}_{-(n-1)\rho}\Sigma^{2\rho-4}\Z^{k\R}\right\}_{\vb^n} \to 0\]
 
 We have $\underline{\pi}_{\rho +1}^{\Ctwo}\Sigma^{2\rho-4}\Z^{k\R}  \cong \underline{\pi}_{5 -\rho}^{\Ctwo}\Z^{k\R}$, which sits in a short exact sequence
 \[ 0 \to \Ext_\Z(\underline{\pi}_{\rho-6}^{\Ctwo}k\R,\Z) \to \underline{\pi}_{5 -\rho}^{\Ctwo}\Z^{k\R} \to \Hom_\Z(\underline{\pi}_{\rho-5}^{\Ctwo}k\R,\Z)\to 0.\]
 But because of connectivity, $\underline{\pi}_{\rho-c}^{\Ctwo}k\R = 0$ for $c\geq 3$. 
 \end{proof}

\section{Duality for $BP\mathbb{R}$}
We will use throughout the abbreviation $B = BP\R$ and will furthermore implicitly localize everything at $2$ so that $\Z = \Z_{(2)}$ etc.\ and all $\Hom$ and $\Ext$ groups are over $\Z = \Z_{(2)}$ unless marked otherwise. Denote by $\underline{\vb}$ a sequence of indecomposable elements $\vb_i \in \pi_{(2^i-1)\rho}^{C_2}B$. The aim of this section is to show that $\Sigma^{2\rho-4}\Z^{B} \simeq \kappa_{M\R}(\underline{\vb}; B)$. 

Recall that $\kappa_{M\R}(\underline{\vb}; B)$ is defined as follows: Given a sequence $\underline{l} = (l_1,l_2,\dots)$ with $l_i\geq 0$, we denote by $B/\underline{\vb}^{\underline{l}}$ the spectrum $B/(\vb^{l_{i_1}}_{i_1}, \vb^{l_{i_2}}_{i_2},\dots)$, where $i_j$ runs over all indices such that $l_{i_j}> 0$. Set 
$$|\underline{l}| = l_1|\vb_1|+l_2|\vb_2| + \cdots $$
Then 
$$\kappa_{M\R}(\underline{\vb}; B) = \hcolim_{\underline{l}} \Sigma^{-|\underline{l}-\underline{1}|}B/\underline{\vb}^{\underline{l}},$$
where $\underline{l}$ runs over all sequences $\underline{l}$ where all but finitely many $l_i$ are zero and $\underline{1}$ denotes the constant sequence of ones. Furthermore, the $i$-th entry of $\underline{l}-\underline{1}$ is defined to be the maximum of $0$ and $l_i-1$. 

Thus, to get a map $\kappa_{M\R}(\underline{\vb}; B) \to \Sigma^{2\rho-4}\Z^{B}$, we have to understand the homotopy classes of maps $B/\underline{\vb}^{\underline{l}} \to \Sigma^{2\rho-4}\Z^{B}$. This will be the content of the next subsection.

\subsection{Preparation}
Recall the Mackey functor $\Zu^*$ defined by
$$\underline{\Z}^*(\Ctwo/\Ctwo) \cong \underline{\Z}^*(\Ctwo/e)\cong \Z$$
with transfer equalling $1$ while restriction is multiplication by $2$. 
\begin{lemma}\label{lem:computation}
 As $\Zu[\vb_1,\vb_2,\dots]$-modules, we have the following isomorphisms.
 \begin{enumerate}
  \item $\underline{\pi}^{C_2}_{*\rho -4}B \cong \Zu^*\tensor_\Z \Z[\vb_1,\vb_2,\dots]$ where $\Zu^*$ is generated by $1$ on underlying and by $2u^{-1}$ on $C_2$-equivariant homotopy groups.
  \item $\underline{\pi}^{C_2}_{*\rho -5}B = 0$
  \item $\pi^{C_2}_{*\rho -6}B \cong \F_2\{a^2\vb_1(-1)\} \tensor_\Z \Z[\vb_1,\vb_2,\dots]$. 
 \end{enumerate}
\end{lemma}
\begin{proof}
 By Theorem \ref{thm:BPR}, the groups $\pi_{*\rho -c}^{C_2}B$ are additively generated by nonzero elements of the form $x = a^l\vb$ with $\vb$ a monomial in the $\vb_i(j)$. Let $\vb_i(j)$ be the one occuring with minimal $i$, where $j$ is chosen such that $\vb = \vb_i(j)\vb'$ with $\vb'$ a monomial in the $\vb_k$ (this is possible by the third relation in Theorem \ref{thm:BPR}). Then $|x| = *\rho +j2^{i+2}+l$ and $0 \leq l< 2^{i+1}-1$. 
 
 For $c=4$, this implies $j=-1$, $i=0$ and $l=0$. Thus, $x$ is of the form $\vb_0(-1)\vb'$. As the restriction of $\vb_0(-1)$ to $\pi_0^eB$ equals $2$, the result follows. 
 
 For $c=5$, we must have $l \geq 2^{i+2}-5$, which implies $l\geq 2^{i+1}-1$ or $i=0$; in the latter case $l$ must be zero, which is not possible.
 
 For $c=6$, we must have $l = -j2^{i+2}-6$, which implies $l\geq 2^{i+1}-1$ or $i\leq 1$ and $j=-1$. As $i=0$ is again not possible, $x = a^2\vb_1(-1)\vb'$ with $\vb' \in \pi_{*\rho}^{C_2}$.
\end{proof}

\begin{lemma}\label{lem:AndersonQuot}
 For a sequence $\ul = (l_1,l_2, \dots)$, the map
 $$\underline{\pi}^{C_2}_{*\rho+4} \Z^{B/\uvb^{\ul}} \to \Hom(\Mpi_{-*\rho-4}B/\uvb^{\ul},\Z) \cong \Zu \tensor_{\Z} (\Z[\vb_1,\vb_2,\dots]/\uvb^{\ul})^*$$
 is an isomorphism, where $\Z[\vb_1,\vb_2,\dots]^{*} = \Hom_\Z(\Z[\vb_1,\vb_2,\dots], \Z)$ (so that the gradings become nonpositive). Here, the second map is the dual of the map 
 $$\Zu^* \tensor_{\Z} \Z[\vb_1,\vb_2,\dots]/\uvb^{\ul} \to \Mpi_{-*\rho-4}B/\uvb^{\ul}$$
 sending $1 \in \Zu^*(C_2/C_2)$ to the image of $u^{-1}$ under the map $B\to B/\uvb^{\ul}$ and $1\in\Zu^*(C_2/e)$ to $1$. 
\end{lemma}
\begin{proof}
 We have a short exact sequence
 $$0 \to \Ext(\Mpi_{-*\rho-5}B/\uvb^{\ul}, \Z) \to \underline{\pi}^{C_2}_{*\rho-4} \Z^{B/\uvb^{\ul}} \to \Hom(\Mpi_{-*\rho-4}B/\uvb^{\ul}, \Z) \to 0.$$
 If $l_1=0$, then Corollary \ref{Cor:QuotientBP} and Lemma \ref{lem:computation} directly imply the statement. If $l_1\neq 0$, Corollary \ref{Cor:QuotientBP} only allows us to identify the homotopy Mackey functor in degree $-*\rho-4$, but not the one in degree $-*\rho-5$. We give a separate argument in this case.
 
 If $l_1\neq 0$, consider the sequence $\ul' = (0,l_2,l_3,\dots)$ and the corresponding cofibre sequence
 $$ \Sigma^{l_1\rho}B/\uvb^{\ul'} \xrightarrow{\vb_1^{l_1}} B/\uvb^{\ul'} \to B/\uvb^{\ul} \to \Sigma^{l_1\rho+1}B/\uvb^{\ul'}.$$
 This induces a short exact sequence
 $$ 0 \to (\Mpi_{*\rho-5}B/\uvb^{\ul'})/\vb_1^{l_1} \to \Mpi_{*\rho-5}B/\uvb^{\ul} \to \{\Mpi_{*\rho-6}B/\uvb^{\ul'}\}_{\vb_1^{l_1}} \to 0.$$
 Here the last term denotes the sub Mackey functor of $\Mpi_{*\rho-6}B/\uvb^{\ul'}$ killed by $\vb_1^{l_1}$.  By Corollary \ref{Cor:QuotientBP} and Lemma \ref{lem:computation}, we see that $\Mpi_{*\rho-5}B/\uvb^{\ul} = 0$. 
 \end{proof}

As $B = BP\R$ is not known to have an $E_\infty$-structure, we have to work with $M\R$-linear maps instead, for which the following lemma is useful:

\begin{lemma}
The map 
\[\Z^B \simeq \Hom_{M\R}(M\R, \Z^B) \to \Hom_{M\R}(B, \Z^B)\]
is an equivalence.
\end{lemma}
\begin{proof}
Let $e\colon M\R \to M\R$ be the Quillen--Araki idempotent. Recall that 
\[B = \hcolim \left(M\R \xrightarrow{e} M\R \xrightarrow{e} \cdots\right).\]
Thus, 
\[\Z^B \simeq \holim \left(\cdots \xrightarrow{e^*} \Z^{M\R} \xrightarrow{e^*} \Z^{M\R}\right).\]
Hence,
\[
\Hom_{M\R}(B, \Z^B) \simeq \holim \left(\cdots \xrightarrow{e^*} \Hom_{M\R}(B,\Z^{M\R}) \xrightarrow{e^*} \Hom_{M\R}(B,\Z^{M\R})\right).
\]
As every $\Hom_{M\R}(B,\Z^{M\R})$ is equivalent to a holim over $\Hom_{M\R}(M\R, \Z^{M\R})\simeq \Z^{M\R}$, connected by $e^*$, we get that 
$\Hom_{M\R}(B,\Z^B)$ is the homotopy limit $\hlim_{\Z^-\times \Z^-}\Z^{M\R}$, where $\Z^-$ denotes the poset of negative numbers and all connecting maps are $e^*$. This is equivalent to the homotopy limit indexed over the diagonal, which in turn is equivalent to the homotopy limit indexed over a vertical.
\end{proof}

Recall that we want to show that $X = \Sigma^{2\rho-4}\Z^B$ is equivalent to $\kappa_{M\R}(\underline{\vb},B)$. The reason for the choice of suspension is essentially (as before) that $H\underline{\Z} \simeq \Sigma^{2\rho-4}H\underline{\Z}^*$.

\begin{prop}\label{Cor:QuotientX}
For a sequence $\ul = (l_1,l_2, \dots)$, we have an isomorphism
\[
 \underline{[\Sigma^{*\rho}B/\uvb^{\ul}, X]}^{C_2}_{M\R} \cong \Zu \tensor_{\Z} (\Z[\vb_1,\vb_2,\dots]/\uvb^{\ul})^*,
\] 
natural with respect to the maps $B/\uvb^{\ul} \to \Sigma^{-|\ul'-\ul|\rho}B/\uvb^{\ul'}$ in the defining homotopy colimit for $\kappa_{M\R}(\uvb; B)$ for $\ul' = (l_1',l_2',\dots)$ a sequence with $l_i' \geq l_i$ for all $i\geq 1$. 
\end{prop}
\begin{proof}
The last lemma implies that we also have
$$\Z^{B/\uvb^{\ul}} \simeq \Hom_{M\R}(B/\uvb^{\ul}, \Z^B)$$
as the functors $\Z^{?}$ and $\Hom_{M\R}(?, \Z^B)$ behave the same way with respect to cofibre sequences and (filtered) homotopy colimits. Then we just have to apply Lemma \ref{lem:AndersonQuot}. 
\end{proof}

\subsection{The theorem}
We first describe the homotopy groups of $X = \Sigma^{2\rho-4}\Z^B$ with $B=BP\R$ as before. 
By Lemma \ref{lem:AndersonQuot}, we get
$$\underline{\pi}^{C_2}_{*\rho}X \cong \Hom(\underline{\pi}_{(*+2)\rho-4}^{C_2}B, \Z) \cong \underline{\Z} \tensor_\Z \Z[\vb_1,\vb_2,\dots]^{*}.$$ 

Let $\underline{l}$ be a sequence with only finitely many nonzero entries. By Proposition \ref{Cor:QuotientX}, the element $(\underline{\vb}^{\underline{l}-\underline{1}})^*$ induces a corresponding $M\R$-linear map $\Sigma^{-|\underline{l}-\underline{1}|}B/\vb^{\underline{l}} \to X$, which is unique up to homotopy. By this uniqueness, these maps are also compatible for comparable $\underline{l}$. By Remark \ref{rmk:hocolim}, this induces a map
\[\kappa_{M\R}(\underline{\vb},B) = \hcolim_{\underline{l}}\left(\Sigma^{-|\underline{l}-\underline{1}|}B/\vb^{\underline{l}}\right) \;\xrightarrow{h}\, X,\]
where $\underline{l}$ ranges over all sequences where only finitely many $l_i$ are nonzero.

\begin{thm}\label{Thm:BPDuality}
This map $h\colon \kappa_{M\R}(\underline{\vb}; B) \to X$ is an equivalence of $\Ctwo$-spectra. 
\end{thm}

\begin{proof}
By Corollary \ref{Cor:crucial}, we get on $\underline{\pi}_{*\rho}$-level
\[\clim_{\underline{l}} \Sigma^{-|\underline{l}-\underline{1}|}\underline{\Z}[\vb_1,\vb_2,\dots]/(\vb_1^{l_1},\dots) \to \underline{\Z} \tensor_\Z \Z[\vb_1,\dots]^*,\]
which is an isomorphism. The odd underlying homotopy groups of both sides are zero. To apply Lemma \ref{lem:regrep}, it is left to show that $\pi^{\Ctwo}_{k\rho-1}\kappa_{M\R}(\underline{\vb}; B) = 0$ for all $k\in\Z$. Again by Corollary \ref{Cor:crucial}, it is even true that $\pi^{\Ctwo}_{k\rho-1}(B/\vb^{\underline{l}})$ is zero for all $k\in\Z$ and all sequences $\underline{l}$.
\end{proof}

\section{Duality for regular quotients}
The goal of this section is to prove our main result Theorem \ref{thm:main}:
\begin{thm}\label{Thm:QuotientDuality}
Let $(m_1,m_2,\dots)$ be a sequence of nonnegative integers with only finitely many entries bigger than $1$. Denote by $\underline{\vb}'$ the sequence of $\vb_i$ in $\pi^{\Ctwo}_{\rost}M\R$ such that $m_i = 0$ and by $m'$ the sum of all $(m_i-1)|\vb_i|$ for $m_i> 1$. Then there is an equivalence
$$\Z^{B/\underline{\vb}^{\underline{m}}} \simeq \Sigma^{-m'+4-2\rho}\kappa_{M\R}(\underline{\vb}'; B/\underline{\vb}^{\underline{m}}).$$
\end{thm}
Here and for the rest of the section we will implicitly localize everything at $2$ again. Before we prove the theorem, we need some preparation.

\begin{lemma}\label{Lem:AndersonQuotient}
Let $\underline{m} = (m_1,\dots)$ be a sequence of nonnegative integers with a finite number $n$ of nonzero entries. Then 
\[\Z^{B/\underline{\vb}^{\underline{m}}} \simeq \Sigma^{-|\underline{m}|-n}(\Z^B)/\underline{\vb}^{\underline{m}}.\]
\end{lemma}
\begin{proof}
Let $Y$ be an arbitrary ($\Ctwo$-)spectrum and $\Sigma^{|v|} Y \xrightarrow{v} Y \to Y/v$ be a cofibre sequence. Then we have an induced cofibre sequence
\[\Z^{Y/v} \to \Z^Y \xrightarrow{v} \Sigma^{-|v|} \Z^Y \to \Sigma \Z^{Y/v} \simeq \Sigma^{-|v|}(\Z^{Y})/v.\]
Thus, $\Z^{Y/v} \simeq \Sigma^{-|v|-1}(\Z^{Y})/v$. The claim follows by induction. 
\end{proof}

\begin{lemma}\label{Lem:Nilpotence}
The element $\vb_i^{3k}$ acts trivially on $B/\vb_i^k$ for every $i\geq 1, k\geq 1$.
\end{lemma}
\begin{proof}
By the commutativity of the diagram
\[\xymatrix{
\Sigma^{k|\vb_i|}B \ar[d]^{\vb_i^k}\ar[r]& \Sigma^{k|\vb_i|}B/\vb_i^k \ar[d]^{\vb_i^k} \ar[d]\\
B \ar[r] & B/\vb_i^k }
\]
we see that the composite $\Sigma^{k|\vb_i|} B \to \Sigma^{k|\vb_i|} B/\vb_i^k \xrightarrow{\vb_i^k} B/\vb_i^k$ is zero, so that the latter map factors over an $M\R$-linear map $\Sigma^{2k|\vb_i|+1}B \to B/\vb_i^k$. As $[\Sigma^{2k|\vb_i|+1}B, B/\vb_i^k]_{M\R}$ is a retract of $[\Sigma^{2k|\vb_i|+1}M\R, B/\vb_i^k]_{M\R} \cong \pi_{2k|\vb_i|+1}^{C_2}B/\vb_i^k$, we just have to show that $\vb_i^{2k}x = 0$ for every $x\in \pi_{2k|\vb_i|+1}B/\vb_i^k$. 

We have a short exact sequence
\[0 \to (\pi_\bigstar^{\Ctwo}B)/\vb_i^k \to \pi_\bigstar^{\Ctwo}(B/\vb_i^k)\to \left\{\pi^{\Ctwo}_{\bigstar -k|\vb_i|-1} B\right\}_{\vb_i^k} \to 0.\]
As $\vb_i^k x$ clearly maps to zero, it is the image of a $y\in (\pi_\bigstar^{\Ctwo}B)/\vb_i^k$. But $\vb_i^ky = 0$.
\end{proof}

\begin{lemma}\label{lem:smash}
 We have $$B/\vb_i^l \tensor_{M\R} B/\vb_j^m \simeq B/(\vb_i^l,\vb_j^m).$$
 Furthermore, there is an equivalence
 \[\hcolim_l \Sigma^{-(l-1)|\vb_i|}B/\vb_i^l\tensor_{M\R} B/\vb_i^m \simeq \Sigma^{|\vb_i| +1}B/\vb_i^m\]
 of $M\R$-modules if $m\geq 1$.
\end{lemma}
\begin{proof}
We have  $$B\tensor_{M\R} B \simeq \hcolim (B\xrightarrow{e} B \xrightarrow{e} \cdots) \simeq B,$$
where $e$ denotes again the Quillen--Araki idempotent, 
and thus also
$$B/\vb_i^l \tensor_{M\R} B/\vb_j^m \simeq B/(\vb_i^l,\vb_j^m).$$

Thus, the maps in the homotopy colimit in the lemma are induced by the following diagram of cofibre sequences:
\[\xymatrix{
\Sigma^{|\vb_i|}B/\vb_i^m \ar[r]^-{\vb_i^l}\ar[d]^{\mathrm{id}} & \Sigma^{-(l-1)|\vb_i|}B/\vb_i^m \ar[r]\ar[d]^{\vb_i} & \Sigma^{-(l-1)|\vb_i|}B/\vb_i^l\tensor_{M\R} B/\vb_i^m \ar[d] 
\\
\Sigma^{|\vb_i|}B/\vb_i^m \ar[r]^-{\vb_i^{l+1}} & \Sigma^{-l|\vb_i|}B/\vb_i^m \ar[r]& \Sigma^{-l|\vb_i|}B/\vb_i^{l+1}\tensor_{M\R} B/\vb_i^m 
}
\]
We can assume that the homotopy colimit only runs over $l\geq 3m$ so that by the last lemma the two cofibre sequences split and we get 
\[\Sigma^{-(l-1)|\vb_i|}B/\vb_i^l\tensor_{M\R} B/\vb_i^m \simeq \Sigma^{-(l-1)|\vb_i|}B/\vb_i^m \oplus \Sigma^{|\vb_i|+1}B/\vb_i^m.\]
The corresponding map 
\[\Sigma^{-(l-1)|\vb_i|}B/\vb_i^m \oplus \Sigma^{|\vb_i|+1}B/\vb_i^m \to \Sigma^{-l|\vb_i|}B/\vb_i^m \oplus \Sigma^{|\vb_i|+1}B/\vb_i^m\]
induces multiplication by $\vb_i$ on the first summand, the identity on the second plus possibly a map from the second summand to the first. 

Using this decomposition, it is easy to show that
\[\hcolim_{l} \Sigma^{-(l-1)|\vb_i|}B/\vb_i^l\tensor_{M\R} B/\vb_i^m\to \Sigma^{|\vb_i|+1}B/\vb_i^m\]
(defined by the projection on the second summand for $l\geq 3m$) is an equivalence. Indeed, on homotopy groups the map is clearly surjective. And if 
$$(x,y) \in \pi_\bigstar^{C_2}\Sigma^{-l|\vb_i|}B/\vb_i^m \oplus \pi_{\bigstar}^{C_2} \Sigma^{|\vb_i|+1}B/\vb_i^m $$
maps to $0 \in \pi_{\bigstar}^{C_2} \Sigma^{|\vb_i|+1}B/\vb_i^m$, then $y = 0$ and $(x,0)$ represents $0$ in the colimit because $\vb_i$ acts nilpotently. 
\end{proof}

\begin{proof}[of theorem]
As in the theorem, let $\underline{\vb}'$ be the sequence of $\vb_i$ such that $m_i = 0$ and also denote by $\underline{\vb}'' = (\vb_{i_1},\vb_{i_2},\dots)$ the sequence of $\vb_i$ such that $m_i \neq 0$. 

We begin with the case that $\underline{m}$ has only finitely many
nonzero entries (say $n$). By Lemma \ref{Lem:AndersonQuotient} we see that 
\[\Z^{B/\underline{\vb}^{\underline{m}}} \simeq \Sigma^{-|\underline{m}|-n}(\Z^B)/\underline{\vb}^{\underline{m}}.\]
Combining this with Theorem \ref{Thm:BPDuality}, we obtain
\begin{align*}
 \Z^{B/\underline{\vb}^{\underline{m}}} &\simeq \Sigma^{-|\underline{m}|-n+4-2\rho}\kappa_{M\R}(\underline{\vb}, B)/\underline{\vb}^{\underline{m}} \\
					 &\simeq \Sigma^{-|\underline{m}|-n+4-2\rho}\kappa_{M\R}(\underline{\vb}',\kappa_{M\R}(\underline{\vb}'', B/\underline{\vb}^{\underline{m}}))
\end{align*}
Thus, we have to show that $\kappa_{M\R}(\underline{\vb}'', B/\underline{\vb}^{\underline{m}}) \simeq \Sigma^{|\vb_{i_1}|+\cdots |\vb_{i_n}|+n}B/\underline{\vb}^{\underline{m}}$.

By Lemma \ref{lem:smash}, we have an equivalence
\[(B/\underline{\vb}^{\underline{m}})/(\vb_{i_1}^{l_{i_1}}, \dots, \vb_{i_n}^{l_{i_n}}) \simeq (B/\vb_1^{l_{i_1}}\tensor_{M\R} B/\vb_1^{m_{i_1}}) \tensor_{M\R}\dots\tensor_{M\R} (B/\vb_n^{l_{i_n}}\tensor_{M\R} B/\vb_n^{m_{i_n}}).\]
If we let now the homotopy colimit run over the sequences $(l_{i_1},\dots, l_{i_n})$, we can do it separately for each tensor factor. Hence, we obtain again by Lemma \ref{lem:smash} an equivalence
\[\kappa_{M\R}(\underline{\vb}'', B/\underline{\vb}^{\underline{m}}) \simeq \Sigma^{|\vb_{i_1}|+\cdots |\vb_{i_n}| +n} B/\underline{\vb}^{\underline{m}}.\]
Thus, we have shown the theorem in the case that $\underline{m}$ has only finitely many nonzero entries. 

We prove the case that $\underline{m}$ has possibly infinitely many nonzero entries by a colimit argument. Define $\underline{m}_{\leq k}$ to be the sequence obtained from $\underline{m}$ by setting $m_{k+1}, m_{k+2}, \dots$ to zero. Then $B/\underline{m} \simeq \hcolim_k B/\underline{m}_{\leq k}$ and thus $\Z^{B/\underline{m}} \simeq \hlim_k \Z^{B/\underline{m}_{\leq k}}$. Denote by $\underline{\vb}'_{\leq k}$ the sequence of $\vb_i$ such that $m_i = 0$ or $i>k$ and by $m'_k$ the quantity $|\underline{m}_{\leq k}-\underline{1}|$; note that $m'_k = m'$ for $k$ large. 

We have to show that the map
\[h\colon \Sigma^{-m'}\kappa_{M\R}(\underline{\vb}', B/\underline{\vb}^{\underline{m}}) \to \hlim_k \Sigma^{-m'_k}\kappa_{M\R}(\underline{\vb}'_{\leq k}, B/\underline{\vb}^{\underline{m}_{\leq k}})\]
is an equivalence. This map is defined as follows: We know that 
$$\kappa_{M\R}(\underline{\vb}', B/\underline{\vb}^{\underline{m}}) \simeq \hcolim_k \kappa_{M\R}(\underline{\vb}', B/\underline{\vb}^{\underline{m}_{\leq k}}).$$
Using this, we get a map induced from the maps $\kappa_{M\R}(\underline{\vb}', B/\underline{\vb}^{\underline{m}_{\leq k}}) \to \kappa_{M\R}(\underline{\vb}'_{\leq k}, B/\underline{\vb}^{\underline{m}_{\leq k}})$ for $k$ large.

By Corollary \ref{Cor:crucial}, we can describe what happens on $\pi_{*\rho}^{C_2}$: The left hand side has as $\Z$-basis monomials of the form $\underline{\vb}^{\underline{n}}$ with only finitely many $n_i$ nonzero, $n_i \leq 0$ and $n_i\geq -m_i+1$ if $m_i\neq 0$. Likewise, 
$$\pi_{*\rho}^{C_2}\left(\Sigma^{m'_k}\kappa_{M\R}(\underline{\vb}'_{\leq k}, B/\underline{\vb}^{\underline{m}_{\leq k}})\right)$$
has as $\Z$-basis monomials of the form $\uvb^{\underline{n}}$ with only finitely many $n_i$ nonzero, $n_i \leq 0$ and $n_i\geq -m_i+1$ if $m_i\neq 0$ and $i\leq k$. The maps in the homotopy limit induce the obvious inclusion maps. Thus, clearly the map 
$$\pi_{*\rho}^{C_2}\left(\Sigma^{m'}\kappa_{M\R}(\underline{\vb}', B/\underline{\vb}^{\underline{m}})\right) \to \lim_k \pi_{*\rho}^{C_2}\left(\Sigma^{m'_k}\kappa_{M\R}(\underline{\vb}'_{\leq k}, B/\underline{\vb}^{\underline{m}_{\leq k}})\right)$$
is an isomorphism. 

It remains to show that $\lim^1_k\pi_{*\rho+1}^{C_2}\left(\Sigma^{m'_k}\kappa_{M\R}(\underline{\vb}'_{\leq k}, B/\underline{\vb}^{\underline{m}_{\leq k}})\right)$ vanishes. By Corollary \ref{cor:rho+}, every term has as $\F_2$-basis monomials of the form $a\uvb^{\underline{n}}$ with only finitely many $n_i$ nonzero, $n_i \leq 0$ and $n_i\geq -m_i+1$ if $m_i\neq 0$ and $i\leq k$. The system becomes stationary in every degree, more precisely if $\ast > -2^{k+1}$. Thus, the $\lim^1$-term vanishes. A similar $\lim^1$-argument also shows that the odd underlying homotopy groups of $\hlim_k \Sigma^{-m'_k}\kappa_{M\R}(\underline{\vb}'_{\leq k}, B/\underline{\vb}^{\underline{m}_{\leq k}})$ vanish.

As the source of $h$ is strongly even by Corollary \ref{Cor:crucial} and by the arguments we just gave the morphism $h$ induces an isomorphism on $\underline{\pi}_{*\rho}^{C_2}$ and on (odd) underlying homotopy groups, Lemma \ref{lem:regrep} implies that $h$ is an equivalence. 
\end{proof}

\vspace{1cm} 
\part{Local cohomology computations}\label{part:LocalCohomology}

In Part 4, we will describe the local cohomology spectral sequence in
some detail, and use it to understand the structure of the
$\HZu$-cellularization of $\BPRn$. The calculation is not difficult,
but on the other hand it is quite hard to follow because it is made up
of a large number of easy calculations which interact a little, and
because one needs to find a helpful way to follow the $RO(\Ctwo )$-graded
calculations. 

In contrast the case of $\kR$ is simple enough to be explained fully without
further scaffolding, and it introduces many of the structures that we will want to
highlight. Since it may also be of wider interest than the general
case of $\BPRn$ we devote Section \ref{sec:kRlcss} to it before
returning to the general case in Section \ref{sec:BPRnlcss}. Section \ref{sec:tmfotlcss} will then give a more detailed account in the interesting case $n=2$. 

Let us also recall some notation used throughout this part. As in the
rest of the paper we work 2-locally, except when speaking about  $\kR$
or $tmf_1(3)$ when fewer primes need be inverted. We often
write $\delta = 1-\sigma \in RO(C_2)$. We also recall the duality
conventions from Section \ref{sec:DAb}; in particular, for an
$\F_2$-vector space $V^{\vee}$ equals the dual vector space
$\Hom_{\F_2}(V,\F_2)$ and for a torsionfree $\Z$-module $M$, we set
$M^*=\Hom(M, \Z)$.

If $R$ is a $C_2$-spectrum, we will use the notation $R^{C_2}_\rost$ for its $RO(C_2)$-graded homotopy groups. We will also write $R^{hC_2}_{\rost}=\pi_{\bigstar}^{C_2}(R^{(EC_2)_+})$ and similarly for geometric fixed points and the Tate construction. 

\section{The local cohomology spectral sequence for $\protect \kR$}
\label{sec:kRlcss}

This section focuses entirely on the classical case of $\kR$, where
there are  already a number of features of interest. This gives a
chance to introduce some of the structures we will use for the general
case. 

\subsection{The local cohomology spectral sequence}

Gorenstein duality for $\kR$ (Corollary \ref{cor:kRGorD}) 
has  interesting implications for the coefficient ring, both
computationally and structurally. 
Writing  $\rost$ for $RO(\Ctwo)$-grading as usual, the local cohomology spectral
sequence \cite[Section 3]{G-M95} takes the following form. 

\begin{prop} 
\label{prop:kRlcss}
 There is a spectral sequence of $\kR^{\Ctwo}_{\rost}$-modules
$$H^*_{(\vb)}(\kR^{\Ctwo}_{\rost}) \Rightarrow \Sigma^{-4+\sigma} \pi^{\Ctwo}_{\rost}(\Z^{\kR}).$$
The homotopy of the Anderson dual in an arbitrary degree $\alpha \in
RO(C_2)$ lies in an exact sequence
$$0\lra \Ext_{\Z}(\kR^{\Ctwo}_{-\alpha -1}, \Z)\lra 
\pi^{\Ctwo}_{\alpha}(\Z^{\kR}) \lra  \Hom_{\Z}(\kR^{\Ctwo}_{-\alpha}, \Z)
\lra 0. $$
Since local cohomology is entirely in cohomological degrees 0 and 1,
the spectral sequence collapses to a short exact sequence 
$$0\lra \Sigma^{-1} H^1_{(\vb)}(\kR^{\Ctwo}_{\rost}) \lra \Sigma^{-4+\sigma}
\pi^{\Ctwo}_{\rost}(\Z^{\kR}) \lra H^0_{(\vb)}(\kR^{\Ctwo}_{\rost}) \lra
0. $$
This sequence is not split, even as abelian groups.  
\end{prop}

One should not view Proposition \ref{prop:kRlcss} as an algebraic
formality: it embodies the fact that $\kR^{\Ctwo}_{\rost}$ is a very special
ring. To illustrate this, we recall the calculation of 
$\kR^{\Ctwo}_{\rost}$ in Subsection \ref{sec:kRgroups}. In Subsection
\ref{subsec:kRloccoh} we  calculate its local cohomology, and how
the Gorenstein duality isomorphism with the known homotopy of the Anderson 
dual works. 

\subsection{The ring $\protect \kR^{\Ctwo}_{\rost}$}\label{sec:kRgroups}

One may easily calculate  $\kR^{\Ctwo}_{\rost}$. This has already been done in \cite{B-G10}, but we sketch a slightly different method. We will first calculate $\kR^{h\Ctwo}_{\rost} $ and then use the Tate square \cite{GMTate}. 

In the homotopy fixed point spectral sequence
$$\Z[\vb, a, u^{\pm 1}]/2a \Rightarrow \kR^{hC_2}_{\rost}$$
all differentials are generated by $d_3(u)=\vb a^3$. Indeed, this differential is forced by $\eta^4 = 0$ and there is no room for further ones. 
It follows that $U=u^2$ is an infinite cycle, and so the
whole  ring is $U$-periodic:  
$$\kR^{h\Ctwo}_{\rost}=BB [U,U^{-1}], $$
where $BB$ is a certain `basic block'. This basic block is a sum
$$BB=BR\oplus (2u)\cdot  \Z [\vb]$$
as $BR$-modules, where 
$$ BR=\Z [\vb,a]/(2a, \vb a^3).$$

It is worth illustrating $BB$ in the plane (with $BB_{a+b\sigma}$
placed at the point $(a,b)$). The squares and circles represent copies of $\Z$, and
the dots represent copies of $\Ftwo$. The left hand vertical column
consists of 1 (at the origin, $(0,0)$) and the powers of $a$, but the feature to concentrate on
is the diagonal lines representing $\Z [\vb]$ submodules. These are
either copies of $\Z [\vb]$ or of $\Ftwo [\vb]$ or simply copies of $\Ftwo$. 

$$\begin{tikzpicture}[scale =1]
\draw[step=0.5, gray, very thin] (-3,-3) grid (3, 3);
\draw (-1.5,1.5 ) node[anchor=east, draw=orange]{\Large{BB}};
\foreach \y in {1,2,3,4,5,6}
\draw (0,-\y/2) node[anchor=east] {$a^{\y}$};
\foreach \y in {1,2,3,4,5,6}
\node at (0,-\y/2) [fill=red, inner sep=1pt, shape=circle, draw] {};

\foreach \y in {0, 1,2,3,4,5}
\node at (\y/2+1/2,\y/2) [fill=red, inner sep=1pt, shape=circle, draw]
{};
\foreach \y in {0, 1,2,3,4,5}
\node at (\y/2+1/2,\y/2-1/2) [fill=red, inner sep=1pt, shape=circle, draw] {};

\draw [->] (0,0)-- (3,3);
\node at (0,0)  [shape = rectangle, draw]{};
\draw (0,0) node[anchor=east]{1};
\foreach \y in {1,2,3,4,5,6}
\draw (\y/2,\y/2) node[anchor=east] {$\vb^{\y}$};
\foreach \y in {1,2,3,4,5,6}
\node at (\y/2,\y/2) [shape=rectangle, draw] {};

\draw[red] (0,-0.5)--(3,2.5);
\draw[red] (0,-1.0)--(3,2.0);

\draw [->](1,-1)-- (3,1);
\node at (1,-1) [shape=circle, draw] {};
\draw (1,-1) node[anchor=east]{$2u$};

\foreach \y in {1,2,3,4}
\node at (1+\y/2,-1+\y/2) [shape=circle, draw] {};
\end{tikzpicture}$$

Proceeding with the calculation, we may invert $a$ to find the
homotopy of the Tate spectrum $\kR^t=F(E(C_2)_+, \kR ) \sm S^{\infty
  \sigma}$: 
$$\kR^{t\Ctwo}_{\rost}=\Ftwo [a,a^{-1}][U,U^{-1}]. $$
One also sees that the homotopy of the geometric fixed points (the
equivariant homotopy of $\kR^{\Phi}=\kR \sm S^{\infty\sigma}$) is 
$$\kR^{\Phi \Ctwo}_{\rost}=\Ftwo [a,a^{-1}][U]  $$
using the following lemma:
\begin{lemma}
\label{lem:Phiconn}
 Let $X$ be a $\Ctwo$-spectrum which is non-equivariantly connective
 and  such that $X^{\Ctwo} \to X^{h\Ctwo}$ is a connective cover. Then $X^{\Phi \Ctwo} \to X^{t\Ctwo}$ is a connective cover as well.
\end{lemma}

\begin{proof}
 This follows from the diagram of long exact sequences
 \[\xymatrix{
  \pi_kX_{h\Ctwo} \ar[r]\ar[d] & \pi_kX^{\Ctwo} \ar[r]\ar[d]& \pi_kX^{\Phi \Ctwo} \ar[r]\ar[d] & \pi_{k-1}X_{h\Ctwo} \ar[r]\ar[d] & \pi_{k-1}X^{\Ctwo}\ar[d] \\
  \pi_kX_{h\Ctwo} \ar[r] & \pi_kX^{h\Ctwo} \ar[r]& \pi_kX^{t\Ctwo} \ar[r] & \pi_{k-1}X_{h\Ctwo} \ar[r] & \pi_{k-1}X^{h\Ctwo},
  }
 \]
 the fact that $X_{h\Ctwo}$ is connective and the $5$-lemma. 
\end{proof}

Now the Tate square 
$$\diagram 
\kR \rto \dto & \kR \sm \siftys \dto\\
\kR^{(E\Ctwo)_+} \rto  & \kR^{(E\Ctwo)_+} \sm \siftys 
\enddiagram$$
gives $\kR^{\Ctwo}_{\rost}$.

It is convenient to observe that the two
rows are of the form $M\lra M[1/a]$, so that the fibre is $\Gamma_a
M$. Since the two rows have equivalent fibres, we calculate the
homotopy of the second and obtain 
$$\kR^{\rost}_{hC_2}=NB [U, U^{-1}], $$
where $NB$ is quickly calculated as the $(a)$-local cohomology
$H^*_{(a)}(BB)$  (and named $NB$ for `negative block'). The element
$a$ acts vertically and we can immediately read  off the answer: the
tower $\Z [a]/(2a)$ gives some $H^1$, and the rest is $a$-power torsion:
$$NB=BB'\oplus \Sigma^{-\pp} \Ftwo [a]^{\vee}, $$
where $BB' \subset BB$ is the sub-$BR$-module
generated by $2, \vb, 2u$ (Informally,  we may say that $BB'$  omits from $BB$ all monomials $a^k$ for
$k\geq 1$ and the generator 1).  
Note that $NB$ is placed so that its element $2$ is in degree 0 for
ease of comparison to $BB$;  all occurrences of $NB$ in $\kR^{\Ctwo}_{\rost}$
involve nontrivial  suspensions. 


Again, it is helpful to display the negative block. This differs from
$BB$ in that the powers of $a$ have been deleted, and replaced by a
new left hand column $\Sigma^{-\pp}\Ftwo [a]^{\vee}$. The other new
feature is that the copy of $\Z [\vb]$ generated by $1$ has been
replaced by the kernel $(2,\vb)$  of $\Z
[\vb]\lra \Ftwo$, as indicated by the circle at the origin, labelled by its generator 2.

$$\begin{tikzpicture}[scale =1]
\draw[step=0.5, gray, very thin] (-3,-3) grid (3, 3);
\draw (-1.5,1.5 ) node[anchor=east, draw=orange]{\Large{NB}};
\foreach \y in {1,2,3,4,5,6}
\node at (-0.5,\y/2) [fill=red, inner sep=1pt, shape=circle, draw] {};
\draw[red] (-0.5,0.5)--(-0.5,3);

\foreach \y in {0, 1,2,3,4,5}
\node at (\y/2+1/2,\y/2) [fill=red, inner sep=1pt, shape=circle, draw]
{};
\foreach \y in {0, 1,2,3,4,5}
\node at (\y/2+1/2,\y/2-1/2) [fill=red, inner sep=1pt, shape=circle, draw] {};

\draw [->] (0,0)-- (3,3);
\node at (0,0)  [shape = circle, draw]{};
\draw (0,0) node[anchor=east]{2};
\foreach \y in {1,2,3,4,5,6}
\draw (\y/2,\y/2) node[anchor=east] {$\vb^{\y}$};
\foreach \y in {1,2,3,4,5,6}
\node at (\y/2,\y/2) [shape=rectangle, draw] {};
\draw[red] (0.5,0)--(3,2.5);
\node at (0.5,0) [fill=red, inner sep=1pt, shape=circle, draw] {};
\draw[red] (0.5,-0.5)--(3,2.0);
\node at (0.5,-0.5) [fill=red, inner sep=1pt, shape=circle, draw] {};

\draw [->](1,-1)-- (3,1);
\node at (1,-1) [shape=circle, draw] {};
\draw (1,-1) node[anchor=east]{$2u$};

\foreach \y in {1,2,3,4}
\node at (1+\y/2,-1+\y/2) [shape=circle, draw] {};

\foreach \y in {1,2,3,4,5,6}
\draw[red] (\y/2,\y/2) -- (\y/2, \y/2-1);
\end{tikzpicture}$$
\vspace{0.2cm}

The Tate square then lets us read off 
$$\kR^{\Ctwo}_{\rost}=\bigoplus_{k\leq -1} NB\cdot \{ U^{k}\} \oplus
\bigoplus_{k\geq 0} BB\cdot \{U^{k}\} =(U^{-1} \cdot NB [U^{-1}] )\oplus BB[U] $$
The $\Z [U]$ module structure is given by letting $U$ act in the
obvious way on the $NB$ and $BB$ parts, and by  the maps
$$NB \lra BB'\lra BB $$
in passage from the $U^{-1}$ factor of $NB$ to the $U^0$ factor of $BB$.

Perhaps it is helpful to note that with the exception of the towers
$U^{-k}\Sigma^{-\pp}\Ftwo [a]^{\vee}$, we have a subring of
$BB[U,U^{-1}]$, which consists of blocks $BB\cdot U^i$ for 
$i \geq 0$ and blocks $BB'\cdot U^i$ for $i<0$.

\subsection{Local cohomology}
\label{subsec:kRloccoh}
Recall that we are calculating local cohomology with respect to the principal
ideal $(\vb)$ so that we only need to consider $\kR^{\Ctwo}_{\rost}$ as a
$\Z [\vb]$-module. As such it is a sum of suspensions of the blocks
$BB$ and $NB$, so we just need to calculate the local cohomology of
these. 

More significantly, $\Z [\vb]$  is graded over multiples of the regular
representation, so local cohomology calculations may be performed on one diagonal
at a time (i.e., we fix $n$ and consider gradings $n+*\rho$). The only modules that occur are
$$\Z [\vb], \Ftwo [\vb], \Ftwo \mbox{ and the ideal } (2,
\vb)\subseteq \Z[\vb], $$
each of which has local cohomology that is very easily calculated.

\begin{lemma}
The local cohomology of the basic block $BB$ is as follows. 
$$H^0_{(\vb)}(BB)=a^3 \Ftwo [a]$$
$$H^1_{(\vb)}(BB)=\Sigma^{-\rho} \Z [\vb]^{*}\oplus \Sigma^{-\rho+2\delta} \Z [\vb]^{*}\oplus
\Sigma^{-\rho-\sigma} \Ftwo [\vb]^{\vee}\oplus \Sigma^{-\rho-2\sigma}
\Ftwo [\vb]^{\vee}. $$
\end{lemma}

\begin{proof}
The local cohomology is the cohomology of the complex
$$BB\lra BB[1/\vb]. $$ 
It is clear that 
$$BB[1/\vb]=\Z [\vb,\vb^{-1}]\oplus u\cdot \Z [\vb,\vb^{-1}]\oplus 
a\cdot \Ftwo  [\vb,\vb^{-1}]\oplus a^2\cdot \Ftwo  [\vb,\vb^{-1}]\qedhere$$
\end{proof}

Turning to $NB$, we recall that $NB=BB'\oplus \Sigma^{-\delta} \Ftwo
[a]^{\vee}$, and we have  a short exact sequence
$$0\lra BB'\lra BB\lra \Ftwo [a]\lra 0. $$ 
The local cohomology is thus easily deduced from that of $BB$.

\begin{lemma}
The local cohomology of the negative block $NB$ is as follows. 
$$H^0_{(\vb)}(NB)=\Sigma^{-\pp}\Ftwo [a]^{\vee}$$
$$H^1_{(\vb)}(NB)=\Sigma^{-\rho} \Z [\vb]^{*}\oplus \Ftwo  \oplus \Sigma^{-\rho+2\delta} \Z [\vb]^{*}\oplus
\Sigma^{-\sigma} \Ftwo [\vb]^{\vee}\oplus \Sigma^{-2\sigma} \Ftwo [\vb]^{\vee}$$
More properly, the $\Z [\vb]$-module structure of the sum of the first
two terms is 
$$\Sigma^{-\rho} \Z [\vb]^{*}\oplus \Ftwo  \cong \Z [\vb]^*/(2\cdot 1^*).$$
\end{lemma}

\begin{proof}
The local cohomology is the cohomology of the complex
$$NB\lra NB[1/\vb]. $$

It is clear that $NB[1/\vb]=BB[1/\vb]$, which makes the part coming
from the $2$-torsion clear. For the $\Z$-torsion free part, it is
helpful to  consider the exact sequence
$$0\lra (2, \vb) \lra \Z [\vb]\lra \Ftwo \lra 0$$
and then consider the long exact sequence in local cohomology.
\end{proof}

Immediately from the defining cofibre sequence $\Gamma_{\vb}\kR \lra
\kR \lra \kR [1/\vb]$ we see that there is a short exact sequence
$$0\lra H^1_{(\vb)}(\Sigma^{-1}\kR^{\Ctwo}_{\rost}) \lra
\pi^{\Ctwo}_{\rost}(\Gamma_{(\vb)}\kR)\lra
H^0_{(\vb)}(\kR^{\Ctwo}_{\rost}) \lra 0. $$
This gives $\pi^{\Ctwo}_{\rost}(\Gamma_{(\vb)}\kR)$ up to
extension. The Gorenstein duality isomorphism can be used to resolve the remaining extension
issues, and the answer is recorded in the proposition below. 

The diagram Figure \ref{fig:GBBkR} should help the reader interpret the statement and proof of
the calculation of the homotopy of $\Gamma_{(\vb)}\kR$. We have
omitted dots, circles and boxes except at the ends of diagonals or
where an additional generator is required. The vertical lines denote
multiplication by $a$ and the dashed vertical line is an exotic
multiplication by $a$ that is not visible on the level of local
cohomology. The green diamond does not denote a class, but marks the
point one has to reflect (non-torsion classes) at to see Anderson
duality. Torsion classes are shifted by $-1$ after reflection (i.e., shifted
one step horizontally to the left). 

\begin{center}
\begin{figure}\includegraphics{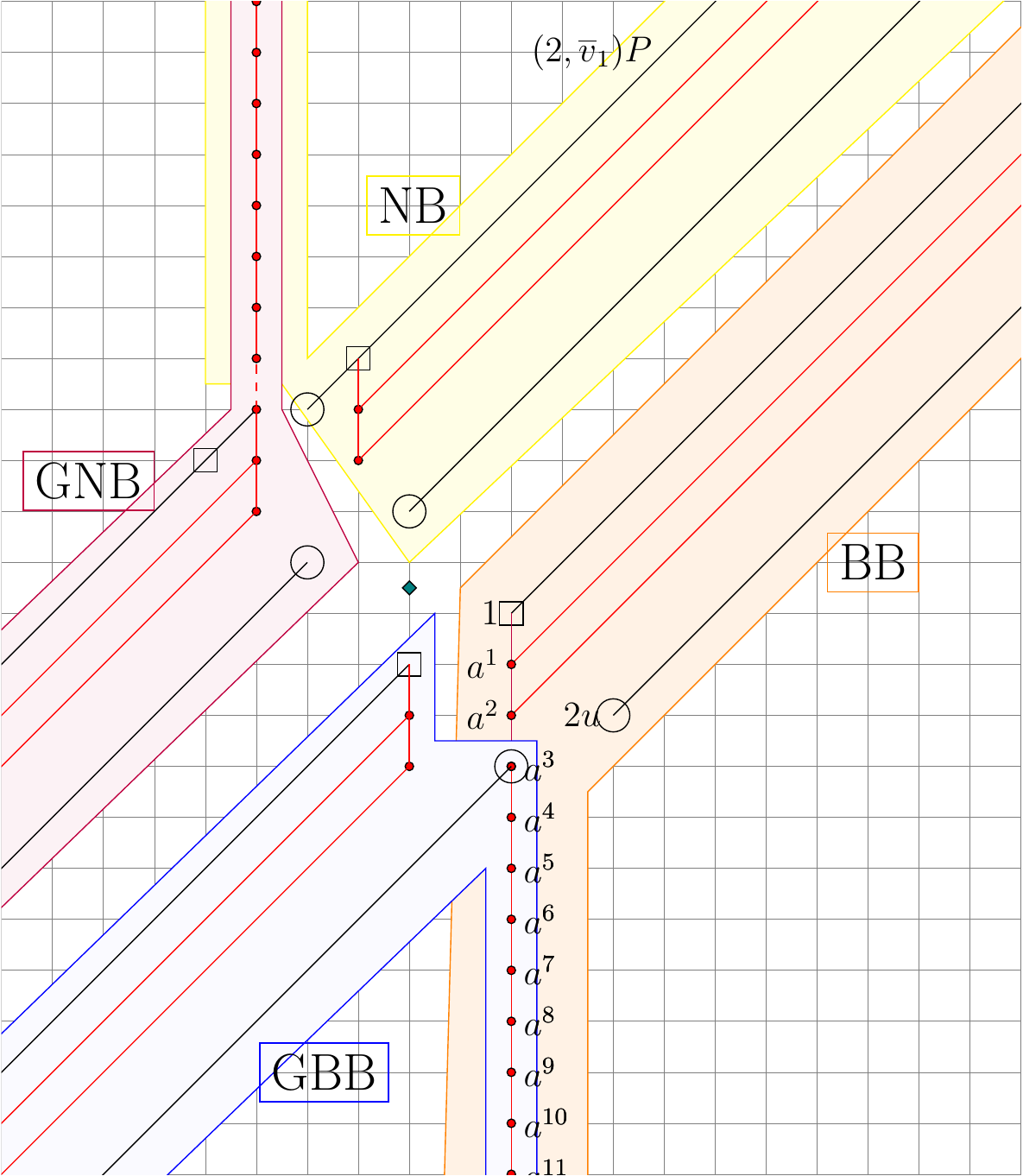}
\caption{Gorenstein duality for $\kR$ \label{fig:GBBkR}}
\end{figure}
\end{center}

\begin{prop}
The homotopy of the derived $\vb$-power torsion is given by 
$$\pi_{\rost}^{\Ctwo} (\Gamma_{(\vb)}\kR)\cong (U^{-1}\cdot GNB [U^{-1}]) \oplus GBB [U]$$
where $GBB$ and $GNB$ are based on the local cohomology of $BB$ and
$NB$ respectively, and described as follows.  We have 
$$GBB= \Sigma^{-2-\sigma} \left[ 
\Z [\vb]^{*}\oplus a\cdot \Ftwo [\vb]^{\vee}
\oplus  a^2\cdot \Ftwo [\vb]^{\vee}
\oplus u\cdot N \right] $$
where $N$ (with top in degree 0) is given by an exact sequence
$$0\lra \Z [\vb]^* \lra N \lra \Ftwo [a]\lra 0, $$
non-split in degree 0.

Similarly, 
$$GNB= \Sigma^{-1} 
\left[  \Z [\vb]^{*}/(2 \cdot (1^*) ) \oplus 
a\cdot \Ftwo [\vb]^{\vee}
\oplus a^2\cdot \Ftwo [\vb]^{\vee}
\oplus  \Sigma^{1-3\sigma}\Z [\vb]^*  \oplus \Sigma^{\sigma}\Ftwo
[a]^{\vee} \right] $$
where the action of $a$ is as suggested by the sum decomposition
except that multiplication by $a$ is non-trivial wherever possible
(i.e., when one dot is vertically above another, or where a box is
vertically above a dot). 
\end{prop}

\begin{proof}
We first note that the contributions from the different blocks do not
interact. Indeed, the only time that  different blocks give
contributions in the same degree come from the $\Ftwo [a]$ towers of
$BB$: one class in that degree is $\vb$-divisible (and not killed
by $\vb$) and the other class is annihilated by $\vb$. We may therefore 
consider the blocks entirely separately.

The block $GBB$ comes from the local cohomology of $BB$ and therefore
lives in a short exact sequence 
$$0\lra H^1_{(\vb)}(\Sigma^{-1}BB) \lra
GBB \lra H^0_{(\vb)}(BB) \lra 0$$

The block $GNB$ comes from the local cohomology of $NB$
and therefore lives in a short exact sequence 
$$0\lra H^1_{(\vb)}(\Sigma^{-1}NB) \lra
GNB \lra H^0_{(\vb)}(NB) \lra 0$$

Most questions about module structure over $BB[U]$ are resolved by
degree, but there are two which remain. These can be resolved Gorenstein duality \ref{cor:kRGorD} and the known module structure in $\Z^{\kR}$. 

In $GBB$, the additive
extension in $\pi^{\Ctwo}_{-3\sigma}$ is non-trivial: 
$$\pi_{-3\sigma}^{\Ctwo} (\Gamma_{(\vb)}\kR)\cong \Z. $$
Also the multiplication by $a$ 
$$\Ftwo \cong GNB_{-1+\sigma} \to GNB_{-1}\cong \Ftwo$$
is nonzero (where $GNB_{-1+\sigma}$ corresponds to $\pi^{C_2}_{-5+5\sigma}(\Gamma_{(\vb)}k\R)$ in the $U^{-1}$-shift). 
\end{proof}

\begin{remark}
 It is striking that the duality relates the top $BB$ to the bottom
$NB$ (i.e., Anderson duality takes the part of $\Gamma_{\vb}\kR$
coming from the local cohomology of $BB$ to $NB$), and it takes the
bottom $NB$ to the top $BB$ (i.e., Anderson duality takes the part of $\Gamma_{\vb}\kR$
coming from the local cohomology of $NB$ to $BB$). 
Indeed, as commented after Lemma \ref{lem:Phiconn}, since $NB=\Gamma_{(a)}BB$, we have
$$\Sigma^{2+\sigma}\Gamma_{(\vb)}BB\simeq (\Gamma_{(a)} BB)^*$$
and 
$$\Gamma_{(\vb, a)} BB\simeq  \Sigma^{-2-\sigma}  BB^*,  $$
with the second stating that $BB$ is Gorenstein of shift $-2-\sigma$
for the ideal $(a, \vb )$.

By extension, Anderson duality takes the part of $\Gamma_{\vb}\kR$
coming from the local cohomology of all copies of $BB$ to all copies
of $NB$ and vice versa. This might suggest separating $\kR$ into a
part with homotopy $BB[U]$, giving a cofibre sequence 
$$\langle BB[U]\rangle\lra \kR \lra \langle U^{-1}NB[U^{-1}]\rangle , $$
where the angle brackets refer to a spectrum with the indicated
homotopy.  However one may see that there is no $C_2$-spectrum with homotopy the Mackey functor corresponding to $BB[U]$
(considering the $b\sigma$ and $(b+1)\sigma$ rows one sees that the
non-equivariant  homotopy of the spectrum would be zero up to about 
degree $2b$; taking all rows together it would have to be
non-equivariantly contractible and hence $a$-periodic). Similarly,
there is no spectrum with homotopy $U^{-1}NB[U^{-1}]$, so these dualities are purely
algebraic. 
\end{remark}

\section{The local cohomology spectral sequence for $\protect \BPRn$}
\label{sec:BPRnlcss}

Gorenstein duality for $\BPRn$ (Example \ref{ex:BPRn}) 
has  interesting implications for the coefficient ring, both
computationally and structurally. 
Writing  $\rost$ for $RO(\Ctwo)$-grading as usual, the local cohomology spectral
sequence \cite[Section 3]{G-M95} takes the form described in the
following proposition.
 We now revert to our standard assumption of
working 2-locally,  so that $\Z$ means the 2-local integers.

\begin{prop} 
\label{prop:BPRnlcss}
There is a spectral sequence of $\BPRn^{\Ctwo}_{\rost}$-modules
$$H^*_{\Jb_n}(\BPRn^{\Ctwo}_{\rost}) \Rightarrow \Sigma^{-(D_n+n+2)-(D_n-2)\sigma} \pi^{\Ctwo}_{\rost}(\Z^{\BPRn})$$
for $\Jb_n = (\vb_1,\dots, \vb_n)$. 
The homotopy of the Anderson dual in an arbitrary degree $\alpha \in
RO(C_2)$ is easily calculated
$$0\lra \Ext_{\Z}(\BPRn^{\Ctwo}_{-\alpha -1}, \Z)\lra 
\pi^{\Ctwo}_{\alpha}\Z^{\BPRn} \lra  \Hom_{\Z}(\BPRn^{\Ctwo}_{-\alpha}, \Z) \lra 0. $$
For $n\geq 2$ the local cohomology spectral sequence has some non-trivial
differentials. 
\end{prop}

One should not view Proposition \ref{prop:BPRnlcss} as an algebraic
formality: it embodies the fact that $\BPRn^{\Ctwo}_{\rost}$ is a very special
ring. 

In the present section we will discuss the implications of this for
the coefficient ring for general $n$. The perspective is a bit distant
so the reader is encouraged to refer back to $\kR$ (i.e., the case
$n=1$) in Section \ref{sec:kRlcss} to anchor the generalities. 

However the case $n=1$ is too simple to show some of what happens, so 
we will also illustrate the case $\tmfot$ (i.e.,
the case $n=2$) in Section \ref{sec:tmfotlcss}.

\subsection{Reduction to diagonals}
For brevity we write $R_{\rost}=\BPRn_{\rost}^{\Ctwo}$. Because the ideal
$\Jb_n=(\vbn{1}, \ldots , \vbn{n})$
 is generated by elements whose degrees are a multiple of
$\rho$, we can
do $\Jb_n$-local cohomology calculations over the subring $R_{*\rho}$
of elements in degrees which are  multiples of $\rho$. 

Thus, for an $R_{\bigstar}$-module $M_{\bigstar}$ we have a direct sum decomposition
$$M_{\bigstar}=\bigoplus_{d} M_{d+*\rho}$$
as $R_{*\rho}$-modules, where we refer to the gradings $d+*\rho$ as the
{\em $d$-diagonal}. Hence, we also have
$$H^i_{\Jb_n}(M_{\bigstar})=\bigoplus_d H^i_{\Jb_n}(M_{d+*\rho}). $$
(We have abused notation by also writing $\Jbn$ for the ideal of
$R_{*\rho}$ generated by $\vbn{1}, \ldots , \vbn{n}$.)

\subsection{The general shape of $\BPRn^{\Ctwo}_{\rost}$}
By the description at the end of Section \ref{sec:BPRnC2}, we have an isomorphism
$$R_{\rost}= U^{-1}\cdot NB[U^{-1}] \oplus BB [U]$$
with $BB$ and $NB$ as described there. It is easy to see that $BB$ and $NB$ decompose as $R_{*\rho}$-modules into modules of a certain form we will describe now. We will implicitly $2$-localize everywhere.  

The modules $BB$ and $NB$ decompose into are
$$P=R_{*\rho}=\Z[\vbn{1}, \ldots, \vbn{n}] \mbox{ and } \Pb{s}=P/(\vbn{0}, \ldots,
\vbn{s})=\Ftwo [\vbn{s+1},\ldots , \vbn{n}] $$
for $s\geq 0$
and the ideals expressed by the exact sequences 
\begin{align*} 0\lra (2, \vbn{1}, \ldots , \vbn{t})\lra P \lra
  \Pb{t}\lra 0  \end{align*}
or 
\begin{align*}0\lra (\vbn{s+1}, \ldots , \vbn{t})\lra \Pb{s}\lra \Pb{t}\lra 0\end{align*}
with $s\geq 0$.

Their local cohomology is easily calculated. In the first two cases, the modules only have local cohomology in a
single degree
\begin{align*}H_{\Jb_n}^*(P)&=H_{\Jb_n}^n(P)=P^*(-D_n\rho) \\
H_{\Jb_n}^*(\Pb{s})&=H_{\Jb_n}^{n-s}(\Pb{s})=\Pb{s}^{\vee}((D_s-D_n)\rho). \end{align*}
The  top non-zero degree of $P^*$ is zero, so that  $1^* \in
P^*(-D_n\rho)$ is in degree $-D_n\rho = -|\vb_1|-\cdots -
|\vb_n|$. We alert the reader to the fact that star is used in two ways: occasionally in
$H^*$ to mean cohomological grading  and rather frequently here in
$P^*$ to mean the $\Z$-dual of $P$. 

Now we turn to the ideal $(\vbn{s+1}, \ldots , \vbn{t})$. If $t=s+1$ the ideal is principal
and $(\vbn{s+1})\cong \Pb{s}((s+1)\rho)$; thus we get a single local cohomology group
$$H^{n-s}_{\Jb_n} ((\vbn{s+1}) \Pb{s})=\Pb{s}^{\vee}((D_s-D_n+s+1)\rho)$$
as can be seen from the long exact sequence of local cohomology. 

Otherwise we get two local cohomology groups
$$H^{n-s}_{\Jb_n} ((\vbn{s+1}, \ldots , \vbn{t})\Pb{s})=\Pb{s}^{\vee}((D_n-D_s)\rho)
\mbox{ and } H^{n-t+1}_{\Jb_n} ((\vbn{s+1}, \ldots , \vbn{t})
\Pb{s})=\Pb{t}^{\vee}((D_n-D_t)\rho).$$

The case of $(2, \vbn{1}, \ldots , \vbn{t})$ is similar but with an extra case.  The case $t=0$ is easy since then
$(2)\cong P$ so the local cohomology is all in cohomological degree $n$ where it is
$P^*(-D_n\rho)$.  If $t=1$ we again get a single local cohomology group
$$H^{n}_{\Jb_n} ((2,\vbn{1}) P)=P^*(-D_n\rho)\oplus \Pb{1}^{\vee}((D_{1}-D_n)\rho).$$
Otherwise we get two local cohomology groups
$$H^{n}_{\Jb_n} ((2, \ldots , \vbn{t})P)=P^*(-D_n\rho)
\mbox{ and } H^{n-t+1}_{\Jb_n} ((2, \ldots , \vbn{t})P) =\Pb{t}^{\vee}((D_t-D_n)\rho).$$

\subsection{The special case $n=1$}
The best way to make the patterns apparent is to look at the simplest
cases.  In this section we begin with $\kR^{C_2}_{\bigstar}$ as treated in
Section \ref{sec:kRlcss} above, and we encourage the reader to relate
the calculations here to the diagrams in Section \ref{sec:kRlcss}. In that case, 
$$P=\kR^{C_2}_{*\rho}=\Z [\vbn{1}], \Pb{0}=\Ftwo [\vbn{1}] \mbox{ and } \Pb{1}=\Ftwo.$$

Displaying $BB$ by $d$-diagonal, we have

$$\begin{array}{c|cc}
&BB&(n=1)\\
\hline
d&1&u\\
\hline
0&P&\\
1&\Pb{0}&\\
2&\Pb{0}&\\
3&\Pb{1}&\\
4&\Pb{1}&(2)P\\
5&\Pb{1}&\\
6&\Pb{1}&\\
7&\Pb{1}&\\
8&\Pb{1}&\\
\end{array}$$
\vspace{0.2cm}

The position of the modules along the $d$-diagonal can be inferred
from the label at the top of the column. Thus the first column  has
generators in degree $-d\sigma$, and the second column similarly, but
 in the column of $u$ (namely the 2-column). Noting that $u$ is on the
 4-diagonal, the $d$th row has generators in $|u| -(d-4)\sigma = 2-(d-2)\sigma$. For example, along the 4-diagonal we have $a^4\Pb{1}
\oplus (2u)P$.

Taking local cohomology, and shifting $H^s_{\Jb_n}$ down by $s$ (as in the local cohomology spectral sequence), we have
$$\begin{array}{c|cc}
&H^*_{(\vb_1)}(BB)&(n=1)\\
\hline
d&1&u\\
\hline
-1&{\color{brown}P^*(-2\rho)}&\\
0&{\color{brown} \Pb{0}^{\vee}(-2\rho)}&\\
1&{\color{brown}\Pb{0}^{\vee}(-2\rho)}&\\
2&&\\
3&\Pb{1}&{\color{brown}P^*(-2\rho)}\\
4&\Pb{1}&\\
5&\Pb{1}&\\
6&\Pb{1}&\\
7&\Pb{1}&\\
8&\Pb{1}&\\
\end{array}$$
\vspace{0.2cm}

Here, we colored $H^1$-groups brown. Note that shifting down by $s$ both lowers $d$ by $s$ and adds a shift by $-s\rho$. For example, considering the 3-diagonal of this table, the $\Pb{1}$ comes directly
from the 3-diagonal of $BB$, whilst the  $P^*(-2\rho)$ comes from the
$(2)P$ on the 4-diagonal of $BB$; the local cohomology is
$P^*(-\rho)$, but its diagonal is shifted  by $-1$ since it is a first
local cohomology, and because it is by reference to the 2-column the
shift is $-\rho$. The top of this module is calculated by
reference to the column of $|u|$ (i.e., the 2-column), and has top in degree
 $2-(3-2)\sigma-2\rho=-3\sigma$.

We saw in Section \ref{sec:kRlcss} that the two modules on the
3-diagonal give a non-trivial additive extension (in degree
$-3\sigma$) after running the spectral sequence.

\subsection{The special case $n=2$}
\label{subsec:tmfotloccoh}
Continuing our effort to make patterns visible,  we consider
$tmf_1(3)^{C_2}_{\bigstar}$ in this subsection (i.e., the case $n=2$). With
$\Z$ denoting the integers with 3 inverted here, this has
$$P=tmf_1(3)^{C_2}_{*\rho}=\Z [\vbn{1}, \vbn{2}], \Pb{0}=\Ftwo [\vbn{1},
\vbn{2}], \Pb{1}=\Ftwo [\vbn{2}] \mbox{ and } \Pb{2}=\Ftwo.$$

Thus for $n=2$ we have
$$\begin{array}{c|cccc}
&BB&(n=2)&&\\
\hline
d&1&u&u^2&u^3\\
\hline
0&P&&&\\
1&\Pb{0}&&&\\
2&\Pb{0}&&&\\
3&\Pb{1}&&&\\
4&\Pb{1}&(2)&&\\
5&\Pb{1}&&&\\
6&\Pb{1}&&&\\
7&\Pb{2}&&&\\
8&\Pb{2}&&(2,\vbn{1})P&\\
9&\Pb{2}&&(\vbn{1})\Pb{0}&\\
10&\Pb{2}&&(\vbn{1})\Pb{0}&\\
11&\Pb{2}&&&\\
12&\Pb{2}&&&(2)\\
13&\Pb{2}&&&\\
\end{array}$$
Once again, the column labelled $u^i$ is the $2i$th column, and shifts
along the diagonal have as reference point where this column meets the
relevant diagonal. 

We take local cohomology, again remembering that $H^s_{\Jb_n}$ is
shifted down by $s$, which changes the diagonal by $s$. For example, on the 7-diagonal, $\Pb{2}$ comes from the 7-diagonal in
$BB$, whereas the $\Pb{0}^{\vee}(-5\rho)$ comes from the 2nd local
cohomology of the entry $(\vbn{1})\Pb{0}$ on the $9$-diagonal; the local
cohomology of $\Pb{0}$ is  $\Pb{0}^{\vee}(-4\rho)$, this is shifted by
a further $-2\rho$ from the change of diagonal, and $+\rho$ because of
the $\vbn{1}$.  
$$\begin{array}{c|cccc}
&H^*_{(\vb_1,\vb_2)}(BB)&(n=2)&&\\
\hline
d&1&u&u^2&u^3\\
\hline
-2&{\color{teal}P^*(-6\rho)}&&&\\
-1&{\color{teal}\Pb{0}^{\vee}(-6\rho)}&&&\\
0&{\color{teal}\Pb{0}^{\vee}(-6\rho)}&&&\\
1&&&&\\
2&{\color{brown}\Pb{1}^{\vee}(-4\rho)}&{\color{teal}P^*(-6\rho)}&&\\
3&{\color{brown}\Pb{1}^{\vee}(-4\rho)}&&&\\
4&{\color{brown}\Pb{1}^{\vee}(-4\rho)}&&&\\
5&{\color{brown}\Pb{1}^{\vee}(-4\rho)}&&&\\
6&&& {\color{brown}\Pb{1}^{\vee}(-5\rho)}\,\oplus\,  {\color{teal}P^*(-6\rho)}&\\
7&\Pb{2}&&{\color{teal}\Pb{0}^{\vee}(-5\rho)}&\\
8&\Pb{2}&&{\color{teal}\Pb{0}^{\vee}(-5\rho)}&\\
9&\Pb{2}&&&\\
10&\Pb{2}&&&{\color{teal}P^*(-6\rho)}\\
11&\Pb{2}&&&\\
12&\Pb{2}&&&\\
13&\Pb{2}&&&\\
\end{array}$$

We have colored again $H^1$-groups in brown and now also $H^2$-groups in teal. We will see below that there are non-trivial extensions on the 2- and 10-diagonals, and that
there are differentials in the local cohomology spectral sequence from
the 7-, 8- and 9-diagonals (differentials go from the $d$-diagonal to
the $(d-1)$-diagonal). 

\subsection{Moving from the basic  block $BB$ to the negative block $NB$}
Moving from $BB$ to $NB$ only affects the $0$ column, where in each
case $M$ is replaced by $\ker (M \lra \Ftwo)=(2)M$. In effect this
replaces $\Pb{n}$ by $0$. It also adds on a new $(-1)$-column 
of $\Pb{n}=\Ftwo$ going up from the $\sigma$ row. We resist the temptation
to display a table for $NB$ explicitly, but note that
$NB=\Gamma_{(a)}BB$ as for $\kR$. 

\subsection{Gorenstein duality}
\label{subsec:GorDBPRn}
With the above data in mind, we may consider the $d$-diagonal $BB_d$,
where the lowest value of $d$ is 0 and the highest is $N=4(2^n-1)$. If
we ignore  the difference
between $BB$ and $NB$ (which is at most $\Ftwo$ in any degree)  we find
approximately that $BB_d$ has a relationship to $BB_{N-d}$, namely 
something like an equality
$$H^n_{\Jb_n}(BB_d)^*=BB_{N-d}. $$
There are various ways in which this is inaccurate and needs to be modified. 
Firstly, if the local cohomology of $BB_d$ is entirely in cohomological 
degree $n-\eps $ with $\eps \neq 0$, there will be a shift of $\eps$
(if it is in several degrees there is a further
complication). Secondly, Anderson duality introduces a shift of 1 diagonal if
applied to torsion modules. Thirdly, we have seen that there may be
extensions between these local cohomology groups, sometimes removing
$\Z$-torsion. Finally, there may be differentials. 

In fact all of these effects are `small' in the sense that the growth
rate along a diagonal is bounded by a polynomial of degree $n-1$. 
Encouraged by this, if we ignore all of these effects we see 
that $BB$ is a Gorenstein module in the sense that the reverse-graded 
version is equivalent to 
the dual of its local cohomology.
$$H^n_{\Jb_n}(BB)^*=\mathrm{rev}(BB). $$ 

This is rather as if there is a cofibre sequence 
$$S\lra \BPRn\lra Q$$
with $S$ Gorenstein and $Q$ a Poincar\'e duality algebra of formal dimension
$N=2(1-\sigma)(2^n-1)$.

\section{The local cohmology spectral sequence for $\tmfot$}
\label{sec:tmfotlcss}
We examine the local cohomology spectral sequence and Gorenstein duality
in more detail for $\tmfot$. Actually, our calculations are equally valid for all forms of $BP\R\langle 2\rangle$, but we prefer the more evocative name $\tmfot$ of the most prominent example. More of the general features are visible for $\tmfot$
than for $\kR$.

 As usual we will implicitly localize everywhere at $2$ (although for  $\tmfot$ itself it would actually suffice to just invert $3$).

\subsection{The local cohomology spectral sequence}
We make explicit the  implications for the coefficient ring, both
computationally and structurally. Writing $\rost$  for
$RO(\Ctwo)$-grading as usual, the spectral sequence takes the
following form. 

\begin{prop} 
\label{prop:tmfotlcss}
There is a spectral sequence of $\tmfot^{\Ctwo}_{\rost}$-modules
$$H^*_{\Jb_n}(\tmfot^{\Ctwo}_{\rost}) \Rightarrow \Sigma^{-8-2\sigma} \pi^{\Ctwo}_{\rost}(\Z^{\tmfot}).$$
The homotopy of the Anderson dual is easily calculated
$$0\lra \Ext_{\Z}(\tmfot^{\Ctwo}_{-\alpha -1}, \Z)\lra 
\pi^{\Ctwo}_{\alpha}\Z^{\tmfot} \lra  \Hom_{\Z}(\tmfot^{\Ctwo}_{-\alpha}, \Z) \lra 0. $$
The local cohomology spectral sequence has some non-trivial differentials. 
\end{prop}

\subsection{The ring $\protect \tmfot^{\Ctwo}_{\rost}$}\label{sec:tmfgroups}

The ring  $\tmfot^{\Ctwo}_{\rost}$ is approximately calculated in \cite{HM} and is more precisely desribed as 
$$BB[U] \oplus U^{-1}NB[U^{-1}]$$
as at the end of Section \ref{sec:BPRnC2} with $n=2$. We already tabulated $BB$ in Section \ref{subsec:tmfotloccoh}, but we want also want to display a bigger chart of $\pi_{\rost}^{C_2}tmf_1(3)$ as Figure \ref{fig:tmf13} to give the reader a feeling of how the blocks piece together.  

A black diagonal line means a copy of $P$ when it starts in a box, a
copy of $(2)P$ when it starts in a small circle, a copy of
$(2,\vb_1)P$ when it starts in a dot and a copy of $(2,\vb_1,\vb_2)$
when it starts in a big circle. A red diagonal line means a copy of
$\overline{P}_0$ and a green diagonal line a copy of $\overline{P}_1$. A red dot is a copy of $\F_2 = \overline{P}_2$.

\begin{center}
\begin{figure}\includegraphics{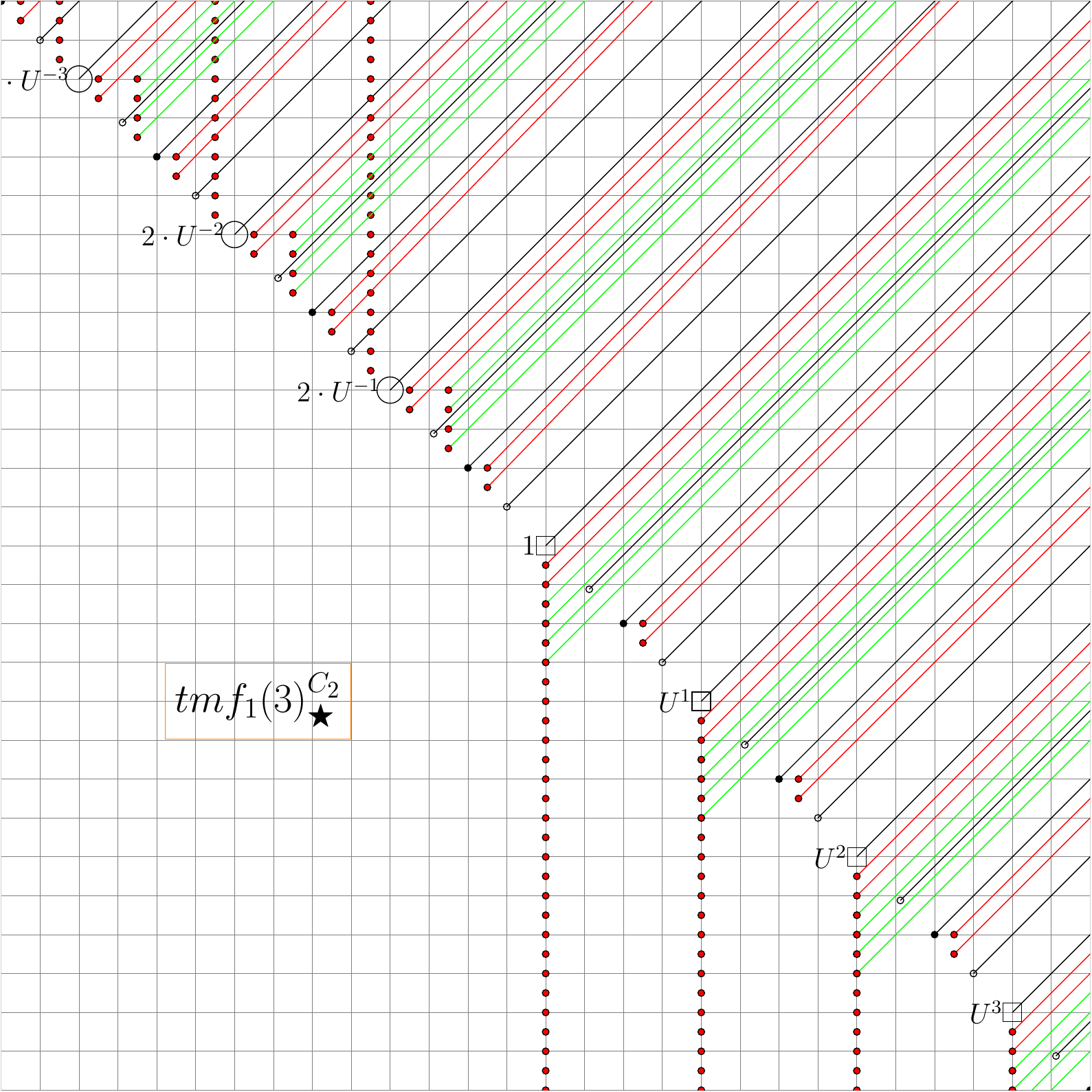}
\caption{The homotopy of $tmf_1(3)$ \label{fig:tmf13}}
\end{figure}
\end{center}

\subsection{Local cohomology}
We are calculating local cohomology with respect to the 
ideal $\Jb_2=(\vob, \vtb)$ so that we only need to consider $\tmfot^{\Ctwo}_{\rost}$ as a
$\Z [\vob, \vtb]$-module. As such it is a sum of suspensions of the blocks
$BB$ and $NB$, so we just need to calculate the local cohomology of
these. This was described in Section \ref{sec:BPRnlcss} above. Here we will
simply describe the extensions and the behaviour of the local
cohomology spectral sequence. 

The basis of this discussion are the tables of $BB$ and $GBB$ from Subsection
\ref{subsec:tmfotloccoh} together with the analogues for $NB$ and
$GNB$. Although these are organized by diagonal, Figure \ref{fig:GBBtmf13}
displaying $BB, GBB, U^{-1} NB$ and $U^{-1}GNB$ may help visualize the 
way the modules are distributed along each diagonal. The vertical lines denote
multiplication by $a$ and the dashed vertical line is an exotic
multiplication by $a$ that is not visible on the level of local
cohomology. The green diamond does not denote a class, but marks the
point one has to reflect (non-torsion classes) at to see Anderson
duality. Torsion classes are shifted after reflection by $-1$ (i.e.,
one step horizontally to the left). 

\begin{center}
\begin{figure}\includegraphics{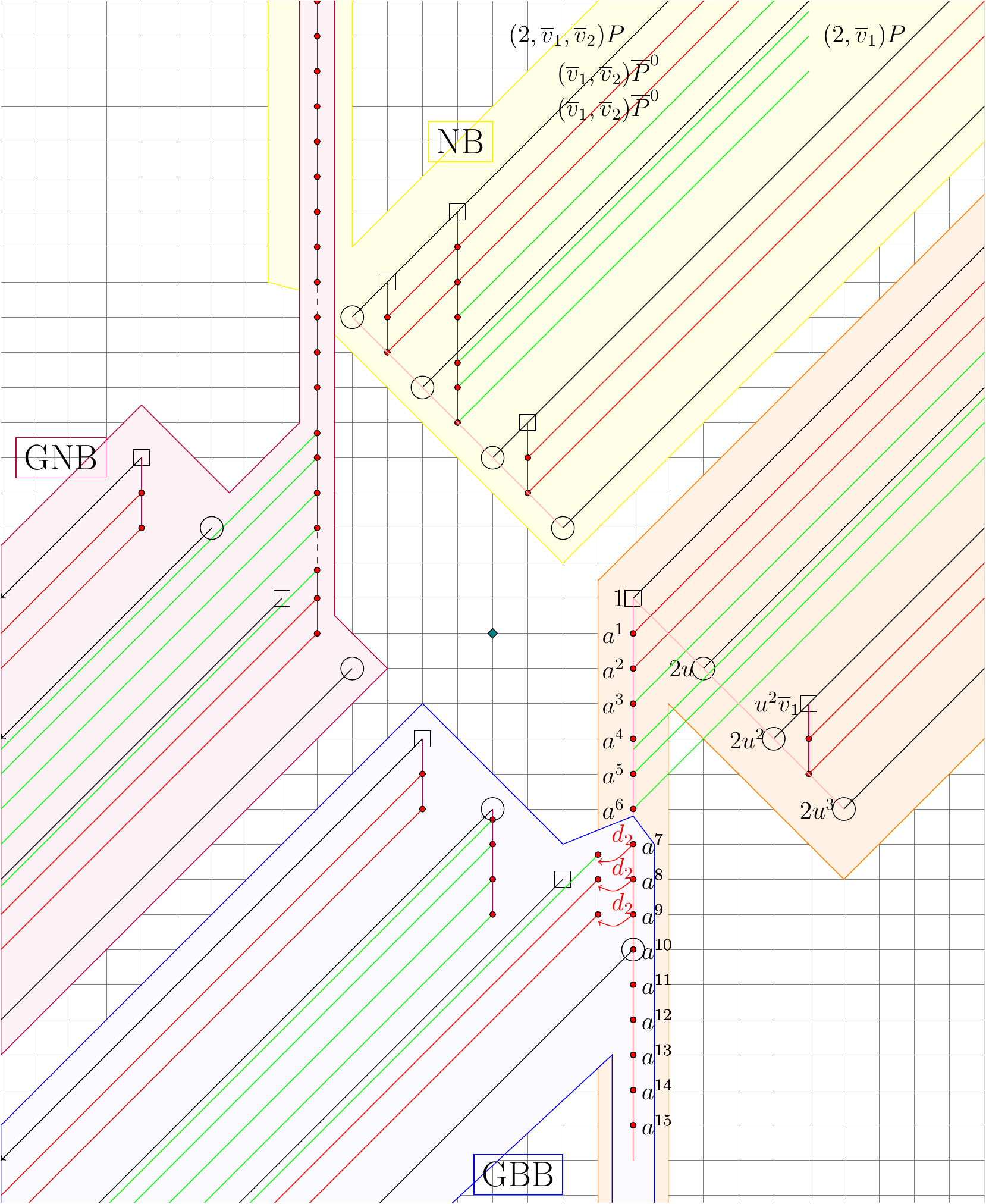}
\caption{Gorenstein duality for $tmf_1(3)$ \label{fig:GBBtmf13}}
\end{figure}
\end{center}

The strategy is to take the known subquotients from the local
cohomology  calculation, and resolve the extension problems using Gorenstein
duality.

\begin{prop}
We have an isomorphism 
$$\pi_{\rost}^{C_2}\Gamma_{\Jb_2}tmf_1(3) \cong GBB[U]\oplus U^{-1}GNB[U^{-1}],$$
where $GBB$ and $GNB$ are described in the following. We will simultaneously describe what differentials and extensions in the local cohomology spectral sequence caused the passage from $H^*_{\Jb_2}(BB)$ and $H^*_{\Jb_2}(NB)$ to $GBB$ and $GNB$ respectively. 

(i) The $\Z [\vob, \vtb]$-modules along the diagonals in $GBB$ are  as
follows. 
$$\begin{array}{c|cl}
&GBB&(n=2)\\
\hline
i&\mbox{Module}&\mbox{Top degree}\\
\hline
-2&P^*&-6-4\sigma\\
-1&\Pb{0}^{\vee}&-6-5\sigma\\
0&\Pb{0}^{\vee}&-6-6\sigma\\
1&0&\\
2&[(2,\vob)P]^*&-4-6\sigma\\
3&\Pb{1}^{\vee}&-4-7\sigma\\
4&\Pb{1}^{\vee}&-4-8\sigma\\
5&\Pb{1}^{\vee}&-4-9\sigma\\
6&[(2,\vob)P]^*&-2-8\sigma\\
7&(\vob, \vtb)\Pb{0}&-2-9\sigma\\
8&(\vob, \vtb)\Pb{0}&-2-10\sigma\\
9&0&\\
10&[(2, \vob, \vtb)P]^*&0-10\sigma\\
10+k\geq 11&\Ftwo&0-(10+k)\sigma\\
\end{array}$$
There are three non-trivial differentials
$$d_2: H^0_{\Jb_2}(BB)\lra  H^2_{\Jb_2}(BB)$$
from the groups at $-7\sigma, -8\sigma, -9\sigma$ to the groups at  
$-7\sigma-1, -8\sigma -1, -9\sigma-1$, which have affected the values
on the 6-, 7-, 8- and 9-diagonals in the table. 

The extensions 
$$0\lra P^* \lra [(2,\vob)P]^* \lra \Ftwo [\vb_2]^{\vee}\lra 0$$
on the 2-diagonal and the 6-diagonal are Anderson dual to the defining short exact sequence
$$0\lra (2,\vob)P\lra P \lra \Ftwo [\vb_2]\lra 0$$
in the following sense: The Anderson dual of the latter exact sequence is a triangle
$$\Ftwo[\vb_2]^* \to P^*\to [(2,\vob)P]^* \to \Sigma\Ftwo[\vb_2]^* \cong \Ftwo[\vb_2]^{\vee},$$
which induces (on homology) the extensions above. 
The extension 
$$0\lra P^* \lra [(2,\vob, \vtb )P]^* \lra \Ftwo \lra 0$$
on the 10-diagonal is Anderson dual to  the short exact sequence
$$0\lra (2,\vob, \vtb)P\lra P \lra \Ftwo \lra 0.$$

(ii) The $\Z [\vob, \vtb]$-modules along the diagonals in $GNB$ are  as
follows (take the direct sum of the two entries for the $(-2)$-,
$(-1)$-, $0$- $1$- and $2$-diagonals) 
$$\begin{array}{c|cl}
&GNB&(n=2)\\
\hline
i&\mbox{Module}&\mbox{Top degree}\\
\hline
-k\leq -3&\Ftwo&-1-k\sigma\\
-2&P^*, \Ftwo &-6-4\sigma, -1+\sigma\\
-1&\Pb{0}^{\vee}, \Ftwo&-6-5\sigma, -1+0\sigma\\
0&\Pb{0}^{\vee}, \Ftwo &-6-6\sigma, -1-\sigma\\
1&\Ftwo &-1-2\sigma\\
2&P^*, \Pb{1}^{\vee}&-4-6\sigma, -1-3\sigma\\
3&\Pb{1}^{\vee}&-1-4\sigma\\
4&\Pb{1}^{\vee}&-1-5\sigma\\
5&\Pb{1}^{\vee}&-1-6\sigma\\
6&[(2,\vob)P]^*&-1-7\sigma\\
7&\Pb{0}^{\vee}&-1-8\sigma\\
8&\Pb{0}^{\vee}&-1-9\sigma\\
9&0&\\
10&P^*&0-10\sigma\\
\end{array}$$

The extension
$$0\lra P^* \lra [(2,\vob)P]^* \lra \Ftwo [v_2]^{\vee}\lra 0$$
on  the 6-diagonal is Anderson dual to the short exact sequence
$$0\lra (2,\vob)P\lra P \lra \Ftwo [v_2]\lra 0.$$
\end{prop}

\begin{proof}
We first note that the contributions from the different blocks do not
interact. Indeed, the only time that  different blocks give
contributions in the same degree comes from the $\Ftwo [a]$ towers of
$BB$, and one class in that degree is divisible by $\vob$ or $\vtb$  and not killed
by both $\vob$ and $\vtb$. We may therefore 
consider the blocks entirely separately.
     
The block $GBB$ comes from the local cohomology of $BB$ in the sense
that there is a spectral sequence
$$H^*_{\Jb_2}(BB)\Rightarrow GBB .$$
Thus there is a filtration 
$$GBB=GBB^0\supseteq GBB^1\supseteq GBB^2\supseteq GBB^3=0$$
with 
$$0\lra GBB^0/GBB^1 \lra H^0_{\Jb_2}(BB) \stackrel{d_2}\lra
\Sigma^{-1}H^2_{\Jb_2}(BB) \lra \Sigma^1 GBB^2\lra  0$$
and 
$$GBB^1/GBB^2\cong \Sigma^{-1} H^1_{\Jb_2}(BB). $$ 

The block $GNB$ comes from the local cohomology of $NB$ in a precisely
analogous way.

Most questions about module structure over $BB[U]$ are resolved by
degree. The remaining issues are resolved by
using Gorenstein duality.

Referring to the table for $H^*_{\Jb_2}(BB)$  in Subsection
\ref{subsec:tmfotloccoh},  the first potential extension is on the
2-diagonal. Using Gorenstein duality to compare with  $NB_{\delta=8}$ 
we see  that the actual  extension on the 2-diagonal of $GBB$ is 
$$0\lra P^*\lra [(2,\vb_1)P]^*\lra \Pb{1}^{\vee} \lra
0,  $$
where we have shifted the modules so they all have top degree 0. 
There is an additive extension on the 10-diagonal by
reference to the Anderson dual. 
Finally the three non-zero $d_2$ differentials from $-1-k\sigma$ for
$k=7,8$ and $9$ are necessary for connectivity (this removes the 
need to discuss the possible extensions on the 7- and 8-diagonals).

The situation is rather similar for $GNB$. We will not explicitly
display $NB$ since the only effect (apart from the addition of 
$\Ftwo [a]^{\vee}$) is on the first column, where a
module is replaced by the kernel of a surjection to $\Ftwo$. It is
perhaps worth displaying $H^2_{\Jb_2}(NB)$, where we leave out the big $\Ftwo[a]^{\vee}$-tower in $H^0_{\Jb_2}NB$. We will color again $H^1$-groups in brown and $H^2$-groups in teal.
$$\begin{array}{c|cccc}
&H^*_{\Jb_2}(NB)&(n=2)&&\\
\hline
i&1&u^2&u^4&u^6\\
\hline
-2&{\color{teal}P^*(-6\rho)}&&&\\
-1&{\color{teal}\Pb{0}^{\vee}(-6\rho)}\oplus {\color{brown}\Pb{2}}&&&\\
0&{\color{teal}\Pb{0}^{\vee}(-6\rho)}\oplus {\color{brown}\Pb{2}}&&&\\
1&{\color{brown}\Pb{2}}&&&\\
2&{\color{brown}\Pb{1}^{\vee}(-4\rho)}&{\color{teal}P^*(-6\rho)}&&\\
3&{\color{brown}\Pb{1}^{\vee}(-4\rho)}&&&\\
4&{\color{brown}\Pb{1}^{\vee}(-4\rho)}&&&\\
5&{\color{brown}\Pb{1}^{\vee}(-4\rho)}&&&\\
6&&& {\color{brown}\Pb{1}^{\vee}(-5\rho)}\oplus  {\color{teal}P^*(-6\rho)}&\\
7&&&{\color{teal}\Pb{0}^{\vee}(-5\rho)}&\\
8&&&{\color{teal}\Pb{0}^{\vee}(-5\rho)}&\\
9&&&&\\
10&&&&{\color{teal}P^*(-6\rho)}\\
11&&&&\\
12&&&&\\
13&&&&\\
\end{array}$$
 In this case all extensions are split, except for the one on the
6-diagonal and there are no differentials. The $a$ multiplications
in the $\Ftwo [a]^{\vee}$ tower are clear from Gorenstein duality and
the $a$-tower $\Ftwo [a]$ in $BB$.
\end{proof}

\begin{remark}
(i) Summarizing the way a diagonal $BB_{\delta}$ contributes to $NB$ as in 
$$H^*_{\Jb_2}(BB_{\delta})^*\sim NB_{\delta'}$$ 
as sketched in Subsection \ref{subsec:GorDBPRn}. We have 

$$\begin{array}{|cc||cc|}
\delta &\delta' s.t. H^*_{\Jb_2}(BB_{\delta})^*\sim
NB_{\delta'}&\delta &\delta' s.t. H^*_{\Jb_2}(NB_{\delta})^*\sim BB_{\delta'}\\
\hline
0&12&0&12\\
1&10&1&10\\
2&9&2&9\\
3&8&3&8\\
4&8,6&4&8,6\\
5&5&5&5\\
6&4&6&4\\
7&2&7&.\\
8&4,3&8&4\\
9&2&9&2\\
10&1,0&10&1\\
11&0&11&.\\
12&0&12&0\\
\hline
\end{array}$$

Because most of the modules are $2$-torsion the most common pairing is
between $\delta$ and $11-\delta$ rather than between 
$\delta$ and $12-\delta$ as happens for the main $U$-power diagonals. 

(ii) We also note as before that since $NB=\Gamma_{(a)}BB$, we have
$$\Sigma^{6+4\sigma}\Gamma_{(\vb_1, \vb_2)}BB\sim (\Gamma_{(a)} BB)^*$$
(where we have written $\sim$ rather than $\simeq$  in recognition of
the differentials) and 
$$\Sigma^{6+4\sigma}\Gamma_{(\vb_1, \vb_2, a)} BB\simeq  BB^*,  $$
with the second stating that $BB$ is Gorenstein of shift $-6-4\sigma$
for the ideal $(\vb_1, \vb_2, a )$. 
\end{remark}

\vspace{1cm}
\addtocontents{toc}{\vspace{\normalbaselineskip}}
\appendix  
\section{The computation of $\pi_\bigstar^{\Ctwo}BP\mathbb{R}$}\label{Appendix}
Our main goal in this appendix is to compute the homotopy fixed point spectral sequence for $BP\R$ and hence for $M\R$. All the results in this appendix and the essential idea of the argument for Proposition \ref{prop:Differentials} are contained in \cite{HK} (see especially Formula 4.16). We just rearranged their arguments and added some details. Our argument for the multiplicative extensions might be considered new though. We have strived for elementary and short proofs though they retain some computational complexity. We hope this is helpful for the reader to understand this crucial computation. Note that even before Hu and Kriz, the computation of $\pi_\bigstar^{\Ctwo}BP\R$ was announced in \cite{A-M}.

We will work throughout $2$-locally. As before, we denote by $\rho$ the regular real $C_2$-representation and by $\sigma$ the sign representation. We need a few facts, first proven by Araki:

\begin{enumerate}
\item If $E$ is a Real-oriented spectrum, then $E^\bigstar_{\Ctwo}(\CP^\infty) \cong E^\bigstar_{\Ctwo}\llbracket u \rrbracket$ with $|u| = -\rho$ and $E^\bigstar_{\Ctwo}(\CP^\infty\times \CP^\infty) \cong E^\bigstar_{\Ctwo}\llbracket 1\tensor u, u\tensor 1 \rrbracket$. This induces a formal group law on $\pi^{\Ctwo}_{*\rho}E$ and the forgetful map $\pi^{\Ctwo}_{*\rho}E \to \pi_{2*}^eE$ maps it to the usual formal group law from the complex orientation of $E$. \cite[Theorem 2.10]{HK}
\item Thus, we get a ring map $\pi_{2*}^e MU \to \pi_{*\rho}^{\Ctwo}M\R$ from the Lazard ring so that $\pi_{2*}^e MU$ is a retract of $\pi_{*\rho}^{\Ctwo}M\R$. For every class in $x\in \pi_{2*}MU$, we have thus a corresponding class $\overline{x}\in\pi_{*\rho}^{\Ctwo}M\R$. \cite[Proposition 2.27]{HK}
\item There is a splitting $M\R_{(2)} \simeq \bigoplus_{m_i}\Sigma^{m_i\rho}BP\R$, where the underlying spectrum of $BP\R$ agrees with $BP$. This splitting corresponds on geometric fixed points to the splitting $MO \simeq \bigoplus_{m_i}\Sigma^{m_i}H\F_2$. \cite[Theorem 2.33]{HK}
\end{enumerate}

Define $a\colon S^0\to S^\sigma$ as before to be the inclusion of the points $0$ and $\infty$; we will denote the image of $a$ in $\pi_\bigstar M\R$ and $\pi_\bigstar BP\R$ by the same symbol. The class $a$ has degree $-\sigma = 1-\rho$. 

\begin{prop}\label{prop:vbn}
We have $a^{2^{n+1}-1}\vb_n = 0$ in $\pi_\bigstar^{\Ctwo}M\R$.
\end{prop}
\begin{proof}
We have a fibre sequence 
$$(E\Ctwo)_+ \tensor M\R \to M\R \to \tilde{E}\Ctwo \tensor M\R.$$
First, we claim that the image of $\vb_n$ under $M\R \to \tilde{E}\Ctwo \tensor M\R$ is zero. Indeed, as $a$ is invertible on $\tilde{E}\Ctwo \tensor M\R$, the formal group law on $\pi_{*\rho}^{\Ctwo}(\tilde{E}\Ctwo \tensor M\R)$ agrees with that on $\pi_*^{\Ctwo}(\tilde{E}\Ctwo \tensor M\R) = \pi_*MO$, which is additive. Therefore, the map 
$$MU_{2*} \to \pi_{*\rho}^{C_2}M\R \to \pi_{*\rho}^{\Ctwo} \tilde{E}\Ctwo\tensor M\R$$
sends all $v_n$ to zero. Thus, $\vb_n$ and hence also $a^{2^{n+1}-1}\vb_n$ are in the image of the map 
$$(E\Ctwo)_+ \tensor M\R \to M\R.$$

Observe that 
$$|a^{2^{n+1}-1}\vb_n| = -(2^{n+1}-1)\sigma + (2^n-1)(1+\sigma) = 2^n-1 - 2^n\sigma.$$
We claim that $\pi_{2^n-1 - 2^n\sigma}^{\Ctwo}\left((E\Ctwo)_+ \tensor M\R\right)$ is zero. Indeed, we have 
$$\pi_{2^n-1 - 2^n\sigma}^{\Ctwo}\left((E\Ctwo)_+ \tensor M\R\right) \cong \pi_{2^n-1}(\Sigma^{2^n\sigma}M\R)_{h\Ctwo}.$$
This can be computed by the homotopy orbit spectral sequence 
\[H_p(\Ctwo; \pi_q\Sigma^{2^n\sigma}M\R) \Rightarrow \pi_{p+q}(\Sigma^{2^n\sigma}M\R)_{h\Ctwo}.\]
But $\pi_q\Sigma^{2^n\sigma}M\R=0$ for $q<2^n$ so that $\pi_{2^n-1}(\Sigma^{2^n\sigma}M\R)_{h\Ctwo} = 0$. Thus, we see that $a^{2^{n+1}-1}\vb_n = 0$ in $\pi_\bigstar^{\Ctwo}M\R$.
\end{proof}

For a $\Ctwo$-spectrum $X$ the $RO(\Ctwo)$ graded homotopy fixed point
spectral sequence is defined by combining the homotopy fixed point
spectral sequences 
$$E_2^{p,q}(r)=H^q(\Ctwo, \pi_{p+q}(X\sm S^{-r\sigma})) \;\Rightarrow\;
\pi^{\Ctwo}_{p}((X\sm S^{-r\sigma})^{h\Ctwo}) \cong \pi^{\Ctwo}_{p+r\sigma}(X^{(EC_2)_+)}) $$
into a single spectral sequence with differential
$$d_n: E_n^{p,q}(r)\lra E_n^{p-1, q+n}(r). $$
Note that we use an Adams grading convention here. We will often call $p+r\sigma$ the \emph{degree} of an element. 

The $RO(\Ctwo)$-graded homotopy fixed point spectral sequence (HFPSS) for $BP\R$ has $E_2$-term
\[ \Z_{(2)}[a, u^{\pm 1},\vb_1, \vb_2,\dots]/2a.\]
with
$$|a|=(-\sigma, 1),\, |u|=(2-2\sigma, 0),\text{ and } |\vb_i|=((2^i-1)\rho, 0).$$
This can be seen, for example, by the identification with the Bockstein spectral sequence for $a$ discussed in \cite[Lemma 4.8]{HM}. As $BP\R$ is a retract of $M\R_{(2)}$, it has the structure of a (homotopy) ring spectrum and thus the $RO(C_2)$-graded homotopy fixed point spectral sequence is multiplicative by \cite[Sec 2.3]{HM}. 

By the discussion above, $a$ and the $\vb_i$ are permanent cycles. As $a^{2^{n+1}-1}\vb_n$ is zero, it must be hit by a differential. This is the crucial ingredient for the following central proposition. It is fully formal in the sense that we do not need any other input in addition to the things we already discussed; we argue just with the form of the spectral sequence. We will set $\vb_0 = 2$ for convenience. 

\begin{prop}\label{prop:Differentials}
In the HFPSS for $BP\R$, we have $E_{2^n} = E_{2^{n+1}-1}$ and it is the subalgebra of 
$$E_2/(a^3\vb_1,\dots, a^{2^n-1}\vb_{n-1})$$ 
generated by $a, u^{\pm 2^{n-1}}$, the $\vb_i$ for $i\geq 0$ and by the  $\vb_iu^{2^ij}$ for $i<n-1$ and $j\in\Z$.
\end{prop}
\begin{proof}
We prove it by induction. It is obviously true for $n=1$ by the checkerboard phenomenon; indeed, for all generators of the $E_2$-term in degree $(a+b\rho, q)$ we have $a+q$ even. 

Now assume it to be true for a given $n$. First, we will show that $d_{2^{n+1}-1}(u^{2^{n-1}}) = a^{2^{n+1}-1}\vb_n$. Indeed, as $a^{2^{n+1}-1}\vb_n$ is nonzero in $E_{2^{n+1}-1}$, it must be hit by a $d_{2^{n+1}-1}$. Its source $x$ is in the zero-line in degree $2^{n+1}-2^n\rho$. As the zero-line in $E_2$ is generated by $u$ of degree $4-2\rho$ and by the $\vb_i$ in regular representation degrees, we see that the exponent of $u$ in $x$ must be $2^{n-1}$ so that there is no room for further $\vb_i$. Thus, $d_{2^{n+1}-1}(u^{2^{n-1}}) = a^{2^{n+1}-1}\vb_n$. 

Next, we want to show that $d_q(\vb_iu^{2^ij}) = 0$ for $2^{n+1}-1\leq q < 2^{n+2}-1$ and $i<n$. Write $d_q(\vb_iu^{2^ij}) = a^qx$. The degree of $x$ is
$$(2^i-1)\rho + 2^ij(4-2\rho)-q(1-\rho)-1 = (2^{i+2}j-q-1) + (2^i-2^{i+1}j+q-1)\rho.$$
Thus, $x = u^{2^ij-\frac{q+1}4}\vb$, where $\vb$ is a polynomial in the $\vb_\nu$. The degree of $\vb$ is $(2^i-2+\frac{q+1}2)\rho$. As $\frac{q+1}2 < 2^{n+1}$, we have 
$$|\vb| < |\vb_{n+1}^2| < |\vb_r|$$
for $r\geq n+2$. Thus, no monomial in $\vb$ is divisible by $\vb_{n+1}^2$ or $\vb_r$. Assume that $|\vb| = |\vb_{n+1}|$. Then $\frac{q+1}2 = 2^{n+1}-1+2-2^i = 2^{n+1}-2^i+1$, which is odd; but then $\frac{q+1}4 \notin \Z$, which is a contradiction. Thus, every monomial in $\vb$ is divisible by some $\vb_k$ for some $k\leq n$ as $\vb \neq 1$ for degree reasons. But $a^q\vb_k = 0$ in $E_q$. Thus, also $a^qx = 0$  in $E_q$.

Similarly, write $d_q(u^{2^n}) = a^qx$ for $2^{n+1}-1\leq q < 2^{n+2}-1$ and assume that this is nonzero. The degree of $x$ is 
$$2^n(4-2\rho)-q(1-\rho)-1 = (2^{n+2}-q-1)+(q-2^{n+1})\rho.$$
Thus, we can write $x$ in $E_2$ as $u^{2^n-\frac{q+1}4}\vb$, where $\vb$ is a polynomial in the $\vb_\nu$. The degree of $\vb$ is $\frac{q-1}2 <2^{n+1}-1$. Thus, no monomial in $\vb$ can be divisible by $\vb_r$ for $r\geq n+1$. Thus, every monomial in $\vb$ is divisible by some $\vb_k$ for some $k\leq n$ as $\vb \neq 1$ for degree reasons. But $a^q\vb_k = 0$ in $E_q$.  Thus, $d_q(u^{2^n}) = 0$.

By the Leibniz rule, this implies the proposition.
\end{proof}

Before we proceed to solve the multiplicative extension issues, we need a technical lemma. 
\begin{lemma}\label{lem:filtration}
 Assume that there is an element $a^ku^l\vb\neq 0$ above the zero line in the $E_\infty$-term of the $RO(C_2)$-graded HFPSS for $BP\R$ with $\vb$ a monomial in the $\vb_\nu$ and in the same degree as $\vb_i\vb_mu^{2^mj}$. Let $p$ be the minimal index such that $\vb_p$ divides $\vb$ (which we will show to exist). Then $i> p+m$. 
\end{lemma}
\begin{proof}
 The degree of $\vb_i\vb_mu^{2^mj}$ is
  $$2^mj(4-2\rho) + (2^i-1+2^m-1)\rho = 2^{m+2}j+(2^i+2^m-2^{m+1}j-2)\rho.$$
  Let $a^ku^l\vb\neq 0$ be an element in $E_\infty$ in this degree with $\vb$ a monomial in the $\vb_{\nu}$ of degree $n\rho$ and assume that $k> 0$. (In the following we will use the notation $||\vb_p|| = |\vb_p|/\rho$ so that $||\vb|| = n$.) We get 
  \begin{align*}
  4l+k &= 2^{m+2}j\\
  n -2l-k &= 2^i+2^m-2^{m+1}j-2.
 \end{align*}
This implies $n= 2^i + 2^m-2+ \frac{k}2 $. We see that $n\neq 0$. Let $p$ be the minimal index such that $\vb_p|\vb$. Then $2^p|l$ and we set $c = l/2^p$. Then $k = 2^{m+2}j-2^{p+2}c$. Due to the relation $a^{2^{p+1}-1}\vb_p = 0$, we have $k\leq 2^{p+1}-2$ and thus $m+2\leq p$ (as else $2^{p+1}|k$ and thus $k\geq 2^{p+1}$). In particular, $2^{m+1}$ divides $\frac{k}2$. Now observe that $n \geq ||\vb_p|| = 2^p-1$ so that 
$$2^i+2^m-1 \geq 2^p-\frac{k}2.$$
As $k\leq 2^{p+1}-2$, the right hand side is positive; as it is also divisible by $2^{m+1}$ it is thus it is at least $2^{m+1}$. We see that $i\geq m+1$. Thus $n\equiv 2^m-2 \mod 2^{m+1}$. As $||\vb_q|| \equiv -1 \mod 2^{m+1}$ for $q\geq p >m+1$, we see that the total exponent of $\vb$ (i.e.\ the degree of $\vb$ as a monomial in the $\vb_{\nu}$) must be $\equiv 2^m+2 \mod 2^{m+1}$. In particular, $n\geq ||\vb_p||(2^m+2) = (2^p-1)(2^m+2)$. Thus,
$$\frac{k}2 = n-2^i-2^m+2 \geq 2^{p+m}-2^i+(2^{p+1}-2^{m+1}).$$
If $p+m\geq i$, then the right hand side is at least $2^p$, which would be a contradiction. Thus $i>p+m$. 
\end{proof}

 Now, we are ready to prove the main result of the appendix. Note that \cite[Theorem 4.11]{HK} gives a different relation than our last one; our relation implies their relation, but not vice versa. Note also that our arguments for the multiplicative relations are completely algebraic (using the form of the spectral sequence), while \cite{HK} uses additionally a $C_2$-equivariant Adams spectral sequence.
\begin{thm}\label{thm:BPR}
 The ring $\pi_\bigstar^{\Ctwo}BP\R$ is isomorphic to the $E_\infty$-term of the homotopy fixed point spectral sequence above, i.e.\ to the subalgebra of
 $$\Z_{(2)}[a, \vb_i, u^{\pm 1}]/(2a, \vb_ia^{2^{i+1}-1})$$
 (where $i$ runs over all positive integers) generated by $\vb_m(n) = u^{2^mn}\vb_m$ (with $m,n\in\Z$ and $m\geq 0$) and $a$ with $\vb_0=2$. Consequently, it is the quotient $R$ of the ring 
 $$\Z_{(2)}[a, \vb_m(n)|m\geq 0, n\in\Z]$$
 by the relations 
 \begin{align*}
\vb_0(0) &= 2,\\
 a^{2^{m+1}-1}\vb_m(n) &= 0,\\
 \vb_i(j)\vb_m(n) &= \vb_i\vb_m(2^{i-m}j+n)\text{ for }i\geq m
 \end{align*}
 with $\vb_i = \vb_i(0)$.

 Here, $|a| = 1-\rho$ and $|\vb_m(n)| = 2^{m+2}n + (2^m-1-2^{m+1}n)\rho$. 
\end{thm}
\begin{proof}
 It is enough to show that the expression above computes the homotopy fixed points $\pi_\bigstar^{\Ctwo}BP\R^{(E\Ctwo)_+}$. Indeed, Proposition \ref{prop:Differentials} implies that $\left(a^{-1}BP\R^{(E\Ctwo)_+}\right)^{\Ctwo} \simeq H\F_2$ so that the map $BP\R^{\Phi C_2} \to BP\R^{t\Ctwo}$ is an equivalence and hence also $BP\R \to BP\R^{(E\Ctwo)_+}$ by the Tate square.
 
 Set $\vb_0(0) = 2$. By Proposition \ref{prop:Differentials}, the classes $u^{2^mn}\vb_m$ are permanent cycles in the HFPSS; choose element $\vb_m(n) \in \pi^{\Ctwo}_\bigstar BP\R^{\eqp}$ representing them. Again by Proposition \ref{prop:Differentials}, the $\vb_m(n)$ generate together with $a$ the $E_\infty$-term of the HFPSS. Thus, we get a surjective map $R \to E_\infty$. The third relation defining $R$ allows to define a normal form: Every monomial in the $\vb_i(j)$ equals in $R$ an element of the form $\vb\, \vb_m(k)$, where $\vb$ is a monomial in the $\vb_i$ and $m$ was the smallest index of all $\vb_i(j)$. Thus, two monomials in the $\vb_i(j)$ are equal in $R$ if they are equal in $E_\infty$; hence, the map $R \to E_\infty$ is also injective. 
 
 We now check that the relations are also satisified in $\pi_\bigstar^{\Ctwo}BP\R^{(E\Ctwo)_+}$. This is clear or was already shown for the first two relations. Let now $i$ be the least number such that $m\leq i$ and 
 $$\vb_i(j)\vb_m(n) \neq \vb_i\vb_m(2^{i-m}j+n)$$
 for some $j,m,n$ if such an $i$ exists. The difference must be detected by a class $a^ku^l\vb$, where $\vb$ is a polynomial in the $\vb_\nu$. Let $p$ the minimal index such that every monomial in $\vb$ is divisible by a $\vb_r$ with $r\leq p$. From Lemma \ref{lem:filtration}, we know that $p\leq i-1$ (and in particular $i\geq 1$). Thus, 
 $$\vb_i(j)\vb_m(n)\vb_{i-1} \neq \vb_i\vb_m(2^{i-m}j+n)\vb_{i-1}$$
 as their difference is detected by a nonzero class $a^ku^l\vb\vb_{i-1}$ (indeed: this could only be zero if $k\geq 2^i-1$, but $k< 2^{p+1}-1$). By the minimality of $i$, we have 
 $$\vb_m(2^{i-m}j+n)\vb_{i-1} = \vb_{i-1}(2j)\vb_m(n).$$
 In addition, $\vb_i\vb_{i-1}(2j) = \vb_i(j)\vb_{i-1}$ because there is no element of higher filtration in the same degree as $\vb_{i-1}\vb_i(j)$ by Lemma \ref{lem:filtration}. The last two equations combine to the chain of equalities
\begin{align*}
\vb_i(j)\vb_m(n)\vb_{i-1} &= \vb_i\vb_{i-1}(2j) \vb_m(n) \\
&=\vb_i\vb_m(2^{i-m}j+n)\vb_{i-1}
\end{align*} 
 
 This is a contradiction to the inequality above. Thus, 
 $$\vb_i(j)\vb_m(n) = \vb_i\vb_m(2^{i-m}j+n)$$
 is always true for $i\geq m$. 
\end{proof}

\begin{remark}
 We remark that all the work above for the multiplicative extensions was actually necessary. For example, we get from the homotopy fixed point spectral sequence only that $\vb_5\vb_1(1) - \vb_5(1)\vb_1(-15)$ has filtration at least $1$. But there are indeed classes in this degree of higher filtration, for example $a^8\vb_3^3\vb_4$.
\end{remark}

\bibliographystyle{alpha}
\bibliography{Chromatic}
\end{document}